\documentclass[11pt]{amsart} 

\usepackage{epsf,amssymb,amsmath,overpic,amscd,psfrag,color,graphicx}
\usepackage[all]{xy}
\usepackage{hyperref}
\usepackage{bm}

\usepackage[margin=1.5in]{geometry}
\usepackage{cleveref}
\crefname{thm}{Theorem}{Theorems}
\Crefname{thm}{Theorem}{Theorems}

\newtheorem{thm}{Theorem}[section]
\newtheorem{prop}[thm]{Proposition}
\newtheorem{lemma}[thm]{Lemma}
\newtheorem{cor}[thm]{Corollary}

\newtheorem{claim}[thm]{Claim}
\newtheorem{fact}[thm]{Fact}

\numberwithin{equation}{subsection}
\numberwithin{thm}{subsection}

\theoremstyle{definition}
\newtheorem{defn}[thm]{Definition}

\newtheorem{notation}[thm]{Notation}

\theoremstyle{remark}

\newtheorem{rmk}[thm]{Remark}

\newcommand{\C}{\mathbb{C}}

\newcommand{\R}{\mathbb{R}}
\newcommand{\Z}{\mathbb{Z}}

\newcommand{\bdry}{\partial}
\newcommand{\s}{\vskip.1in}
\newcommand{\n}{\noindent}

\newcommand{\be}{\begin{enumerate}}
\newcommand{\ee}{\end{enumerate}}

\newcommand{\ol}{\overline}
\newcommand{\ind}{\operatorname{ind}}
\newcommand{\Span}{\operatorname{Span}}
\newcommand{\ev}{\operatorname{ev}}

\newcommand{\mev}{\operatorname{mev}}

\newcommand{\Coker}{\operatorname{Coker}}

\newcommand{\modsp}[2]{\mathcal M(#1 \, | \, #2)}
\newcommand{\Hom}{\operatorname{Hom}}
\newcommand{\Ker}{\operatorname{Ker}}

\makeatletter
\def\@tocline#1#2#3#4#5#6#7{\relax
  \ifnum #1>\c@tocdepth % then omit
  \else
  \ifnum #1=1
    \par \addpenalty\@secpenalty\addvspace{#2}%
    \begingroup \hyphenpenalty\@M
    \@ifempty{#4}{%
      \@tempdima\csname r@tocindent\number#1\endcsname\relax
    }{%
      \@tempdima#4\relax
    }%
    \parindent\z@ \leftskip#3\relax \advance\leftskip\@tempdima\relax
    \rightskip\@pnumwidth plus4em \parfillskip-\@pnumwidth
    #5\leavevmode\hskip-\@tempdima
      \ifcase #1
       \or\or \hskip 1em \or \hskip 2em \else \hskip 3em \fi%
      \textbf{#6}\nobreak\relax
    \hfill\hbox to\@pnumwidth{\@tocpagenum{\textbf{#7}}}\par
    \nobreak
    \endgroup
  \else
    \par \addpenalty\@secpenalty\addvspace{#2}%
    \begingroup \hyphenpenalty\@M
    \@ifempty{#4}{%
      \@tempdima\csname r@tocindent\number#1\endcsname\relax
    }{%
      \@tempdima#4\relax
    }%
    \parindent\z@ \leftskip#3\relax \advance\leftskip\@tempdima\relax
    \rightskip\@pnumwidth plus4em \parfillskip-\@pnumwidth
    #5\leavevmode\hskip-\@tempdima
      \ifcase #1
       \or\or \hskip 2.1em \or \hskip 3.1em \else \hskip 4.1em \fi%
      #6\nobreak\relax
    \dotfill\hbox to\@pnumwidth{\@tocpagenum{#7}}\par
    \nobreak
    \endgroup
  \fi
  \fi}
\makeatother

\allowdisplaybreaks

\begin{document}

%===================%
% ADMINISTRATIVE INFO %
%===================%

\title  [Cobordism maps in embedded contact homology]
	{Cobordism maps in embedded contact homology}

\author 	{Jacob Rooney}
\address 	{Simons Center for Geometry and Physics, Stony Brook University}
\email 	{\href{mailto:jrooney@scgp.stonybrook.edu}{jrooney@scgp.stonybrook.edu}}
\urladdr 	{\url{https://www.jacobrooney.com}}
\date 	{This version: \today}

%=======================================
% TITLE, ABSTRACT, AND TABLE OF CONTENTS
%=======================================

\begin{abstract}
Given an exact symplectic cobordism $(X, \lambda)$ between contact $3$-manifolds $(Y_+, \lambda_+)$ and $(Y_-, \lambda_-)$ with no elliptic Reeb orbits up to a certain action, we define a chain map from the embedded contact homology (ECH) chain complex of $(Y_+, \lambda_+)$ to that of $(Y_-, \lambda_-)$, both taken with coefficients in $\Z / 2 \Z$. The map is defined by counting punctured holomorphic curves with ECH index $0$ in the completion of the cobordism and new objects that we call ECH buildings, answering a question of Hutchings.
\end{abstract}

\maketitle
\setcounter{tocdepth}{1}
\tableofcontents

%==============
% SECTION 1
% INTRODUCTION
%==============

\section{Introduction}
\label{sec:intro}

In this paper, we answer a question of Hutchings on the foundations of ECH: given contact $3$-manifolds $(Y_\pm, \lambda_\pm)$ and an exact symplectic cobordism $(X, \lambda)$ from $(Y_+, \lambda_+)$ to $(Y_-, \lambda_-)$, how can we define a chain map from the ECH chain complex of $Y_+$ to that of $Y_-$ by counting $J$-holomorphic curves? We answer this question when $(Y_\pm, \lambda_\pm)$ have no elliptic orbits up to a certain action $L$. Namely, given the above setup and assuming that $(Y_\pm, \lambda_\pm)$ have no elliptic Reeb orbits up to an action $L$, we define a chain map
\begin{equation*}
	\Phi_{X, \lambda, J, \mathbf c} \colon ECC^L(Y_+, \lambda_+, J_+) \to ECC^L(Y_-, \lambda_-, J_-)
\end{equation*}
by counting $J$-holomorphic curves in the completion $\widehat X$ and new objects that we call \textbf{index $0$ ECH buildings}. Here, $J_\pm$ is a generic almost complex structure on the symplectization $\R \times Y_\pm$, $J$ is a generic almost complex structure on $\widehat X$ that is compatible with $J_+$ at the positive end and with $J_-$ at the negative end, and $\mathbf c$ is a choice of auxiliary data that is explained in \Cref{defn:asymptotic-restriction}. We show in \Cref{thm:main-thm} that $\Phi_{X, \lambda, J, \mathbf c}$ is a chain map and is independent of the choice of $\mathbf c$. The definition of $\Phi_{X, \lambda, J, \mathbf c}$ relies on some new developments for holomorphic curves in the $L$-supersimple setting of Bao-Honda \cite{BH1,BH2} and Colin-Ghiggini-Honda \cite{CGH1,CGH2,CGH3}, and we restrict our attention to that setting throughout the paper.

ECH is isomorphic to both Heegaard Floer homology and Seiberg-Witten Floer (co)homology (see \cite{KLT1,KLT2,KLT3,KLT4,KLT5,CGH1,CGH2,CGH3}), and the latter isomorphism was used by Hutchings-Taubes in \cite{HT3} to define maps induced by exact symplectic cobordisms between contact $3$-manifolds. However, a definition of such maps that involves counting $J$-holomorphic curves has proved elusive. Chris Gerig has given a construction in a specific case \cite{G}, and Hutchings has given an example where one must take into account multi-level SFT buildings \cite[Section 5]{H3}.

In \Cref{sec:background,sec:ev-map}, we give appropriate background information for ECH and the evaluation map defined by Bao-Honda. In \Cref{sec:ech-ind,sec:degen,sec:obs-bund-glue,sec:models}, we discuss the details of these new developments. In \Cref{sec:cob-map-defn}, we prove the main result of this paper, namely, that $\Phi_{X, \lambda, J, \mathbf c}$ is a chain map. The remainder of this section is an outline of the paper, culminating in the definition of $\Phi_{X, \lambda, J, \mathbf c}$; see \Cref{thm:main-thm} and \Cref{defn:cob-map}, which depend on some auxiliary definitions in this section.

This paper is a heavily revised version of the author's doctoral thesis \cite{R}, from which portions of this work have been excerpted.

%=======================%
% SECTION 1.1                         %
% L-SUPERSIMPLE SETTING  %
%=======================%

\subsection{The $L$-supersimple setting and filtered ECH}
\label{subsec:intro-l-supersimple}

We begin with a discussion of the $L$-supersimple setting. Recall that the \textbf{action} of a Reeb orbit $\alpha$ on the contact manifold $(Y, \lambda)$ is the integral $\mathcal A(\alpha) = \int_\alpha \lambda$, while the \textbf{total action} of an orbit set $\bm \alpha$ is the sum $\mathcal A(\bm \alpha) = \sum_{\alpha \in \bm \alpha} \int_\alpha \lambda$.

\begin{defn}
\label{defn:L-supersimple}
	A contact form $\lambda$ on a smooth $3$-manifold $Y$ is \textit{$L$-supersimple} if every Reeb orbit with action less than $L$ is non-degenerate, hyperbolic, and satisfies the conclusions of \Cref{thm:eliminate-elliptics}.
\end{defn}

Our chain map is defined on the level of \textit{filtered ECH}, defined as follows. Let $(Y, \lambda)$ be a non-degenerate contact $3$-manifold, and let $J$ by a generic, compatible almost complex structure on $\R \times Y$. Let $L > 0$ and consider the subgroup $ECC^L(Y, \lambda, J) \subset ECC(Y, \lambda, J)$ generated by orbit sets $\bm \alpha$ with total action less than $L$. Every non-degenerate contact form can be made into an $L$-supersimple form by a small perturbation. That is, for any $L > 0$ and $\varepsilon > 0$, there is a positive smooth function $f$ on $Y$ that is $C^1$-close to $1$ such that $f \lambda$ is $L$-supersimple. Furthermore, if $f_i \lambda$ is $L_i$-supersimple for $i = 1, 2$ and $L_1 < L_2$, we can ensure that the set of Reeb orbits of $f_2 \lambda$ with action less than $L_1$ coincides with the corresponding set of Reeb orbits for $f_1 \lambda$, i.e., that there is a natural inclusion map
\begin{equation*}
	ECC^{L_1}(Y, f_1 \lambda, J)
		\hookrightarrow
			ECC^{L_2}(Y, f_2 \lambda, J).
\end{equation*}
See \cite[Theorem 2.0.2]{BH1} and \cite[Theorem 2.5.2]{CGH1} for details.

We can reconstruct $ECH(Y, \lambda, J)$ from these filtered groups in the following way, as described in \cite[Theorem 3.2.1]{CGH0}. Let $\{ f_i \}_{i = 1}^\infty$ be a sequence of positive smooth functions on $Y$ with $1 \ge f_1 \ge f_2 \ge \cdots$ and such that $f_i \lambda$ is $L_i$-supersimple for some sequence $\{ L_i \}_{i = 1}^\infty$ of positive real numbers with $\displaystyle\lim_{i \to \infty} L_i = \infty$. Then there is a canonical isomorphism
\begin{equation*}
	ECH(Y, \lambda, J)
		\simeq
			\lim_{i \to \infty} ECH^{L_i}(Y, f_i \lambda, J).
\end{equation*}
Thus, it suffices to define the chain map $\Phi_{X, \lambda, J, \mathbf c}$ on each level of the filtration $ECC^{L_i}(Y, f_i \lambda, J)$, where there are no elliptic Reeb orbits. We do not lose any generality in assuming that the contact forms on $Y_\pm$ are $L$-supersimple aside from the need to assume invariance results of Hutchings-Taubes \cite{HT3}.

%=================%
% SECTION 1.2	     %
% INDEX INEQUALITY %
%=================%

\subsection{The ECH index inequality}
\label{subsec:intro-index-inequality}

The first of our developments is an improvement to the ECH index inequality in the $L$-supersimple setting. On one-dimensional moduli spaces, the inequality is in fact an equality and gives information about the topology of punctured $J$-holomorphic curves that violate the ECH partition conditions. One can also show that the improved equality is an equality for generic curves with higher Fredholm index using the evaluation map from \Cref{sec:ev-map}. The inequality is implicit in the work of Hutchings \cite{H2}. Gardiner-Hind-McDuff give a similar improvement in \cite{CGHD}, and Gardiner-Hutchings-Zhang recently showed that the improved inequality is an equality for generic curves \cite{CGHZ}. The advantages of the $L$-supersimple setting are that (1) the extra term in the improved inequality is given by a simple formula that involves only the multiplicities of the ends of the curve, and (2) the analysis required to prove generic equality is greatly simplified.

The starting point for our improved inequality is Hutchings' \textbf{ECH index inequality} from \cite{H2}: If $u$ is a somewhere injective $J$-holomorphic curve in a symplectization $\R \times Y$, then
\begin{equation}
\label{eqn:hut-ind-ineq}
	I(u) \ge \ind(u) + 2\delta(u),
\end{equation}
where $\delta(u)$ is a non-negative count of singularities of $u$.

\begin{defn}
\label{defn:pos-neg-orbit-set}
	Let $u \colon \dot \Sigma \to \R \times Y$ be a punctured $J$-holomorphic curve asymptotic to an orbit set $\bm \alpha$ at the positive ends and to an orbit set $\bm \beta$ at the negative ends. We say that $\bm \alpha$ is the \textbf{positive orbit set} of $u$, that $\bm \beta$ is the \textbf{negative orbit set} of $u$, and that $u$ goes from $\bm \alpha$ to $\bm \beta$.
\end{defn}

\begin{defn}
\label{defn:emb-orb-set}
	Let $\Gamma^+(u)$ denote the set of embedded Reeb orbits in the positive orbit set of $u$ (i.e., forgetting their multiplicities), and let $\Gamma^-(u)$ denote the set of embedded Reeb orbits in the negative orbit set of $u$.
\end{defn}

\begin{defn}
\label{defn:ech-orb-deficit}
	The \textbf{ECH deficit of $u$ at an orbit} $\gamma \in \Gamma^+(u)$ is defined as follows. If $\gamma$ is negative hyperbolic, suppose $u$ has ends at (covers of) $\gamma$ of multiplicities $q_1, \ldots, q_n$, ordered so that the first $k$ ends hove odd multiplicity and the last $n - k$ ends have even multiplicity. Then
	\begin{equation*}
		\Delta(u, \gamma) = \sum_{i = 1}^k \left( \frac{q_i - 1}{2} + i - 1 \right) + \sum_{i = k + 1}^n \left( \frac{q_i}{2} - 1 \right)
	\end{equation*}
	If $\gamma$ is positive hyperbolic and $u$ has ends at (covers of) $\gamma$ of multiplicities $q_1, \ldots, q_n$, then
	\begin{equation*}
		\Delta(u, \gamma) = \sum_{i = 1}^n (q_i - 1).
	\end{equation*}
    The ECH deficit $\Delta(u, \gamma)$ for $\gamma \in \Gamma^-(u)$ is defined similarly.
\end{defn}

\begin{defn}
\label{defn:ech-deficit}
	The \textbf{ECH deficit} of $u$ is
	\begin{equation*}
		\Delta(u) = \sum_{\gamma \in \Gamma^+(u)} \Delta(u, \gamma) + \sum_{\gamma \in \Gamma^-(u)} \Delta(u, \gamma).
	\end{equation*}
\end{defn}

\begin{thm}
\label{thm:gen-ind-ineq}
	If $J$ is generic and $u$ is a somewhere injective $J$-holomorphic curve in a symplectization, then
	\begin{equation}
	\label{eqn:gen-ech-ineq}
		I(u) \ge \ind(u) + 2 \delta(u) + \Delta(u).
	\end{equation}
	Equality holds if $\mathcal A(\bm \alpha) < L$ and $\ind(u) = 1$.
\end{thm}

%================%
% SECTION 1.3          %
% DEGENERATIONS %
%================%

\subsection{Degenerations of one-dimensional families in cobordisms}
\label{subsec:intro-degenerations}

The next development is an analysis of possible degenerations of one-dimensional families of punctured holomorphic curves in exact symplectic cobordisms, which we discuss in \Cref{sec:degen}.

Let $(Y_\pm, \lambda_\pm)$ be $L$-supersimple contact $3$-manifolds and let $(X, \lambda)$ be an exact symplectic cobordism from $(Y_+, \lambda_+)$ to $(Y_-, \lambda_-)$. Let $J$ be a generic, $L$-simple, admissible almost complex structure on the completion $(\widehat X, \widehat \lambda)$ that restricts to $L$-simple, admissible almost complex structures $J_+$ and $J_-$ on the ends $[0, \infty) \times Y_+$ and $(-\infty, 0] \times Y_-$, respectively, of $\widehat X$.

\begin{notation}
\label{notation:moduli-spaces}
	Let $\bm \alpha$ and $\bm \beta$ be orbits sets in a contact manifold $(Y, \lambda)$. We denote the moduli space of $J$-holomorphic curves $u$ from $\bm \alpha$ to $\bm \beta$ in $\R \times Y_\pm$ with $\ind(u) = p$ and $I(u) = q$ by $\mathcal M^{p, q}_{\R \times Y_\pm}(\bm \alpha, \bm \beta)$.
	
	Let $(X, \lambda)$ be an exact symplectic cobordism from $(Y_+, \lambda_+)$ to $(Y_-, \lambda_-)$, let $\bm \alpha$ be an orbit set in $(Y_+, \lambda_+)$, and let $\bm \beta$ be an orbit set in $(Y_-, \lambda_-)$. We denote the moduli space of $J$-holomorphic curves $u$ from $\bm \alpha$ to $\bm \beta$ in the completion $\widehat X$ with $\ind(u) = p$ and $I(u) = q$ by $\mathcal M^{p, q}_X(\bm \alpha, \bm \beta)$.
\end{notation}

Let $\bm \alpha$ and $\bm \beta$ be generators of $ECC^L(Y_+, \lambda_+, J_+)$ and $ECC^L(Y_-, \lambda_-, J_-)$, respectively. Consider the moduli space $\mathcal M_X^{1, 1}X^1(\bm \alpha, \bm \beta)$ and let $\overline{\mathcal M}_X^{1, 1}(\bm \alpha, \bm \beta)$ denote its SFT compactification as described in \cite{BEHWZ}. We denote an SFT building in $\bdry \overline{\mathcal M}_X^{1, 1}(\bm \alpha, \bm \beta)$ by $[u_{-a}] \cup \cdots [u_{-1}] \cup u_0 \cup [u_1] \cup \cdots \cup [u_b]$, where $a$ and $b$ are positive integers, the levels go from bottom to top as we read from left to right, the levels with negative indices are in $(\R \times Y_-) / \R$, the level $u_0$ is in $\widehat X$, and the levels with positive indices are in $(\R \times Y_-) / \R$.

\begin{thm}
\label{thm:degen-class-thm}
	The points in $\bdry \overline{\mathcal M}_X^{1, 1}(\bm \alpha, \bm \beta)$ are two-level buildings of the form $[u_{-1}] \cup u_0$ or $u_0 \cup [u_1]$, where $\ind(u_0) = 0$ and $\ind(u_{\pm 1}) = 1$. Let $\bm \gamma$ denote the negative orbit set of $u_+$. When $\bm \gamma$ is a generator of the ECH chain complex $ECC^L(Y_\pm, \lambda_\pm, J_\pm)$, we have $I(u_0) = 0$ and $I(u_{\pm 1}) = 1$, and both levels are somewhere injective. When $\bm \gamma$ is not a generator of $ECC^L(Y_\pm, \lambda_\pm, J_\pm)$, the buildings occur in pairs unless they are of the form $u_0 \cup [u_1]$ and the following conditions hold:
	\begin{enumerate}
		\item
		$u_1$ is somewhere injective;
		
		\item
		$u_0$ is multiply covered;
		
		\item
		$I(u_1) > 1$ and $I(u_0) < 0$;
		
		\item
		each Reeb orbit in $\bm \gamma$ has multiplicity $1$ except for finitely many negative hyperbolic orbits $\gamma_1, \gamma_2, \ldots, \gamma_k$ with multiplicities $n_1, n_2, \ldots, n_k$, respectively;
		
		\item
		$u_1$ has $n_i$ negative ends at $\gamma_i$, each with multiplicity $1$;
		
		\item
		for each $i = 1, \ldots, k$, $u_0$ contains an unbranched, disconnected, $n_i$-fold multiple cover of an embedded holomorphic plane with its positive end at $\gamma_i$, and each multiply covered component of $u_0$ is of this form.
	\end{enumerate}
\end{thm}

%==============================
% SECTION 1.4
% PROTOTYPICAL GLUING PROBLEM
%==============================

\subsection{The prototypical gluing problem}
\label{subsec:intro-prototypical-gluing}

The last development is an obstruction bundle gluing calculation for certain branched covers of trivial cylinders with high Fredholm index, which we discuss in \Cref{sec:obs-bund-glue}. Here, a \textbf{trivial cylinder} is a cylinder $\R \times \beta_0 \subset \R \times Y_+$, where $\beta_0$ is an embedded Reeb orbit in $Y$. We use the notation of Hutchings-Taubes from \cite{HT1} for moduli spaces of such branched covers.

\begin{defn}
\label{defn:branch-cov-mod-sp}
	Let $\beta_0$ be a Reeb orbit in $(Y_+, \lambda_+)$. Let
	\begin{equation*}
		\modsp{a_1, a_2, \ldots, a_k}{a_{-1}, a_{-2}, \ldots, a_{-\ell}}
	\end{equation*}
	denote the moduli space of genus $0$ branched covers $\dot \Sigma \to \R \times \beta_0$ with ends labeled and asymptotically marked and such that the $i^\text{th}$ end is asymptotic to an $a_i$-fold cover of $\beta_0$.
\end{defn}

\begin{defn}
\label{defn:top-piece-mod-sp}	
	For each $n \ge 3$, let $\mathcal M_n = \modsp{1, 1, \ldots, 1}{1, 1, \ldots, 1, 3}$, where there are $n$ positive ends of multiplicity $1$, $n - 3$ negative ends of multiplicity $1$, and one negative end of multiplicity $3$.
\end{defn}

The prototypical gluing problem considered in this paper is the following. Let $u_1 \colon \dot \Sigma \to \R \times Y_+$ be an embedded $J$-holomorphic curve with $\ind(u) = 1$ such that
\begin{enumerate}
	\item
	the positive ends of $u_1$ are asymptotic to an ECH generator $\bm \alpha$ with total action less than $L$;
	
	\item
	the negative ends of $u_1$ are asymptotic to an orbit set $\bm \beta$ in which each Reeb orbit has multiplicity $1$ except for a single negative hyperbolic orbit $\beta_0$;
	
	\item
	the curve $u_1$ has $n$ negative ends at $\beta_0$, each with multiplicity $1$;
	
	\item
	$I(u_1) = 1 + \binom{n}{2}$.
\end{enumerate}
We wish to glue branched covers in $\mathcal M_n$ to the curve $u_1$.

The main source of trouble in the above gluing problem is that the moduli spaces $\mathcal M_n$ are not transversely cut out. However, by standard techniques, there should be an \textbf{obstruction bundle}
\begin{equation*}
	\mathcal O
		\to
			[R, \infty) \times (\mathcal M_n / \R),
\end{equation*}
for $R \gg 0$ sufficiently large, with fiber
\begin{equation*}
	\mathcal O_{(T, u)} = \Hom \left( \Coker D_u^N, \R \right),
\end{equation*}
where $D_u^N$ is the normal part of the linearized $\overline{\bdry}$-operator for $u$.

In analogy with \cite[Definition 5.9]{HT2}, there should also be an \textbf{obstruction section} $\mathfrak s$ for $\mathcal O$ whose zero set is the set of branched covers that glue to $u_+$. Such glued curves lie in the moduli space $\mathcal M_{\R \times Y_+}^{2n - 3, 1}(\bm \alpha, \bm \beta)$. In \Cref{subsec:asymptotic-operator-deformation}, we describe a perturbation of the asymptotic operator for $\beta_0$ that allows us to replace elements of $\Coker D_u^N$ with anti-meromorphic $1$-forms on $\dot \Sigma$, which we use to write down the zero set of $\mathfrak s$ explicitly.

\begin{defn}
\label{defn:asymptotic-restriction}
Let $u \colon \dot \Sigma \to \R \times Y$ be a punctured $J$-holomorphic curve in $\mathcal M^{2n - 3, 1}_{\R \times Y_+}(\bm \alpha, \bm \beta)$ with $n - 3$ negative ends of multiplicity $1$ at $\beta_0$ and one negative end of multiplicity $3$ at $\beta_0$. Label the negative ends of $u$ at (covers of) $\beta_0$ with by elements of $I_- = \{-1, \ldots, -n\}$, where the multiplicity $3$ end is labeled $1$. The curve $u$ satisfies the \textbf{asymptotic restrictions} $\mathbf c \in \C^{n - 2}$ if $\ev_{I_-}(u) = \mathbf c$, where the evaluation map $\ev_{I_-}$ maps $u$ to the leading complex coefficient in the asymptotic expansion of $u$ at the negative ends labeled by $I_-$. See \Cref{defn:multi-end-ev-map} for the full definition of the evaluation map.
\end{defn}

\begin{defn}
\label{defn:permissible-restriction}
	We say that $\mathbf c \in (\C^*)^{n - 2}$ is an \textbf{admissible asymptotic restriction} if it is not in the big diagonal of $(\C^*)^n$.
\end{defn}

\begin{thm}
\label{thm:gluing-top-piece}
	In the prototypical gluing problem, $\mathfrak s^{-1}(0)$ is non-empty. If $\mathbf c \in \C^{n - 2}$ is a generic choice of admissible asymptotic restriction and $T \ge R$, the mod $2$ count of curves in $\mathfrak s^{-1}(0)$ with gluing parameter $T$ that satisfy the asymptotic restriction is $1$.
\end{thm}

%=====================
% SECTION 1.5
% CHAIN MAP DEFINITION
%=====================

\subsection{Definition of the chain map}
\label{subsec:intro-chain-map-defn}

As described above, there are two contributions to the curve count in the definition of $\Phi_{X, \lambda, J, \mathbf c}$. Suppose that we have ECH generators $\bm \alpha \in ECC(Y_+, \lambda_+, J_+)$ and $\bm \beta \in ECC(Y_-, \lambda_-, J_-)$, that we write
\begin{equation*}
	\Phi_{X, \lambda, J, \mathbf c}(\bm \alpha)
		=
			\sum_{\mathcal A(\bm \beta) < \mathcal A(\bm \alpha)}
				\langle \Phi_{X, \lambda, J, \mathbf c}(\bm \alpha), \bm \beta \rangle \cdot \bm \beta,
\end{equation*}
and that we want to define the coefficient $\langle \Phi_{X, \lambda, J, \mathbf c}(\bm \alpha), \bm \beta \rangle$. The first contribution is the mod $2$ count $\#_2 \mathcal M_X^{0, 0}(\bm \alpha, \bm \beta)$. The second contribution is the mod $2$ count of new objects that we call \textbf{ECH buildings} satisfying certain admissible asymptotic restrictions.

\begin{defn}
\label{defn:ech-building}
	Assume the setup described above. An \textbf{index $0$ ECH building} from $\bm \alpha$ to $\bm \beta$ satisfying the admissible asymptotic restriction $\mathbf c$ is a pair $(u_0, [u])$ satisfying the following conditions:
	\begin{enumerate}
		\item
		$[u]$ is in $(\R \times Y_+) / \R$ and $u_0$ is in $\widehat X$;
		
		\item
		$u$ has positive orbit set $\bm \alpha$ and $u_0$ has negative orbit set $\bm \beta$;
		
		\item
		the negative orbit set $\bm \gamma$ of $u$ coincides with the positive orbit set of $u_0$;
		
		\item
		the partition of the negative ends of $u$ coincides with the partition of the positive ends of $u_0$, except possibly for some negative hyperbolic Reeb orbits $\gamma_1, \ldots, \gamma_\ell$ in $\bm \gamma$ of multiplicities $m_1, \ldots, m_\ell$ where the partition for the negative ends of $u$ at each $\gamma_i$ is $(3, 1, \ldots, 1)$ and the partition for the positive ends of $u_0$ at each $\gamma_i$ is $(1, 1, \ldots, 1)$;
		
		\item
		$\ind(u_0) = 0$ and $I(u_0) = - \sum_{j = 1}^\ell \binom{m_j}{2}$;
		
		\item
		$\ind(u) = \sum_{j = 1}^\ell (2m_j - 4)$ and $I(u) = - I(u_0)$; and
		
		\item
		$[u]$ has a (necessarily unique) representative $u$ that satisfies the asymptotic restriction $\mathbf c$, where we use all of the negative ends at the orbits $\gamma_1, \ldots, \gamma_\ell$ for the evaluation map.
	\end{enumerate}
	We denote the set of index $0$ ECH buildings from $\bm \alpha$ to $\bm \beta$ satisfying the admissible asymptotic restriction $\mathbf c$ by $\mathcal B^0(\bm \alpha, \bm \beta; \mathbf c)$.
\end{defn}

\begin{defn}
\label{defn:cob-map}
	Let $(Y_\pm, \lambda_\pm)$ be $L$-supersimple contact $3$-manifolds and let $(X, \lambda)$ be an exact symplectic cobordism from $(Y_+, \lambda_+)$ to $(Y_-, \lambda_-)$. Let $J$ be a generic, $L$-simple, admissible almost complex structure on the completion $(\widehat X, \widehat \lambda)$ that restricts to $L$-simple, admissible almost complex structures $J_+$ and $J_-$ on the ends $[0, \infty) \times Y_+$ and $(-\infty, 0] \times Y_-$, respectively, of $\widehat X$. Let $\mathbf c$ be a generic choice of admissible asymptotic restriction. The map
	\begin{equation*}
		\Phi_{X, \lambda, J, \mathbf c}
			\colon
				ECC^L(Y_+, \lambda_+, J_+)	 \to 	ECC^L(Y_-, \lambda_-, J_-)
	\end{equation*}
	induced by $(X, \lambda)$ is defined by
	\begin{equation*}
		\Phi_{X, \lambda, J, \mathbf c}(\bm \alpha)
			=
				\sum_{\mathcal A(\bm \beta) < \mathcal A(\bm \alpha)}
					\left[
						\#_2 \mathcal M_X^{0, 0}(\bm \alpha, \bm \beta)
						+
						\#_2 \mathcal B^0(\bm \alpha, \bm \beta; \mathbf c)
					\right] \cdot \bm \beta.
	\end{equation*}
\end{defn}

The following theorem is the main result of this paper. Its proof is given in \Cref{sec:cob-map-defn}. The set of \textbf{generic} asymptotic restrictions is described in \Cref{defn:generic-restriction}.

\begin{thm}
\label{thm:main-thm}
	The map $\Phi_{X, \lambda, J, \mathbf c}$ in \Cref{defn:cob-map} is a chain map and is independent of the choice of generic, admissible asymptotic restriction $\mathbf c$.
\end{thm}

\n
{\em Acknowledgements.}  First and foremost, the author thanks Ko Honda for his generous support and endless patience. The author also thanks Michael Hutchings, Katrin Wehrheim, and Erkao Bao for helpful conversations during the development of the ideas in this paper. The work for this article was completed while the author was at UCLA.

%==============%
% SECTION 2		%
% BACKGROUND	%
%==============%

\section{Background}
\label{sec:background}

In this section, we establish some notation, briefly review the definition of embedded contact homology, and recall some basic facts about the $L$-supersimple setting of Bao-Honda.

%=================%
% SECTION 2.1	      %
% BASIC DEFINITIONS %
%=================%

\subsection{Basic definitions}
\label{subsec:basic-defns}

Let $Y$ be a smooth $3$-manifold, let $\lambda$ be a non-degenerate contact form on $Y$, let $\xi = \Ker(\lambda)$ be the associated contact structure, and let $R_\lambda$ be the Reeb vector field of $\lambda$, defined as the unique vector field on $Y$ satisfying $\lambda(R_\lambda) = 1$ and $d\lambda(R_\lambda, \cdot \, ) = 0$.

\begin{defn}
\label{defn:admissible-ac-str}
	An almost complex structure $J$ on $\R \times Y$ is \textbf{admissible} if it satisfies the following properties:
	\begin{enumerate}
		\item
		$J$ is invariant under $\R$-translation;
		
		\item
		$J(\bdry_s) = R_\lambda$, where $s$ is the $\R$-coordinate of $\R \times Y$;
		
		\item
		$J$ restricts to an orientation-preserving isomorphism of $\xi$.
	\end{enumerate}
\end{defn}

Let $\alpha$ be a Reeb orbit in $(Y, \lambda)$ and let $\tau$ be a trivialization of $\xi$ over $\alpha$. We denote the \textbf{Conley-Zehnder} index of $\alpha$ in the trivialization $\tau$ by $\mu_\tau(\alpha)$. We recall here some simple expressions for the Conley-Zehnder index in dimension $3$. If $\alpha$ is elliptic, then there is some irrational number $\theta \in (0, 1)$ such that $\mu_\tau(\alpha^k) = 2 \lfloor k \theta \rfloor + 1$. If $\alpha$ is hyperbolic, then $\mu_\tau(\alpha^k) = kn$ for some integer $n$. In the latter case, we say that $\alpha$ is positive hyperbolic if $n$ is even and negative hyperbolic if $n$ is odd.

\begin{defn}
\label{defn:orbit-sets}
	An \textbf{orbit set} is a tuple of ordered pairs
	\begin{equation*}
		\bm \alpha = \big( (\alpha_1, m_1), (\alpha_2, m_2), \ldots, (\alpha_k, m_k) \big)
	\end{equation*}
	such that each $\alpha_i$ is an embedded Reeb orbit in $Y$ and each $m_i$ is a positive integer.
\end{defn}

\begin{defn}
\label{defn:cz-index}
	If $\bm \alpha = ((\alpha_1, m_1), \ldots, (\alpha_k, m_k))$ is an orbit set, we define
	\begin{equation*}
		\mu_\tau(\bm \alpha) = \sum_{i = 1}^k \mu_\tau(\alpha_k^{m_k})
		\quad \text{and} \quad
		\mu_\tau^I(\bm \alpha) = \sum_{i = 1}^k \sum_{j = 1}^{m_k} \mu_\tau(\alpha^j).
	\end{equation*}
	If $\bm \beta$ is another orbit set, we define
	\begin{equation*}
		\mu_\tau(\bm \alpha, \bm \beta) = \mu_\tau(\bm \alpha) - \mu_\tau(\bm \beta)
		\quad \text{and} \quad
		\mu_\tau^I(\bm \alpha, \bm \beta) = \mu_\tau^I(\bm \alpha) - \mu_\tau^I(\bm \beta).
	\end{equation*}
\end{defn}

%=====================%
% SECTION 2.2		      %
% HOLOMORPHIC CURVES %
%=====================%

\subsection{Punctured holomorphic curves}
\label{subsec:punctured-holom-curves}

Let $(\Sigma, j)$ be a closed Riemann surface with complex structure $j$. Let $P \subset \Sigma$ be a finite set of points, called \textbf{punctures}, which are partitioned into subsets $P^+$ and $P^-$ of \textbf{positive} and \textbf{negative punctures}, respectively. Define $\dot \Sigma = \Sigma \setminus P$; we refer to $\dot \Sigma$ as a \textbf{punctured Riemann surface}. If $J$ is an admissible almost complex structure on $\R \times Y$, a \textbf{punctured holomorphic curve} is a smooth map
\begin{equation*}
	u \colon \dot \Sigma \to \R \times Y
\end{equation*}
such that
\begin{equation*}
	du + J \circ du \circ j = 0.
\end{equation*}

A $J$-holomorphic curve $u \colon \dot \Sigma \to \R \times Y$ is said to be \textbf{multiply covered} if it factors through a (possibly branched) cover $\phi \colon \dot \Sigma' \to \dot \Sigma$ for some punctured Riemann surface $\dot \Sigma'$. A curve is said to be \textbf{simply covered} if it is not multiply covered. We also refer to such curves as \textbf{simple}.

%===============%
% SECTION 2.3	 %
% MODULI SPACES %
%===============%

\subsection{Moduli spaces}
\label{subsec:moduli-spaces}

We distinguish between two types of moduli spaces of $J$-holomorphic curves, marked and unmarked, and make use of both types. Marked moduli spaces are used in \Cref{sec:obs-bund-glue} for obstruction bundle gluing problems, and ECH is defined using unmarked moduli spaces.

Let $u \colon \dot \Sigma \to \R \times Y$ be $J$-holomorphic, and assume that $u$ is asymptotic to Reeb orbits $\alpha_1, \alpha_2, \ldots, \alpha_n$ at the positive punctures and to $\beta_1, \beta_2, \ldots, \beta_m$ at the negative punctures. For each such Reeb orbit, let $(\alpha_i)_e$ denote the underlying embedded Reeb orbit for $\alpha_i$, choose a point $\zeta_i$ on each $(\alpha_i)_e$, and for each $z_i \in P^+$, choose an element $r_i \in (T_{z_i} \Sigma \setminus \{0\}) / \R_+$ that maps to $\zeta_i$ under the map $\alpha_i \to (\alpha_i)_e$. Similarly, let $(\beta_j)_e$ denote the underlying Reeb orbit for $\beta_j$, choose a point $\eta_j$ on each $(\beta_j)_e$, and for each $w_j \in P^-$, choose an element $r_j \in (T_{w_j} \Sigma \setminus \{0\}) / \R_+$ that maps to $\eta_j$ under the map $\beta_j \to (\beta_j)_e$. We refer to each such choice as an \textbf{asymptotic marker} at the relevant puncture; we refer to markers at positive punctures as \textbf{positive markers} and to markers at negative punctures as \textbf{negative markers}. Let $\mathbf r$ denote the set of markers that we have chosen.

Given orbit sets $\bm \alpha$ and $\bm \beta$, the moduli space of marked, punctured holomorphic curves from $\bm \alpha$ to $\bm \beta$ in $\R \times Y$ is the space of pairs $(u, \mathbf r)$, where $u$ is asymptotic to $\bm \alpha$ at the positive punctures and to $\bm \beta$ at the negative punctures, and $\mathbf r$ is a set of asymptotic markers for $u$, modulo biholomorphisms of domains that send positive punctures to positive punctures, negative punctures to negative punctures, positive markers to positive markers, and negative markers to negative markers. Moduli spaces of marked curves can be compactified using SFT buildings; see \cite{BEHWZ} for details.

Unmarked moduli spaces are defined similarly to marked moduli spaces, except we do not choose asymptotic markers at each puncture. Consequently, we identify two such maps if they are related by a biholomorphism of the domains that maps positive punctures to positive punctures and negative punctures to negative punctures. ECH uses unmarked moduli spaces and identifies two maps if they represent the same current in $\R \times Y$.

A curve $u \in \mathcal M_J(\bm \alpha, \bm \beta)$ has a \textbf{Fredholm index} given by
\begin{equation*}
	\ind(u)
		=
			- \chi(\dot \Sigma) + 2 c_1(u^*\xi, \tau) + \mu_\tau(\bm \alpha, \bm \beta),
\end{equation*}
where $c_1(u^*\xi, \tau)$ is the relative first Chern class of $\xi$ over $u$ in the trivialization $\tau$. See \cite[Section 2]{H2} for the definition of the relative first Chern class. If $\mathcal M_J(\bm \alpha, \bm \beta)$ is transversely cut out, then the (real) dimension of a neighborhood of $u \in \mathcal M_J(\bm \alpha, \bm \beta)$ is precisely $\ind(u)$ by results of Dragnev \cite{D}.

%===================%
% SECTION 2.4		 %
% ECH CHAIN COMPLEX %
%===================%

\subsection{The ECH chain complex}
\label{subsec:ech-chain-complex}

We now define the ECH chain complex with $\Z / 2 \Z$ coefficients. (It is possible to define ECH with $\Z$ coefficients, but we do not treat that case here.) Let $\Gamma \in H_1(Y)$ and let $J$ be a generic, admissible almost complex structure on $\R \times Y$. The groups $ECC(Y, \lambda, \Gamma, J)$ are generated by orbits sets $\bm \alpha = ((\alpha_1, m_1), (\alpha_2, m_2), \ldots, (\alpha_k, m_k))$ such that $m_i = 1$ if $\alpha_i$ is hyperbolic and such that
\begin{equation*}
	\sum_{i = 1}^k m_i [ \alpha_i ] = \Gamma.
\end{equation*}
Hutchings defines an \textbf{ECH index} $I$ for $J$-holomorphic currents $\mathcal C$ in $\R \times Y$. More specifically, he defines a \textbf{relative self-intersection number} $Q_\tau(\mathcal C)$ and sets
\begin{equation*}
	I(\mathcal C) = c_1( \xi|_{\mathcal C}, \tau ) + Q_\tau(\mathcal C) + \mu_\tau^I(\bm \alpha, \bm \beta).
\end{equation*}

The differential $\bdry$ counts punctured $J$-holomorphic currents with ECH index $1$ in $\R \times Y$ going from $\bm \alpha$ to $\bm \beta$. More precisely, consider the moduli space $\mathcal M_J^{I = 1}(\bm \alpha, \bm \beta)$ of $J$-holomorphic currents $\mathcal C$ with $I(\mathcal C) = 1$ that are asymptotic to $\bm \alpha$ at the positive ends and to $\bm \beta$ at the negative ends. There is an $\R$-action on $\mathcal M_J^{I = 1}(\bm \alpha, \bm \beta)$ induced by translation in the $\R$-direction of $\R \times Y$.

\begin{lemma}
\label{lemma:bdry-decreases-I}
	If $\mathcal M_J^{I = 1}(\bm \alpha, \bm \beta)$ is non-empty, then $\mathcal A(\bm \beta) < \mathcal A(\bm \alpha)$.
\end{lemma}

\begin{proof}
See \cite[Section 5]{H3}.
\end{proof}

\begin{lemma}\cite[Lemma 5.10]{H3}
\label{lemma:dim-one-manifold}
	If $J$ is generic and admissible and $\bm \alpha$ and $\bm \beta$ are orbit sets, then $\mathcal M_J^{I = 1}(\bm \alpha, \bm \beta) / \R$ is finite.
\end{lemma}

The differential $\bdry$ on the chain complex $ECC(Y, \lambda, \Gamma, J)$ is defined by
\begin{equation*}
	\bdry(\bm \alpha)
		=
			\sum_{\mathcal A(\bm \beta) < \mathcal A(\bm \alpha)} \# \left( \mathcal M_J^{I = 1}(\bm \alpha, \bm \beta) / \R \right) \cdot \bm \beta.
\end{equation*}

Currents counted by the differential $\bdry$ satisfy a rigid requirement on the multiplicities of their positive and negative ends. This requirement is crucial in \cite{HT1,HT2} to show that $\bdry^2 = 0$ and is leveraged extensively in this paper.

\begin{defn}
\label{defn:partition-conditions}
	Let $\alpha$ be an embedded hyperbolic Reeb orbit in $Y$. Let $\mathcal C$ be a $J$-holomorphic current in $\R \times Y$ with positive ends of multiplicities $m_1, m_2, \ldots, m_k$ and negative ends of multiplicities $n_1, n_2, \ldots, n_l$ at covers of $\alpha$. Set $m = \sum_{i = 1}^k m_i$ and $n = \sum_{j = 1}^l n_j$. We say that $\mathcal C$ satisfies the \textbf{ECH partition conditions} at its positive ends at $\alpha$ if the multiplicities $m_i$ are as in \Cref{table:partition-conditions}.
	\begin{table}[h!]
	\begin{center}
		\begin{tabular}{|c|c|c|}
		\hline
								& $m$ even 		& $m$ odd		 \\ \hline
		$\alpha$ positive hyperbolic 	& $(1, \ldots, 1)$	& $(1, \ldots, 1)$	 \\ \hline
		$\alpha$ negative hyperbolic	& $(2, \ldots, 2)$	& $(2, \ldots, 2, 1)$ \\
		\hline
		\end{tabular}
	\end{center}
\caption{The partition conditions for hyperbolic Reeb orbits.}
\label{table:partition-conditions}
\end{table}
	Similarly, we say that $\mathcal C$ satisfies the partition conditions at its negative ends if the multiplicities $n_j$ are as in \Cref{table:partition-conditions} with $m$ replaced by $n$. We say that $\mathcal C$ satisfies the ECH partition conditions if it satisfies the partition conditions at all of its positive and negative ends.
\end{defn}

\begin{rmk}
We do not concern ourselves with the partition conditions for elliptic Reeb orbits in this paper, as we work completely in the $L$-supersimple setting. Interested readers can consult \cite{H3} for details.
\end{rmk}

\begin{rmk}
A $J$-holomorphic curve $u \colon \dot \Sigma \to \R \times Y$ gives rise to a $J$-holomorphic current $\mathcal C = u(\dot \Sigma)$.
\end{rmk}

%==========================%
% SECTION 2.5				%
% THE L-SUPERSIMPLE SETTING %
%==========================%

\subsection{The $L$-supersimple setting}
\label{subsec:l-supersimple-setting}

We now review the relevant background for the $L$-supersimple setting of Bao-Honda. As stated in \Cref{sec:intro}, every non-degenerate contact form can be made into an $L$-supersimple form by a small perturbation. The precise statement of this result, which we take from \cite{BH1}, is as follows.

\begin{thm}\cite[Theorem 2.0.2]{BH1}
\label{thm:eliminate-elliptics}
	Let $\lambda$ be a non-degenerate contact form for $(Y, \xi)$. Then, for any $L > 0$ and $\epsilon > 0$, there exists a smooth function $\phi \colon Y \to \R_+$ such that
	\begin{enumerate}
		\item $\phi$ is $\epsilon$-close to $1$ with respect to a fixed $C^1$-norm;
		\item all the orbits of $R_{\phi \lambda}$ of $\phi \lambda$-action less than $L$ are hyperbolic.
	\end{enumerate}
	Moreover, we may assume that
	\begin{enumerate}
	\setcounter{enumi}{2}
		\item each positive hyperbolic orbit $\alpha$ has a neighborhood $(\R / \Z) \times D^2_{\delta_0}$ with coordinates $(t, x, y)$ such that
		\begin{enumerate}
			\item $D^2_{\delta_0} = \left\{ x^2 + y^2 \le \delta_0 \right\}$, where $\delta_0 > 0$ is small;
			\item $\phi \lambda = H \; dt + \eta$;
			\item $H = c(\alpha) - \epsilon x y$, with $c(\alpha), \epsilon > 0$ and $c(\alpha) \gg \epsilon$;
			\item $\eta = 2 x \, dy + y \, dx$;
			\item $\alpha = \left\{ x = y = 0 \right\}$.
		\end{enumerate}
		\item each negative hyperbolic orbit $\alpha$ has a neighborhood $([0, 1] \times D^2_{\delta_0}) / \sim$ with coordinates $(t, x, y)$, where $\sim$ identifies $(1, x, y) \sim (0, -x, -y)$ and the conditions (a) through (e) above hold.
	\end{enumerate}
\end{thm}

One major advantage of working in the $L$-supersimple setting is that the Fredholm index is well-behaved under taking multiple covers.

\begin{lemma}\cite[Lemma 3.3.2]{BH1}
\label{lemma:supersimple-fredholm-index}
	Let $(Y, \lambda)$ be a contact $3$-manifold and let $\bm \alpha$ and $\bm \beta$ be orbit sets where every orbit is hyperbolic. If $v$ is a $J$-holomorphic curve from $\bm \alpha$ to $\bm \beta$ in $\R \times Y$ and $u$ is a degree $k$ branched cover of $u$ with total branching order $b$, then
	\begin{equation*}
		\ind(u) = k \ind(v) + b.
	\end{equation*}
	In particular, $\ind(u) \ge 0$ for all $J$-holomorphic curves $u$ from $\bm \alpha$ to $\bm \beta$ in $\R \times Y$.
\end{lemma}

Another major advantage of the $L$-supersimple setting is that, by choosing the almost complex structure $J$ appropriately, we can ensure that the $\overline \bdry$-equation is linear for curves that are close to and graphical over trivial cylinders. The set of $J$ for which this assertion is true is described in the following definitions.

\begin{defn}
\label{defn:tame-J}\cite[Definition 3.1.2]{BH2}
	Let $\lambda$ be a contact form on $Y$. An almost complex structure $J$ on $\R \times Y$ is \textbf{$\lambda$-tame} if the following three conditions hold:
	\begin{enumerate}
		\item $J$ is $\R$-invariant;
		\item $J(\bdry_s) = g R_\lambda$ for some positive function $g$ on $Y$; and
		\item there exists a $2$-plane field $\xi'$ on $Y$ such that $J$ preserves $\xi'$, $d\lambda$ is a symplectic form on $\xi'$, and $J$ restricts to an orientation-preserving isomorphism on $\xi'$.
	\end{enumerate}
\end{defn}

\begin{defn}
\label{defn:L-tame-J}\cite[Definition 3.1.3]{BH2}
	Let $L > 0$, let $\lambda$ be an $L$-supersimple contact form, and let $\alpha$ be an embedded Reeb orbit of $\lambda$. A $\lambda$-tame almost complex structure $J$ is \textbf{$L$-simple} for $\lambda$ if, inside the neighborhood of $\alpha$ given by \Cref{thm:eliminate-elliptics}, the following conditions hold:
	\begin{enumerate}
		\item $\xi' = \Span\left( \bdry_x, \bdry_y \right)$;
		\item $J(\bdry_x) = \bdry_y$; and
		\item the function $g$ in \Cref{defn:tame-J} satisfies $g R_\lambda = \bdry_t + X_H$, where $X_H$ is the Hamiltonian vector field of the function $H$ from \Cref{thm:eliminate-elliptics} with respect to the symplectic form $dx \wedge dy$.
	\end{enumerate}
\end{defn}

We can now state the second advantage precisely.

\begin{prop}
\label{prop:linear-hol-curve-eqn} \cite{BH2}
	Let $\lambda$ be an $L$-supersimple contact form on $Y$ and let $J$ be an $L$-simple almost complex structure for $\lambda$. If $u \colon [R, \infty) \times S^1 \to \R \times Y$ is a $J$-holomorphic half-cylinder asymptotic to a Reeb orbit $\alpha$, and if we write $u(s, t) = (s, t, \widetilde u(s, t))$, then the function $\widetilde u$ satisfies
	\[
		\bdry_s \widetilde u + j_0 \bdry_t \widetilde u + S \widetilde u = 0,
	\]
	where $j_0$ is the standard complex structure on $\R^2$ and
	\[
		S = 
		\begin{pmatrix}
			0 & \epsilon \\
			\epsilon & 0
		\end{pmatrix}.
	\]
\end{prop}

\begin{proof}
See \cite{BH2} between Definition 3.1.3 and Convention 3.1.4.
\end{proof}

%======================%
% SECTION 3				%
% THE EVALUATION MAP	%
%======================%

\section{The Evaluation Map}
\label{sec:ev-map}

In this section, we review the Bao-Honda evaluation map from \cite{BH1,BH2}. It is used in \Cref{sec:models} to cut out $1$-dimensional families of holomorphic curves in high-dimensional moduli spaces.

Throughout this section, let $Y$ be a smooth $3$-manifold, let $\lambda$ be a non-degenerate, $L$-supersimple contact form on $Y$, and let $R_\lambda$ be the Reeb vector field of $\lambda$ on $Y$. All Reeb orbits and orbit sets under consideration in this section are tacitly assumed to have (total) action less than $L$.

%==========================%
% SECTION 3.1			       %
% THE ASYMPTOTIC OPERATOR %
%==========================%

\subsection{The asymptotic operator}
\label{subsec:asymp-operator}

Let $\gamma$ be a Reeb orbit of $\lambda$ with period $2 \pi a$, where $a \in \Z_+$. Recall from \cite{BH1} that there is an \textbf{asymptotic operator}
\begin{equation*}
	A_\gamma \colon W^{1, 2}(\R / 2 \pi a \Z, \R^2) \to L^2(\R / 2 \pi a \Z, \R^2)
\end{equation*}
defined by
\begin{equation*}
	A_\gamma = - j_0 \frac{\bdry}{\bdry t} - S(t).
\end{equation*}
Here, $j_0$ is the standard complex structure on $\R^2$ and $S(t)$ is a loop of $2 \times 2$ symmetric matrices. Recall from \cite{BH1} that the eigenspaces of $A_\gamma$ have dimension at most $2$. If $\gamma$ is negative hyperbolic, then every eigenspace of $A_\gamma$ has real dimension $2$, and if we label the eigenvalues so that
\begin{equation*}
	\cdots \le \lambda_{-2} \le \lambda_{-1} < 0 < \lambda_1 \le \lambda_2 \le \cdots,
\end{equation*}
then we can choose the corresponding eigenfunctions $\{ f_i(t) \}_{i \in \Z \setminus \{0\}}$ so that they form an orthonormal basis for $L^2(\R / 2 \pi a \Z, \R^2)$. If $\gamma$ is positive hyperbolic, then the eigenvalues can be labeled so that
\begin{equation*}
	\cdots \le \lambda_{-3} \le \lambda_{-2} < \lambda_{-1} < 0 < \lambda_1 < \lambda_2 \le \lambda_3 \le \cdots,
\end{equation*}
the eigenspaces for $\lambda_{\pm 1}$ have real dimension $1$, and all other eigenspaces have real dimension $2$. The corresponding eigenfunctions can again be chosen to be an orthonormal basis for $L^2(\R / 2 \pi a \Z, \R^2)$.

Next, we recall from \cite[Section 6]{BH1} some properties of the above-mentioned eigenfunctions. Let $u$ be a punctured holomorphic curve in $\R \times Y$ and suppose that $u$ has a negative end at $\gamma$. If we choose a trivialization $\tau$ of the contact structure $\xi = \Ker(\lambda)$ over $\gamma$, then the negative end of $u$ in question can be written in cylindrical coordinates $(s, t) \in (-\infty, -R] \times (\R / 2 \pi a \Z)$, $R \gg 0$, as the graph of a function $\widetilde u(s, t)$, i.e., we have
\begin{equation*}
	u(s, t) = (s, t, \widetilde u(s, t)).
\end{equation*}
In the $L$-supersimple setting, the function $\eta$ admits a Fourier-type expansion
\begin{equation*}
	\widetilde u(s, t)
		=
			\sum_{i = 1}^\infty c_i e^{\lambda_i s} f_i(t),
\end{equation*}
where the $c_i$ are real constants. Similarly, a positive end of $u$ asymptotic to $\gamma$ can be written as the graph of a function that has a Fourier-type expansion in negative-indexed eigenfunctions of $A_\gamma$.

Let $\rho_\tau(f_i)$ denote the winding number of the eigenfunction $f_i$ of $A_\gamma$. Recall the following facts from \cite[Lemma 6.4]{H1}.

\begin{fact}
\label{fact:eigenfunc-props} \hfill
	\begin{enumerate}
		\item
		If $i \le j$, then $\rho_\tau(f_i) \le \rho_\tau(f_j)$.
		
		\item
		We have
		\begin{equation*}
			\rho_\tau(f_1) = \left\lceil \frac{ CZ_\tau(\gamma) }{ 2 } \right\rceil
				\quad \text{and} \quad
			\rho_\tau(f_{-1}) = \left\lfloor \frac{ CZ_\tau(\gamma) }{ 2 } \right\rfloor.
		\end{equation*}
	\end{enumerate}
\end{fact}

\begin{defn} \cite[Definition 3.2]{HT2}
\label{defn:non-degenerate}
	A $J$-holomorphic curve $u$ has \textbf{non-degen\-erate ends} if at each negative (resp.\ positive) end, the coefficient $c_1$ (resp.\ $c_{-1}$) in the Fourier-type expansion of $u$ is non-zero.
\end{defn}

\begin{defn} \cite[Definition 3.8]{HT2}
\label{defn:non-overlapping}
	A $J$-holomorphic curve $u$ has \textbf{non-overla\-pping ends} if it has non-degenerate ends and and following holds. For every pair of negative (reps.\ positive) ends asymptotic to covers $\gamma^{a_1}$ and $\gamma^{a_2}$ of the same Reeb orbit $\gamma$ where the smallest positive (reps.\ largest negative) eigenvalues of $A_{\gamma^{a_1}}$ and $A_{\gamma^{a_2}}$ coincide, the leading coefficients in the Fourier-type expansions of $u$ do not differ by a factor of a $d^\text{th}$ root of unity, where $d = \gcd(a_1, a_2)$.
\end{defn}

%===================%
% SECTION 3.2		  %
% THE EVALUATION MAP %
%===================%

\subsection{The evaluation map}
\label{subsec:evaluation-map}

We now recall the definition of the evaluation map in \cite{BH1} and review some of the map's properties. Throughout, we use $\mathcal M_J(\bm \alpha, \bm \beta)$ to denote a transversely cut out moduli space of $J$-holomorphic curves in $\R \times Y$ with positive orbit set $\bm \alpha$ and negative orbit set $\bm \beta$.

\begin{defn}
\label{defn:half-cyl-ev-map}
	Let $u \colon (-\infty, -R] \times (\R / 2 \pi a \Z) \to \R \times Y$ be a $J$-holomorphic half-cylinder, and assume that $u$ is asymptotic to a Reeb orbit $\gamma$ at the negative end. Write $u$ in cylindrical coordinates as the graph of a function $\widetilde u(s, t)$, and write the Fourier-type expansion of $\widetilde u$ as
	\begin{equation*}
		\widetilde u(s, t) = \sum_{i = 1}^\infty c_i e^{\lambda_i s} f_i(t),
	\end{equation*}
	where the $c_i$ are real constants. Then the \textbf{evaluation map} on $u$ is defined as
	\begin{equation*}
		\ev(u) = (c_1, c_2).
	\end{equation*}
\end{defn}

\begin{defn}
\label{defn:single-end-ev-map}
	Let $u \in \mathcal M_J(\bm \alpha, \bm \beta)$ and label the negative ends of $u$ by $1, 2, \ldots, m$. The \textbf{evaluation map} at the $i^\text{th}$ negative end of $u$ is defined as
	\begin{gather*}
		\ev_i		\colon	\mathcal M_J(\bm \alpha, \bm \beta) \to \R^2 \\
			u 	\mapsto 	(c_1, c_2),
	\end{gather*}
	where we have identified $u$ with a half-cylinder near the $i^\text{th}$ end and used the evaluation map from \Cref{defn:half-cyl-ev-map}.
\end{defn}

\begin{defn}
\label{defn:multi-end-ev-map}
	Assume the setup in \Cref{defn:single-end-ev-map} and let $I = \{ i_1, \ldots, i_p \}$ be a subset of $\{1, \ldots, m\}$. At the $l^\text{th}$ negative end of $u$, write the Fourier-type series as
	\begin{equation*}
		\sum_{i > 0} c_{l, i} e^{\lambda_i s} f_i(t).
	\end{equation*}
	The \textbf{evaluation map} at the ends specified by $I$ is defined as
	\begin{gather*}
		\ev_{I}	\colon	\mathcal M_J(\bm \alpha, \bm \beta)	\to \prod_{i \in I} \R^2 \\
						u 	\mapsto 	\left( \ev_i(u) \right)_{i \in I}.
	\end{gather*}
	The \textbf{total order} of the map $\ev_{I}$ is defined to be $2 \left| I \right|$.
	
	If all of the ends labeled by $I$ have odd multiplicity, the asymptotic eigenspaces at those ends all have multiplicity $2$. If the relevant asymptotic operators are also complex-linear, we can view the eigenspaces as complex vector spaces and take complex coefficients with the evaluation map. This modification is used extensively in \Cref{sec:models}. We can also define higher-order evaluation maps $\ev_i^k \colon \mathcal M_J(\bm \alpha, \bm \beta) \to \R^k$, where $i \in I$ and $k > 2$, by $\ev_i^k(u) = (c_1, \ldots, c_k)$, where we have identified $u$ with a half-cylinder near the $i^\text{th}$ end. This modification will be used in \Cref{sec:cob-map-defn}.
\end{defn}

\begin{rmk}
\label{rmk:evmap-translate}
	If $\ev_i^k(u) = (c_1, \ldots, c_k)$ and $v$ is the curve obtained by translating $u$ by $a$ in the $\R$-direction, then $\ev_i^k(v) = (e^{-\lambda_1 a} c_1, \ldots, e^{- \lambda_k a} c_k)$.
\end{rmk}

\begin{fact}
	The above evaluation maps are all smooth.
\end{fact}

%================%
% SECTION 3.3	   %
% TRANSVERSALITY %
%================%

\subsection{Transversality for the evaluation map}
\label{subsec:ev-map-transversality}

One of the key advantages of the $L$-supersimple setting exploited in \cite{BH1,BH2} is the abundance of transversality for the evaluation map at ends of punctured holomorphic curves. We now briefly justify why similar transversality results hold for evaluation maps on multiple ends, beginning with a mild generalization of \cite[Theorem 6.0.4]{BH1}.

\begin{thm}
\label{thm:ev-map-transversality}
	Let $J$ be generic, and let $\mathcal M_J(\bm \alpha, \bm \beta)$ be a transversely cut out moduli space of curves in $\R \times Y$ with Fredholm index $k$. Let $K \subset \mathcal M_J(\bm \alpha, \bm \beta)$ be compact and let $Z \subset \R^{k-1}$ be a submanifold. Then there exists a generic $J'$, arbitrarily close to $J$, and a compact subset $K' \subset \mathcal M_{J'}(\bm \alpha, \bm \beta)$, arbitrarily close to $K$, such that the evaluation map $\ev_{I}$ on $K'$ is transverse to $Z$.
\end{thm}

\begin{proof}
Let $u \in \mathcal M_J(\bm \alpha, \bm \beta)$. The perturbation constructed in the proof of \cite[Theorem 6.0.4]{BH1} is supported over a single end, so we can repeat the construction over the relevant ends of $u$ separately.
\end{proof}

\begin{prop}
\label{prop:ev-map-non-degen}
	Let $J$ be generic, let $\mathcal M$ be a transversely cut out moduli space of punctured holomorphic curves in $\R \times Y$, and consider the evaluation map $\ev_{I}$ on $\mathcal M$. The set of $u \in \mathcal M$ such that $\ev_{I}(u)$ intersects a coordinate hyperplane $\{ x_i = 0 \}$ in $\R^{2 \left| I \right|}$ has codimension $1$ in $\mathcal M$.
\end{prop}

\begin{proof}
By \Cref{thm:ev-map-transversality}, we can make $\ev_{I}$ transverse to the coordinate plane $\{ x_i = 0 \}$ if $J$ is generic.
\end{proof}

\begin{rmk}
\label{rmk:higher-order-analogue}
	Analogues of \Cref{thm:ev-map-transversality} and \Cref{prop:ev-map-non-degen} hold for higher-order evaluation maps $\ev_i^k \colon \mathcal M \to \R^k$.
\end{rmk}

%======================%
% SECTION 4				%
% THE INDEX FORMULA		%
%======================%

\section{Index Calculations}
\label{sec:ech-ind}

We prove \Cref{thm:gen-ind-ineq} in this section. We use it in \Cref{sec:degen} to classify degenerations of $1$-dimensional families of $J$-holomorphic curves in the $L$-supersimple setting. The proof involves strengthening the various inequalities involved in Hutchings' proof of the inequality \eqref{eqn:hut-ind-ineq}. The relationship between $\Delta(u)$ and the ECH partition conditions is partially expressed in the following result, whose proof follows easily from the derivation of the formulas for $\Delta(u, \gamma)$ in \Cref{subsec:link-nums}. The subsequent corollary is a crucial ingredient to our arguments in \Cref{sec:cob-map-defn}.

\begin{prop}
\label{prop:deficit-and-part-cond}
	If $J$ is generic, $u$ is a somewhere injective $J$-holomorphic curve in a symplectization, and $\Delta(u) = 0$, then $u$ satisfies the ECH partition conditions.
\end{prop}

\begin{cor}
\label{cor:ind-equal}
	If $J$ is generic, $u$ is a somewhere injective $J$-holomorphic curve in a symplectization, and $I(u) = \ind(u)$, then $u$ is embedded and satisfies the ECH partition conditions.
\end{cor}

%=============%
% SECTION 4.1    %
% INGREDIENTS %
%=============%

\subsection{Ingredients in the proof of Hutchings' inequality}
\label{subsec:hut-ineq-ingr}

We begin by fixing notation and collating the results used in Hutchings' proof of \eqref{eqn:hut-ind-ineq}. Our notation closely, but not exactly, matches that used in \cite{H1}. We only analyze negative ends in this discussion; the analysis for positive ends is similar.

\begin{notation}
\label{notation:ech-braid}
	Let $u$ be a somewhere injective $J$-holomorphic curve in a symplectization. Let $\bm \beta$ be the negative orbit set of $u$ and fix an embedded Reeb orbit $\beta \in \Gamma^-(u)$. Let $m$ be the multiplicity of $\beta$ in the orbit set $\bm \beta$, let $n$ be the number of negative ends of $u$ that are asymptotic to (covers of) $\beta$, and let $q_1, q_2, \ldots, q_n$ be the multiplicities of these negative ends. Let $\zeta_1, \zeta_2, \ldots, \zeta_n$ be the braids determined by these negative ends, and let $\bm \zeta$ denote the union of the braids $\zeta_1, \ldots, \zeta_n$. Let $\tau$ denote a trivialization of $\xi$ over $\beta$. Let $\mu_\tau(\beta^k)$ denote the Conley-Zehnder index of the $k$-fold cover of $\beta$. For each braid $\zeta_i$, let $\rho_\tau(\zeta_i)$ denote the winding number of $\zeta_i$ around $\beta$ it the trivialization $\tau$, let $w_\tau(\zeta_i)$ the asymptotic writhe of $\zeta_i$ with respect to $\tau$, and let $\ell_\tau(\zeta_i, \zeta_j)$ denote the linking number of the braids $\zeta_i$ and $\zeta_j$ with respect to $\tau$.
\end{notation}

With the above notation, the five ingredients in the proof of Hutchings' inequality are the following.

\begin{equation}
\label{eqn:hut-eqn-1}
	\rho_\tau(\zeta_i)
		\ge
			\left\lceil \frac{\mu_\tau(\beta^{q_i})}{2} \right\rceil
\end{equation}

\begin{equation}
\label{eqn:hut-eqn-2}
	w_\tau(\zeta_i)
		\ge
			(q_i - 1) \rho_\tau(\zeta_i)
\end{equation}

\begin{equation}
\label{eqn:hut-eqn-3}
	\ell_\tau(\zeta_i, \zeta_j)
		\ge
			\min( q_i \rho_\tau(\zeta_j), q_j \_\tau(\zeta_i) )
\end{equation}

\begin{equation}
\label{eqn:hut-eqn-4}
	w_\tau(\bm \zeta)
		\ge
			\sum_{i = 1}^{n} \rho_\tau(\zeta_i) (q_i - 1) + \sum_{i \neq j} \min( q_i \rho_\tau(\zeta_j), q_j \rho_\tau(\zeta_i) )
\end{equation}

\begin{equation}
\label{eqn:hut-eqn-5}
	\sum_{i = 1}^n \rho_\tau(\zeta_i) (q_i - 1)
		+
			\sum_{i \neq j} \min( q_i \rho_\tau(\zeta_j), q_j \rho_\tau(\zeta_i) )
				\ge
					\sum_{k = 1}^m \mu_\tau(\beta^k)
						-
							\sum_{i = 1}^n \mu_\tau(\beta^{q_i})
\end{equation}

%==================%
% SECTION 4.2	        %
% THE WRITHE BOUND %
%==================%

\subsection{The writhe bound}
\label{subsec:writhe-bound}

We first use \Cref{prop:ev-map-non-degen} to improve \eqref{eqn:hut-eqn-1} slightly. Our proof of the next Lemma closely follows the one given for \cite[Lemma 6.6]{H1}

\begin{lemma}
\label{lemma:hut-eqn-1-imp}
	Let $u$ be a somewhere injective curve in $\R \times Y$. Assume that the contact form on $Y$ is $L$-supersimple and that $\beta$ is an orbit in the negative orbit set $\bm \beta$ of $u$. If $J$ is generic, then
	\begin{equation}
	\label{eqn:hut-eqn-1-imp}
		\rho_\tau(\zeta_i) \ge \left\lceil \frac{ \mu_\tau(\beta^{q_i}) }{ 2 } \right\rceil.
	\end{equation}
	Equality holds if $\mathcal A(\bm \beta) < L$ and $\ind(u) = 1$.
\end{lemma}

\begin{proof}
The inequality is proved in \cite[Lemma 6.6]{H1}. So assume that $\mathcal A(\bm \beta) < L$ and $\ind(u) = 1$. Let $(s, t)$ be cylindrical coordinates over the relevant negative end of $u$ and take an asymptotic expansion
\[
	u(s, t) = \left( s, t, \sum_{i = 1}^\infty c_i e^{-\lambda_i s} f_i(t) \right)
\]
of $u$ for $s \ll 0$, as in \Cref{sec:ev-map}. Since $u$ has Fredholm index $1$ and $J$ is generic, \Cref{prop:ev-map-non-degen} implies that $c_1 \neq 0$. Thus, $\rho_\tau(\zeta_i)$ equals the winding number of $f_1$ around $\beta$ in the trivialization $\tau$. By computations in \cite[Section 3]{HWZ}, said winding number is precisely $\left\lceil \dfrac{\mu_\tau(\gamma^{q_i})}{2} \right\rceil$.
\end{proof}

We now turn our attention to \eqref{eqn:hut-eqn-2}. Our proof of the next Lemma closely follows the one given for \cite[Lemma 6.7]{H1}.

\begin{lemma}
\label{lemma:hut-eqn-2-imp}
	Let $u$ be a somewhere injective curve in $\R \times Y$. Assume that the contact form on $Y$ is $L$-supersimple and that $\beta$ is an orbit in the negative orbit set $\bm \beta$ of $u$. If $J$ is generic, then
	\begin{equation}
	\label{eqn:hut-eqn-2-imp}
		w_\tau(\zeta_i) \ge \rho_\tau(\zeta_i) (q_i - 1) + (d_i - 1),
	\end{equation}
	where $d_i = \gcd(q_i, \rho_\tau(\zeta_i))$. Equality holds if $\mathcal A(\bm \beta) < L$ and $\ind(u) = 1$.
\end{lemma}

\begin{proof}
Direct calculation shows that equality holds when $\rho_\tau(\zeta_i) = q_i$. We proceed by complete induction on $\rho_\tau(\zeta_i)$. First assume that $d_i = 1$. (This is true when $\rho_\tau(\zeta_i) = 1$, but the more general result is useful in the inductive step.) The proof of \cite[Lemma 6.7]{H1} shows that $w_\tau(\zeta_i) = \rho_\tau(\zeta_i) (q_i - 1)$. Now assume $\rho_\tau(\zeta_i) > 1$ and $d_i > 1$. The same proof shows that $\zeta_i$ is the cabling of a braid $\zeta_i'$ with $q_i / d_i$ strands and winding number $\rho_\tau(\zeta_i) / d_i$ by a braid $\zeta_i''$ with $d_i$ strands and winding number $\rho_\tau(\zeta_i'') \ge \rho_\tau(\zeta_i)$. Write $\rho_\tau(\zeta_i'') = \rho_\tau(\zeta_i) + k$ and $d_i' = \gcd(\rho_\tau(\zeta_i''), d_i)$. We know inductively that
\begin{equation*}
	w_\tau(\zeta_i') = \frac{ \rho_\tau(\zeta_i) }{ d_i } \left( \frac{ q_i }{ d_i } - 1 \right)
	\quad \text{and} \quad
	w_\tau(\zeta_i'') \ge (\rho_\tau(\zeta_i) + k) ( d_i - 1 ) + (d_i' - 1),
\end{equation*}
and thus
\begin{align*}
	w_\tau(\zeta_i)
		&=
			d_i^2 w_\tau(\zeta_i') + w_\tau(\zeta_i'')			\\
		&\ge
			\rho_\tau(\zeta_i) (q_i - d_i) + (\rho_\tau(\zeta_i) + k)(d_i - 1) + (d_i' - 1)	\\
		&=
			\rho_\tau(\zeta_i) (q_i - 1) + k (d_i - 1) + (d_i' - 1).
\end{align*}
If $k = 0$, then $d_i' = d_i$ and
\begin{equation*}
	w_\tau(\zeta_i) = \rho_\tau(\zeta_i) (q_i - 1) + (d_i - 1).
\end{equation*}
If $k > 0$, then
\begin{align*}
	w_\tau(\zeta_i)
		&=
			\rho_\tau(\zeta_i) (q_i - 1) + (d_i - 1) + (k - 1)(d_i - 1) + (d_i' - 1)	\\
		&\ge
			\rho_\tau(\zeta_i) (q_i - 1) + (d_i - 1).
\end{align*}

Now assume that $\mathcal A(\bm \beta) < L$ and $\ind(u) = 1$. Equality is also proved by induction on $\rho_\tau(\zeta_i)$, and the case $\rho_\tau(\zeta_i) = 1$ is handled in the same way (i.e., by proving the result when $d_i = 1$). So assume $\rho_\tau(\zeta_i) > 1$ and $d_i > 1$. By \Cref{prop:ev-map-non-degen}, either $\rho_\tau(\zeta_i'') = \rho_\tau(\zeta_i)$ or $\rho_\tau(\zeta_i'') = \rho_\tau(\zeta_i) + 1$. The former is the case $k = 0$ above, where $d_i' = d_i$. Here, we know inductively that $w_\tau(\zeta_i'') = \rho_\tau(\zeta_i) (d_i - 1) + (d_i - 1) = (\rho_\tau(\zeta_i) + 1) (d_i - 1)$, so
\begin{align*}
	w_\tau(\zeta_i)
		&=
			d_i^2 w_\tau(\zeta_i') + w_\tau(\zeta_i'')			\\
		&=
			\rho_\tau(\zeta_i) (q_i - d_i) + (\rho_\tau(\zeta_i) + 1) (d_i - 1)			\\
		&=
			\rho_\tau(\zeta_i) (q_i - 1) + (d_i - 1).
\end{align*}
The latter is the case $k = 1$ above, where $d_i' = 1$. Here, we again know inductively that $w_\tau(\zeta_i'') = (\rho_\tau(\zeta_i) + 1) (d_i - 1)$, and equality follows as in the previous case.
\end{proof}

%=================%
% SECTION 4.3	     %
% LINKING NUMBERS %
%=================%

\subsection{Linking numbers}
\label{subsec:link-nums}

Now we turn out attention to \eqref{eqn:hut-eqn-3}. If $J$ is a generic almost complex structure on $\R \times Y$, any Fredholm index $1$, simple curve $u$ in $\R \times Y$ has non-degenerate and non-overlapping ends. In particular, the proof of \cite[Lemma 6.9]{H1} implies the following strengthened result.

\begin{lemma}
\label{lemma:hut-eqn-3-imp}
	Let $u$ be a somewhere injective curve in $\R \times Y$. Assume that the contact form on $Y$ is $L$-supersimple and that $\beta$ is an orbit in the negative orbit set $\bm \beta$ of $u$. If $J$ is generic, then
	\begin{equation}
	\label{eqn:hut-eqn-3-imp}
		\ell_\tau(\zeta_i, \zeta_j) \ge \min( q_i \rho_\tau(\zeta_j), q_j \rho_\tau(\zeta_i) ).
	\end{equation}
	Equality holds if $\mathcal A(\bm \beta) < L$ and $\ind(u) = 1$.
\end{lemma}

Now we put the preceding lemmas together to derive stronger versions of \eqref{eqn:hut-eqn-4} and \eqref{eqn:hut-eqn-5} in the $L$-supersimple setting; these new inequalities are implicit in work of Hutchings \cite{H2}.

\begin{lemma}
\label{lemma:hut-eqn-4-imp-neg}
	Let $u$ be a somewhere injective curve in $\R \times Y$. Assume that the contact form on $Y$ is $L$-supersimple and that $\beta$ is a negative hyperbolic orbit in the negative orbit set $\bm \beta$ of $u$. As in \Cref{notation:ech-braid}, suppose that $u$ has negative ends of multiplicity $q_1, \ldots, q_n$ at $\beta$. In addition, order the ends of $u$ at $\beta$ so that $q_1, \ldots, q_k$ are the ends with odd multiplicity, ordered so that $q_1 \ge q_2 \ge \cdots \ge q_k$ and so that $q_{k + 1}, q_{k + 2}, \ldots q_n$ are the ends with even multiplicity. Then
	\begin{equation}
	\label{eqn:hut-eqn-4-imp-neg}
		w_\tau(\bm \zeta)
			\ge
				\sum_{i = 1}^m \mu_\tau(\beta^i)
					- \sum_{i = 1}^n \mu_\tau(\beta^{q_i})
					+ \sum_{i = 1}^k \left( \frac{q_i - 1}{2} + i - 1 \right)
					+ \sum_{i = k + 1}^n \left( \frac{q_i}{2} - 1 \right).
	\end{equation}
	Equality holds if $\mathcal A(\bm \beta) < L$ and $\ind(u) = 1$.
\end{lemma}

\begin{proof}
Choose the trivialization $\tau$ so that $\mu_\tau(\beta) = 1$ and set
\begin{equation*}
	\rho_i
	=
	\left\lceil \frac{ \mu_\tau(\beta^{q_i}) }{ 2 } \right\rceil.
\end{equation*}
Set $d_i = \gcd(q_i, \rho_\tau(\zeta_i))$ and note that by \eqref{eqn:hut-eqn-1-imp},
\begin{equation*}
	\rho_\tau(\zeta_i)(q_i - 1) + (d_i - 1)
		\ge
			\begin{cases}
				\rho_i(q_i - 1),				&	i = 1, 2, \ldots, k		\\
				\rho_i(q_i - 1) + \left(\frac{q_i}{2} - 1\right),	&	i = k + 1, k + 2, \ldots, n
			\end{cases}.
\end{equation*}
The above inequality, combined with \eqref{eqn:hut-eqn-1-imp}, \eqref{eqn:hut-eqn-2-imp}, and \eqref{eqn:hut-eqn-3-imp}, implies that
\begin{align*}
	w_\tau(\bm \zeta)	&=		\sum_{i = 1}^n w_\tau(\zeta_i) + \sum_{i \neq j} \ell_\tau(\zeta_i, \zeta_j)	\\
					&\ge		\sum_{i = 1}^n \left[ \rho_\tau(\zeta_i) (q_i - 1) + (d_i - 1) \right] + \sum_{i \neq j} \min( q_i \rho_\tau(\zeta_j), q_j \rho_\tau(\zeta_i))	\\
					&\ge		\sum_{i = 1}^n \rho_i (q_i - 1) + \sum_{i \neq j} \min(q_i \rho_j, q_j \rho_i) + \sum_{i = k + 1}^n \left( \frac{q_i}{2} - 1 \right).
\end{align*}
By a computation in the proof of \cite[Lemma 4.19]{H2}, we have
\begin{equation*}
	\sum_{i = 1}^n \rho_i (q_i - 1) + \sum_{i \neq j} \min(q_i \rho_j, q_j \rho_i)
		=
			\sum_{i = 1}^m \mu_\tau(\beta^i)
			- \sum_{i = 1}^n \mu_\tau(\beta^{q_i})
			+ \sum_{i = 1}^k \left( \frac{q_i - 1}{2} + i - 1 \right),
\end{equation*}
and the result follows.
\end{proof}

\begin{lemma}
\label{lemma:hut-eqn-4-imp-pos}
	Let $u$ be a somewhere injective curve in $\R \times Y$. Assume that the contact form on $Y$ is $L$-supersimple and that $\beta$ is a positive hyperbolic orbit in the negative orbit set $\bm \beta$ of $u$. As in \Cref{notation:ech-braid}, suppose that $u$ has negative ends of multiplicity $q_1, \ldots, q_n$ at $\beta$. Then
	\begin{equation}
	\label{eqn:hut-eqn-4-imp-pos}
		w_\tau(\bm \zeta)
			\ge
				\sum_{i = 1}^m \mu_\tau(\beta^i)
				- \sum_{i = 1}^n \mu_\tau(\beta^{q_i})
				+ \sum_{i = 1}^n \left( q_i - 1 \right).
	\end{equation}
	Equality holds if $\mathcal A(\bm \beta) < L$ and $\ind(u) = 1$.
\end{lemma}

\begin{proof}
Choose the trivialization $\tau$ so that $\mu_\tau(\beta) = 0$ and set $d_i = \gcd( q_i, \rho_\tau(\zeta_i) )$. By \eqref{eqn:hut-eqn-1-imp}, we have $\rho_\tau(\zeta_i) \ge 0$. There are two cases, $\rho_\tau(\zeta_i) = 0$ and $\rho_\tau(\zeta_i) > 0$, and in both cases we have $w_\tau(\zeta_i) \ge q_i - 1$. Hence, with this choice of $\tau$, the inequalities \eqref{eqn:hut-eqn-2-imp} and \eqref{eqn:hut-eqn-3-imp} imply that
\begin{align*}
	w_\tau(\bm \zeta)	&=		\sum_{i = 1}^n w_\tau(\zeta_i) + \sum_{i \neq j} \ell_\tau(\zeta_i, \zeta_j)	\\
					&\ge		\sum_{i = 1}^n (q_i - 1) + \sum_{i \neq j} \min( q_i \rho_\tau(\zeta_j), q_j \rho_\tau(\zeta_i))	\\
					&\ge		\sum_{i = 1}^n \left( q_i - 1 \right)
\end{align*}
With our choice of $\tau$, we have
\begin{equation*}
	\sum_{i = 1}^m \mu_\tau(\beta^i) - \sum_{i = 1}^n \mu_\tau(\beta^{q_i})
	=
	0,
\end{equation*}
and the result follows.
\end{proof}

%==============================%
% SECTION 4.4				       %
% PROOF OF THE INDEX INEQUALITY %
%==============================%

\subsection{Proof of the inequality}
\label{subsec:proof-of-index-inequality}

The proofs of \Cref{lemma:hut-eqn-4-imp-neg,lemma:hut-eqn-4-imp-pos} show that we have
\begin{equation*}
	w_\tau(\bm \zeta)
		\ge
			\sum_{i = 1}^m \mu_\tau(\beta^i)
			- \sum_{i = 1}^n \mu_\tau(\beta^{q_i})
			+ \Delta(u, \beta)
\end{equation*}
and that equality holds if $\mathcal A(\bm \beta) < L$ and $\ind(u) = 1$. If $u$ has a positive end at $\alpha$, computations similar to those in \Cref{lemma:hut-eqn-4-imp-neg,lemma:hut-eqn-4-imp-pos} show that
\begin{equation*}
	w_\tau(\bm \zeta)
		\le
			\sum_{i = 1}^m \mu_\tau(\alpha^i)
			- \sum_{i = 1}^n \mu_\tau(\alpha^{q_i})
			- \Delta(u, \alpha)
\end{equation*}
and that equality holds if $\mathcal A(\bm \alpha) < L$ and $\ind(u) = 1$. Thus, if we set
\begin{equation*}
	w_\tau(u)
		=
			\sum_{\text{positive ends}} w_\tau(\bm \zeta)
			- \sum_{\text{negative ends}} w_\tau(\bm \zeta),
\end{equation*}
we have
\begin{equation*}
	w_\tau(u)
		\le
			\mu_\tau^I(\bm \alpha, \bm \beta)
			- \mu_\tau(\bm \alpha, \bm \beta)
			- \Delta(u),
\end{equation*}
and equality holds if $\mathcal A(\bm \alpha), \mathcal A(\bm \beta) < L$ and $\ind(u) = 1$. Thus, $\Delta(u)$ measures how much the curve $u$ violates the ECH partition conditions at its ends.

\begin{proof}[Proof of \Cref{thm:gen-ind-ineq}]
	Recall the relative adjunction formula for somewhere injective curves:
	\begin{equation}
	\label{eqn:rel-adj-formula}
		c_1(u^*\xi, \tau)
			=
				\chi(\dot \Sigma)
				+ Q_\tau(u)
				+ w_\tau(u)
				- 2 \delta(u).
	\end{equation}
	By the above formula for the asymptotic writhe, we have
	\begin{align*}
		I(u)
			&=
				c_1(u^*\xi, \tau) + Q_\tau(u) + \mu_\tau^I(\bm \alpha, \bm \beta)	\\
			&=
				- \chi(\dot \Sigma) + 2 c_1(u^*\xi, \tau) - w_\tau(u) + 2 \delta(u) + \mu_\tau^I(\bm \alpha, \bm \beta)	\\
			&\ge
				- \chi(\dot \Sigma) + 2 c_1(u^*\xi, \tau) + \mu_\tau(\bm \alpha) - \mu_\tau(\bm \beta) + 2 \delta(u) + \Delta(u)	\\
			&=
				\ind(u) + 2 \delta(u) + \Delta(u).
	\end{align*}
	Equality clearly holds if $\mathcal A(\bm \alpha) < L$ and $\ind(u) = 1$.
\end{proof}

\begin{rmk}
\label{rmk:cob-ineq}
	The inequality \eqref{eqn:gen-ech-ineq} also holds for curves in exact symplectic cobordisms; equality holds if $\mathcal A(\bm \alpha), \mathcal A(\bm \beta) < L$ and $\ind(u) = 0$. In general, equality holds if $\mathcal A(\bm \alpha), \mathcal A(\bm \beta) < L$ and $u$ has non-degenerate, non-overlapping ends.
\end{rmk}

%==================%
% SECTION 5			%
% DEGENERATIONS	%
%==================%

\section{Degenerations in Cobordisms}
\label{sec:degen}

In this section, we prove \Cref{thm:degen-class-thm}. We first recall its setup. Let $(Y_\pm, \lambda_\pm)$ be $L$-supersimple contact $3$-manifolds and let $(X, \lambda)$ be an exact symplectic cobordism from $(Y_+, \lambda_+)$ to $(Y_-, \lambda_-)$. Let $J$ be a generic, $L$-simple, admissible almost complex structure on the completion $(\widehat X, \widehat \lambda)$ that restricts to $L$-simple, admissible almost complex structures $J_+$ and $J_-$ on the ends $[0, \infty) \times Y_+$ and $(-\infty, 0] \times Y_-$, respectively, of $\widehat X$. Let $\bm \alpha$ be a generator of $ECC^L(Y_+, \lambda_+, J_+)$ and let $\bm \beta$ be a generator of $ECC^L(Y_-, \lambda_-, J_-)$. Consider the moduli space $\mathcal M_X^{1, 1}(\bm \alpha, \bm \beta)$ and let $\overline{\mathcal M}_X^{1, 1}(\bm \alpha, \bm \beta)$ denote its SFT compactification as described in \cite{BEHWZ}. As before, we denote an SFT building in $\bdry \overline{\mathcal M}_X^{1, 1}(\bm \alpha, \bm \beta)$ by $[u_{-a}] \cup \cdots \cup [u_{-1}] \cup u_0 \cup [u_1] \cup \cdots [u_b]$, where $a$ and $b$ are positive integers, the levels go from bottom to top as we read from left to right, the levels with negative indices are in $(\R \times Y_-) / \R$, the level $u_0$ is in $\widehat X$, and the levels with positive indices are in $(\R \times Y_+) / \R$.

Let $[u_{-a}] \cup \cdots \cup [u_{-1}] \cup u_0 \cup [u_1] \cup \cdots [u_b]$ be a building in $\bdry \overline{\mathcal M}_X^{1, 1}(\bm \alpha, \bm \beta)$. By \Cref{lemma:supersimple-fredholm-index}, each level of the building has non-negative Fredholm index, and the symplectization levels have positive Fredholm index. Since $\ind$ is additive and the total Fredholm index of the building is $1$, there must be only one symplectization level, which has Fredholm index $1$, and the cobordism level $u_0$ must have Fredholm index $0$.

%==========================%
% SECTION 5.1			       %
% MULTIPLY COVERED CURVES %
%==========================%

\subsection{Multiply covered curves}
\label{subsec:multiply-covered-curves}

We begin with a classification of multiply covered curves in $\widehat X$ with non-positive ECH index in the $L$-supersimple setting.

\begin{lemma}
\label{lemma:neg-index-curves}
Let $u$ be a $J$-holomorphic curve in $\widehat X$ with $\ind(u) = 0$ and connected image. The curve $u$ has negative ECH index if and only if it is an unbranched, disconnected cover of a $J$-holomorphic plane in with ECH and Fredholm index $0$. In that case,
\[
	I(u) = - \binom{d}{2},
\]
where $d$ is the degree of the covering.
\end{lemma}

\begin{proof}
Suppose that $I(u) < 0$. By the ECH index inequality \eqref{eqn:hut-ind-ineq}, somewhere injective curves in cobordisms have non-negative ECH index, so $u$ must be a $d$-fold multiple cover of a somewhere injective curve $v \colon \dot \Sigma' \to \widehat X$ with $\ind(v) \ge 0$ and $d \ge 2$. Recall the index inequality
\begin{equation}
\label{eq:multiple-cover-inequality}
	I(u) \ge d \cdot I(v) + \binom{d}{2} ( 2g(\dot \Sigma') - 2 + \ind(v) + h(v) )
\end{equation}
from \cite{H2}, where $h(v)$ is the number of ends of $v$ at hyperbolic orbits. Since $h(v) \ge 1$ and $\ind(v) \ge 0$, the only way for $I(u)$ to be negative is if $g(\dot \Sigma') = 0$ and $\ind(v) + h(v) = 1$. Since $\ind(u) = 0$, \Cref{lemma:supersimple-fredholm-index} implies that $u$ is an unbranched cover of $v$ and $\ind(v) = 0$. Hence $h(v) = 1$. It follows that $u$ is an unbranched, disconnected cover of a plane $v$. Let $\gamma$ be the orbit at the positive end of $v$. If we choose the trivialization $\tau$ of $\gamma^*\xi$ such that $c_1(v^*\xi, \tau) = 0$ we see that $0 = \ind(v) = \mu_\tau(\gamma) - 1$, so $\mu_\tau(\gamma) = 1$. Thus, $I(v) = Q_\tau(v) + 1$. If $I(v) \ge 1$, then $Q_\tau(v) \ge 0$, and an easy computation shows that $I(u) > 0$. Thus, $I(v) = 0$.

Suppose there is a component $\dot \Sigma$ of the domain of $u$ such that $\dot \Sigma \to \dot \Sigma'$ is an $m$-fold (unbranched) covering with $m \ge 2$. Then $m = \chi(\dot \Sigma) = 2 - 2 g(\dot \Sigma) - m$, so $g(\dot \Sigma) = 1 - m < 0$, which is impossible. It follows that every component of the domain of $u$ maps diffeomorphically onto $\dot \Sigma'$.

Conversely, suppose that $u \colon \dot \Sigma \to \widehat X$ is such a cover of a plane $v \colon \dot \Sigma' \to \widehat X$ with a positive end at a hyperbolic orbit $\gamma$ and such that $\ind(v) = I(v) = 0$. As above, we choose the trivialization $\tau$ of $\gamma^*\xi$ such that $c_1(v^*\xi, \tau) = 0$. Then $0 = \ind(v) = \mu_\tau(\gamma) - 1$, so $\mu_\tau(\gamma) = 1$. Thus, $0 = I(v) = Q_\tau(v) + 1$, so $Q_\tau(v) = -1$. The relative self-intersection number $Q_\tau$ is quadratic under taking multiple covers (see the discussion in \cite[Section 3.5]{H2}), so $Q_\tau(u) = -d^2$ and
\begin{equation*}
	I(u) = -d^2 + \sum_{i = 1}^d i = - \binom{d}{2},
\end{equation*}
as desired.
\end{proof}

\begin{defn}
\label{defn:degenerate-cover}
	We refer to an unbranched, disconnected, negative-index cover of a plane as in \Cref{lemma:neg-index-curves} as a \textbf{degenerate cover} of said plane.
\end{defn}

\begin{lemma}
\label{lemma:index-zero-multiple-covers}
Let $\bm \gamma$ be an orbit set with $\mathcal A(\bm \beta) < \mathcal A(\bm \gamma)$. If $u \in \mathcal M_X^{0, 0}(\bm \gamma, \bm \beta)$ is multiply covered, then $u$ is an immersion and the underlying somewhere injective curve is a $J$-holomorphic cylinder with ECH index $0$ and no negative ends.
\end{lemma}

\begin{proof}
Assume first that $u$ is a $d$-fold cover, $d \ge 2$, of a somewhere injective curve $v \colon \dot \Sigma' \to \widehat X$. Since $\ind(u) = 0$, \Cref{lemma:supersimple-fredholm-index} implies that $u$ is necessarily an unbranched cover of $v$. Since $\bm \beta$ is an ECH generator, it follows immediately that $u$ has no negative ends. Since $I(u) = 0$, the inequality \eqref{eq:multiple-cover-inequality} implies that $2g(\dot \Sigma') - 2 + h(v) \le 0$. Thus, $h(v) = 1$ or $2$. If $h(v) = 1$, then $v$ is a plane and, by the arguments in the proof of \Cref{lemma:neg-index-curves}, either $I(u) > 0$ or $I(u) < 0$. It follows that $h(v) = 2$, $v$ is a cylinder, and $I(v) = 0$. Thus, $v$ is embedded, so $u$ is an immersion. Clearly $v$ has no negative ends.
\end{proof}

%==========================%
% SECTION 5.2			        %
% CANCELING DEGENERATIONS %
%==========================%

\subsection{Canceling degenerations}
\label{subsec:canceling-degenerations}

Now we prove a sequence of lemmas that eliminates various cases in our analysis of $\bdry \overline{\mathcal M}_X^{1, 1}(\bm \alpha, \bm \beta)$ by showing that certain types of buildings occur in canceling pairs. Given a two-level building in the boundary, we say that the negative orbit set of the top level is the \textbf{intermediate orbit set} of the building.

\begin{lemma}
\label{lemma:symp-level-somewhere-inj}
	The symplectization level of a building in $\bdry \overline{\mathcal M}_X^{1, 1}(\bm \alpha, \bm \beta)$ is somewhere injective.
\end{lemma}

\begin{proof}
Without loss of generality, assume that the building is of the form $u_0 \cup [u_1]$. If $[u_1]$ is multiply covered, \Cref{lemma:supersimple-fredholm-index} implies that it is a branched cover of a trivial cylinder in $(\R \times Y_+) / \R$, contradicting the assumption that its positive orbit set $\bm \alpha$ is a generator of the ECH chain complex for $(Y_+, \lambda_+)$.
\end{proof}

\begin{lemma}
\label{lemma:no-neg-on-top}
	The top level of a building in $\bdry \overline{\mathcal M}_X^{1, 1}(\bm \alpha, \bm \beta)$ has non-negative ECH index.
\end{lemma}

\begin{proof}
First assume that the building is of the form $u_0 \cup [u_1]$, so that $\ind(u_1) = 1$. Then $[u_1]$ is somewhere injective by \Cref{lemma:symp-level-somewhere-inj}, so $I(u_1) \ge 1$ by \eqref{eqn:hut-ind-ineq}.

Now assume that the building is of the form $[u_{-1}] \cup u_0$, so that $\ind(u_0) = 0$. If $I(u_0) < 0$, then $u_0$ must contain a degenerate cover of a plane by \Cref{lemma:neg-index-curves}. The underlying embedded plane cannot have a negative end since $\widehat X$ is exact; see the proof of \cite[Lemma 3.4.2]{BH1}. Hence, $\bm \alpha$ must contain a Reeb orbit with multiplicity greater than $1$, which contradicts the assumption that it is a generator of $ECC(Y_+, \lambda_+, J_+)$.
\end{proof}

\begin{lemma}
\label{lemma:nondegen-endpoint-class}
	The count of buildings in $\bdry \overline{\mathcal M}_X^{1, 1}(\bm \alpha, \bm \beta)$ where the bottom level has non-negative ECH index and such that the intermediate orbit set $\bm \gamma$ has at least one orbit of multiplicity greater than $1$ is even.
\end{lemma}

\begin{proof}
First assume that the building is of the form $u_0 \cup [u_1]$. Then $[u_1]$ is somewhere injective by \Cref{lemma:symp-level-somewhere-inj} and $I(u_1) \ge 1$ by the proof of \Cref{lemma:no-neg-on-top}. Since $I(u_0) \ge 0$ and $I(u_0) + I(u_1) = 1$, we see that in fact $I(u_1) = 1$ and $I(u_0) = 0$. By \Cref{cor:ind-equal}, $[u_1]$ satisfies the ECH partition conditions, and hence so does $u_0$ since its negative orbit set $\bm \beta$ is a generator of $ECC(Y_-, \lambda_-, J_-)$. But then the multiply covered components of $u_0$ are unbranched covers of cylinders with no negative ends by \Cref{lemma:index-zero-multiple-covers}, and the count of such buildings is even.

Now assume that the building is of the form $[u_{-1}] \cup u_0$. Then $[u_{-1}]$ is somewhere injective by \Cref{lemma:symp-level-somewhere-inj}, so by \eqref{eqn:hut-ind-ineq}, $I(u_{-1}) \ge \ind(u_{-1}) = 1$. Since $I(u_0) \ge 0$ by \Cref{lemma:no-neg-on-top}, the same argument as above implies that $I(u_{-1}) = \ind(u_{-1}) = 1$ and $I(u_0) = 0$. By \Cref{cor:ind-equal}, $[u_{-1}]$ satisfies the ECH partition conditions. If $u_0$ is multiply covered, then its multiply covered components are unbranched covers of cylinders with no negative ends by \Cref{lemma:index-zero-multiple-covers}. But then $\bm \alpha$ cannot be a generator of $ECC(Y_+, \lambda_+, J_+)$, and we have reached a contradiction. Hence $u_0$ is also somewhere injective. Since $\bm \gamma$ contains a hyperbolic orbit with multiplicity greater than $1$, $u_0$ must either have multiple negative ends with multiplicity $1$ asymptotic to the same positive hyperbolic orbit or at least one negative end with multiplicity $2$ asymptotic to a double cover of a negative hyperbolic orbit. In either case, the count of such buildings is even.
\end{proof}

\begin{proof}[Proof of \Cref{thm:degen-class-thm}]
	Note that the building must be of the form $u_0 \cup [u_1]$. Let $\bm \gamma$ denote the intermediate orbit set, and let $n_\gamma$ denote the multiplicity of the orbit $\gamma$ in $\bm \gamma$. By \Cref{lemma:neg-index-curves}, $u_0$ must contain a multiply covered component that is a degenerate cover of a plane. Let $\widetilde{\Gamma^+(u_0)}$ denote the set of orbits $\gamma$ in $\bm \gamma$ such that $u_0$ contains a degenerate cover of a plane whose positive end is at $\gamma$. For each $\gamma \in \widetilde{\Gamma^+(u_0)}$, let $m_\gamma$ denote the multiplicity of the covering of the plane with its positive end at $\gamma$. By \Cref{thm:gen-ind-ineq},
	\begin{equation}
	\label{eqn:top-ech-ind}
		I(u_1)
			\ge
				1 + \sum_{ \gamma \in \widetilde{\Gamma^+(u_0)} } \binom{m_\gamma}{2}
	\end{equation}
	and
	\begin{equation}
	\label{eqn:bottom-ech-ind}
		I(u_0)
			\ge
				- \sum_{ \gamma \in \widetilde{\Gamma^+(u_0)} } \binom{m_\gamma}{2}.
	\end{equation}
	Since $I(u_0) + I(u_1) = 1$, both inequalities must in fact be equalities. Thus, $u_0$ satisfies the ECH partition conditions except for degenerate covers of planes at orbits in $\widetilde{\Gamma^+(u_0)}$.

	We claim that buildings where $u_0$ contains other multiply covered components occur in canceling pairs. Any multiple covers besides the degenerate ones have non-negative ECH index. Covers with ECH index $0$ are unbranched covers of cylinders with no negative ends satisfying the partition conditions, and buildings containing such curves occur in canceling pairs. There are no multiply covered components of $u_0$ with positive ECH index, as then the inequality \eqref{eqn:bottom-ech-ind} is strict and $I(u_0) + I(u_1) > 1$.

	We now claim that buildings where there exists a $\gamma \in \widetilde{\Gamma^+(u_0)}$ with $m_\gamma < n_\gamma$ occur in canceling pairs. So assume that such a $\gamma$ exists. If $u_0$ has a non-planar component with a positive end asymptotic to $\gamma^k$ for $k$ odd or $k \ge 4$ even, then \eqref{eqn:top-ech-ind} is a strict inequality and $I(u_0) + I(u_1) > 1$. The buildings where $u_0$ has a non-planar component with a positive end asymptotic to $\gamma^2$ occur in canceling pairs.

	Finally, we claim that every Reeb orbit $\gamma$ in $\Gamma^+(u_0) \setminus \widetilde{\Gamma^+(u_0)}$ has multiplicity $1$ or else the building is part of a canceling pair. So let $u_0 \cup [u_1]$ be a building such that some $\gamma$ in $\Gamma^+(u_0) \setminus \widetilde{\Gamma^+(u_0)}$ has multiplicity greater than $1$. By \Cref{thm:gen-ind-ineq}, the negative ends of $[u_1]$ at covers of $\gamma$ satisfy the ECH partition conditions. If $\gamma$ is positive hyperbolic, there are at least two negative ends of $[u_1]$ of multiplicity $1$ at $\gamma$, and such buildings occur in canceling pairs. If $\gamma$ is negative hyperbolic, there is at least one negative end of $[u_1]$ at $\gamma^2$, and such buildings again occur in canceling pairs.
\end{proof}

%==============================%
% SECTION 6						%
% OBSTRUCTION BUNDLE GLUING	%
%==============================%

\section{Obstruction Bundle Gluing}
\label{sec:obs-bund-glue}

In this section, we set up the gluing machinery in preparation for the proof of \Cref{thm:gluing-top-piece} in \Cref{sec:models}. We first review the prototypical gluing problem from \Cref{sec:intro}. Recall that $(Y_+, \lambda_+)$ is a smooth $3$-manifold with an $L$-supersimple contact form and $u_1 \colon \dot \Sigma \to \R \times Y_+$ is an embedded $J$-holomorphic curve with Fredholm index $1$ such that
\begin{enumerate}
	\item
	the positive orbit set of $u_1$ is an ECH generator $\bm \alpha$ with $\mathcal A(\bm \alpha) < L$;
	
	\item
	the negative ends of $u_1$ are asymptotic to an orbit set $\bm \beta$ in which each Reeb orbit has multiplicity $1$ except for a single negative hyperbolic orbit $\beta_0$;
	
	\item
	$u_1$ has $n$ negative ends at $\beta_0$, each with multiplicity $1$;
	
	\item
	$I(u_1) = 1 + \binom{n}{2}$.
\end{enumerate}
Recall that, for each $n \ge 3$, we set $\mathcal M_n = \mathcal M(1, 1, \ldots, 1 \, | \, 1, 1, \ldots, 1, 3)$, where there are $n$ positive ends of multiplicity $1$, $n - 3$ negative ends of multiplicity $1$, and one negative end of multiplicity $3$. We wish to glue branched covers in $\mathcal M_n$ to the curve $u_1$ above. Note that each branched cover in $\mathcal M_n$ has total branching index $2n - 4$.

\begin{prop} \cite[Proposition 5.2.2]{R}
\label{prop:coker-dim}
	Let $(Y, \lambda)$ be a non-degenerate contact $3$-manifold and let $J$ be a generic $\R$-invariant almost complex structure on $\R \times Y$. Let $\alpha$ be a negative hyperbolic Reeb orbit of $\lambda$ and let $u$ be a branched cover of the trivial cylinder $\R \times \alpha$ in $\R \times Y$. If $u$ has $k$ branch points, counted with multiplicity, then $\ind(u) = k$ and $\dim \Coker D_u^N = k$. In particular, the obstruction bundle $\mathcal O \to [R, \infty) \times (\mathcal M_n / \R)$ has rank $2n - 4$.
\end{prop}

\begin{proof}
The computation of $\ind(u)$ follows immediately from Lemma \ref{lemma:supersimple-fredholm-index}. From \cite[Theorem 3]{W}, we know that $\dim \Ker D_u^N = \dim \Ker D\ol\bdry_J - 2k = 0$. From the computation immediately preceding that theorem, we also know that $\ind(D_u^N) = \ind(u) - 2k = -k$, so $\dim \Coker D_u^N = k$.
\end{proof}

\begin{notation}
\label{notation:products}
	For any two subsets $\{ p_1, \ldots, p_n \}$ and $\{ q_1, \ldots, q_{n - 2} \}$ of $\C$, where the $p_i$ and $q_j$ are pairwise distinct, we set
	\begin{alignat*}{5}
		A(z)	&=	\prod_{i = 1}^n (z - p_i),	&&	\quad	\mathbb A(z)	&&=	\prod_{i = 1}^{n - 2} (z - p_i),	&&	\quad	B(z)	&&=	\prod_{i = 2}^{n - 2} (z - q_i),	\\
		A_k(z)	&=	\prod_{ \substack{ i = 1 \\ i \neq k } }^n (z - p_i),	&&	\quad	\mathbb A_k(z)	&&=	\prod_{ \substack{ i = 1 \\ i \neq k } }^{n - 2} (z - p_i),	&&	\quad	B_k(z)	&&=	\prod_{ \substack{ i = 2 \\ i \neq k } }^{n - 2} (z - q_i).
	\end{alignat*}
	Note that we suppress the dependence on $n$ for the functions considered above.
\end{notation}

%==============================%
% SECTION 6.1				        %
% MODULI SPACE PARAMETRIZATION %
%==============================%

\subsection{Parametrization of the moduli space}
\label{subsec:moduli-space-parametrization}

We parametrize the reduced moduli space $\mathcal M_n / \R$ by choosing a smooth section of the bundle $\mathcal M_n \to \mathcal M_n / \R$ in the following way. Curves in $\mathcal M_n$ have genus $0$, so the domain for each map is a punctured Riemann sphere $\widehat{\C} \setminus \left( P^+ \cup P^- \right)$, where $P^+ = \{ p_1, p_2, \ldots, p_n \}$ and $P^- = \{ q_1, q_2, \ldots, q_{n - 2} \}$ are the (disjoint) sets of positive and negative punctures and $q_1$ is the multiplicity $3$ negative puncture. View $\widehat{ \C }$ as $\C \cup \{\infty\}$, fix the positive punctures $p_{n - 1}$ and $p_n$ in $\C$, and fix the negative puncture $q_1$ to be the point at infinity. The other punctures $p_1, \ldots, p_{n - 2}, q_2, \ldots, q_{n - 2}$ are free to move in $\C$. Then the data consisting of the punctures $p_1, \ldots, p_{n - 2}, q_2, \ldots, q_{n - 2} \in \C$ and $\theta \in \R / 6 \pi \Z$ are sent to the map
\begin{gather*}
	u \colon \C \setminus \left( P^+ \cup P^- \right)	\to 			\C^*		\\
									z	\mapsto 		e^{i \theta} \frac{ B(z) }{ A(z) }.
\end{gather*}
Roughly speaking, changing the parameter $\theta$ simultaneously rotates the branch points of $u$ in the $S^1$-factor of the image cylinder.

The asymptotic marker $\tau_i \in S^1$ at each puncture is determined as follows. For each positive puncture $p_i$, there is an $\epsilon > 0$ and a complex-valued function $f(t)$, $0 < t < \epsilon$, such that $\displaystyle\lim_{t \to 0^+} f(t) = 0$ and $u( p_i + f(t) ) = e^{1/t}$. Then
\begin{equation}
\label{eqn:tau_i}
	\tau_i
		=
			\lim_{t \to 0^+} \frac{ f(t) }{ |f(t)| }
		=
			e^{i \theta} \frac{ B(p_i) }{ A_i(p_i) } \left| \frac{ B(p_i) }{ A_i(p_i) } \right|^{-1}.
\end{equation} 
For each negative puncture $q_j$, $j = 2, \ldots, n - 2$, there is an $\epsilon > 0$ and a complex-valued function $f(t)$, $0 < t < \epsilon$, such that $f(t) \to 0$ as $t \to 0^+$ and $u(q_j + f(t)) = e^{- 1 / t}$. Then
\begin{equation*}
	\tau_{-j}
		=
			\lim_{t \to 0^+} \frac{ f(t) }{ |f(t)| }
		=
			e^{-i \theta} \frac{ A(q_j) }{ B_j(q_j) } \left| \frac{ A(q_j) }{ B_j(q_j) } \right|^{-1}.
\end{equation*}
For $q_1$, there is an $\epsilon > 0$ and a complex-valued function $f(t)$, $0 < t < \epsilon$, such that $f(t) \to 0$ as $t \to 0^+$ and $u( f(t)^{-1} ) = e^{-1/t}$. Then
\begin{equation*}
	\tau_{-1}^3
		=
			\lim_{t \to 0^+} \frac{ f(t) }{ |f(t)| }
		=
			e^{-i \theta}.
\end{equation*}

%======================%
% SECTION 6.2			%
% OBSTRUCTION SECTIONS %
%======================%

\subsection{The obstruction sections}
\label{subsec:obstruction-sections}

Recall that Hutchings-Taubes define a \textbf{linearized obstruction section} that is homotopic to the full obstruction section and whose zero set is much easier to compute \cite{HT1,HT2}. We define a $\Z_+$-indexed family of sections $\mathfrak s_m$, all homotopic to each other and to $\mathfrak s$, such that $\mathfrak s_1$ is the analogue of the Hutchings-Taubes linearized section in our setting.\footnote{Hutchings-Taubes use $\mathfrak s_0$ to denote the linearized obstruction section. However, since we use a $\Z_+$-indexed family of sections, it makes more sense for us to denote the linearized section by $\mathfrak s_1$.} In \Cref{sec:models}, we show that the count of zeros of $\mathfrak s$ and $\mathfrak s_1$ are the same.

Te define $\mathfrak s_m$, let $m \in \Z_+$ and assume that the positive ends of $u$ and the negative ends of $u_+$ are labeled so that the $i^\text{th}$ positive end of $[u] \in \mathcal M_n / \R$ matches up with the $i^\text{th}$ negative end of $u_1$. We first restrict our attention to the $i^\text{th}$ positive end of $u$. Consider the asymptotic expansion of $u_1$ over its $i^\text{th}$ negative end, written in cylindrical coordinates, and let $\Pi_{i, m} u_1$ denote its projection onto the $m$ leading eigenspaces of the asymptotic operator $A_{\beta_0}$ from \Cref{subsec:asymp-operator}. Let $\sigma \in \Coker(D_u^N)$ and let $\sigma_i$ denote the restriction of $\sigma$ to the $i^\text{th}$ positive end of $u$, written in cylindrical coordinates. Then set
\begin{equation*}
	\mathfrak s_m(T, u)(\sigma)
		=
			\sum_{i = 1}^n \langle \Pi_{i, m} u_1, \sigma_i(T, \cdot \,) \rangle.
\end{equation*}

\begin{notation}
	We denote the zero set $\mathfrak s_m^{-1}(0)$ by $\mathcal Z_m$. We denote the zero set $\mathfrak s^{-1}(0)$ of the full obstruction section by $\mathcal Z$.
\end{notation}

%================%
% SECTION 6.3	   %
% COKERNEL BASIS %
%================%

\subsection{A basis for the cokernel}
\label{subsec:cokernel-basis}

We now choose a convenient basis for the space $\Coker(D_u^N)$, which we identify with $\Ker(D_u^N)^*$. If $\sigma \in \Coker(D_u^N)$ and $\tau$ is a trivialization of $\xi$ over $\beta_0$, let $\rho_\tau(\sigma_i)$ denote the \textbf{asymptotic winding number} of $\sigma$ restricted to the $i^\text{th}$ positive end of $u$ in the trivialization $\tau$, defined as follows. On the $i^\text{th}$ positive end, write $\sigma = \sigma_i \otimes ( ds - i dt )$ in cylindrical coordinates. Then $\rho_\tau(\sigma_i)$ is defined as the winding number of the leading asymptotic eigenfunction in the series expansion of $\sigma_i$. Recall from \cite[Section 3]{HWZ} that, for each positive end of $u$, we have $2 \rho_\tau(\sigma_i) \ge \mu_\tau(\beta_0)$ and for each negative end, we have $2 \rho_\tau(\sigma_i) \le \mu_\tau(\beta_0)$.

\begin{lemma}
\label{lemma:coker-elt-zeros}
	If $u \in \mathcal M_n$, where $n \ge 3$, and $\sigma \in \Coker(D_u^N)$, then $\left| \# \sigma^{-1}(0) \right| \le n - 3$, where the zeros of $\sigma$ are counted with multiplicities.
\end{lemma}

\begin{proof}
Note that every zero of $\sigma$ has negative multiplicity and that
\begin{equation*}
	\chi(\dot \Sigma)
		=
			4 - 2n.
\end{equation*}
On the ends of $u$, write
\begin{align*}
	\rho_\tau(\sigma_i)
		&= \left\lceil \frac{ \mu_\tau(\beta_0) }{ 2 } \right\rceil + k_i
			\quad \text{for $i > 0$,} \\
	\rho_\tau(\sigma_j)
		&= \left\lfloor \frac{\mu_\tau(\beta_0) }{ 2 } \right\rfloor - k_j
			\quad \text{for $j = -2, \ldots, -(n - 2)$, and} \\
	\rho_\tau(\sigma_{-1})
		&= \left\lfloor \frac{ \mu_\tau(\beta_0^3) }{ 2 } \right\rfloor - k_{-1}.
\end{align*}
Then, choosing $\tau$ so that $\mu_\tau(\beta_0) = 1$, we have
\begin{align*}
	0
		&\ge
			\# \sigma^{-1}(0)	\\
		&=	\chi(\dot \Sigma)
				+ \sum_{i = 1}^n \rho_\tau(\sigma_i)
				- \sum_{j = 1}^{n - 2} \rho_\tau(\sigma_{-j})	\\
		&=
			3 - n
				+ \sum_{i = 1}^n k_i
				+ \sum_{j = 1}^{n - 2} k_{-j}	\\
		&\ge
			3 - n,
\end{align*}
as claimed.
\end{proof}

\begin{rmk}
\label{rmk:winding-degen-relation}
	The proof of \Cref{lemma:coker-elt-zeros} also shows that a cokernel element $\sigma$ cannot be too degenerate at the ends. More precisely, we have
	\begin{equation}
	\label{eqn:coker-zero-degen-bound}
		0
			\le
				\left| \# \sigma^{-1}(0) \right|
					+ \sum_{i = 1}^n k_i
					+ \sum_{j = 1}^{n - 2} k_{-j}
			\le
				n - 3.
	\end{equation}
\end{rmk}

\begin{prop}
\label{prop:basis-construction}
	There exists a basis $\sigma^1, \sigma^2, \ldots, \sigma^{2n - 5}, \sigma^{2n - 4}$ for $\Ker(D_u^N)^*$ such that, for each $i = 1, 2, \ldots, n - 2$, the projection of $\{ \sigma^{2i - 1}, \sigma^{2i} \}$ to the leading eigenspace on the $j^\text{th}$ positive end of $u$ is a basis for that eigenspace if $j = i$, $n - 1$, or $n$ and vanishes otherwise.
\end{prop}

\begin{proof}
Let $\sigma^1, \sigma^2, \ldots, \sigma^{2n - 5}, \sigma^{2n - 4}$ be a basis for $\Ker(D_u^N)^*$. We give an algorithm to converting this basis into one with the desired properties.

First, note that there must be a pair of basis elements whose projections to the leading eigenspace on the positive end labeled $1$ are linearly independent. For if not, then row reduction yields a cokernel element $\sigma$ with $k_1 \ge n - 2$, which contradicts \Cref{rmk:winding-degen-relation}. After possibly relabeling the elements of the basis, we may assume that $\sigma^1$ and $\sigma^2$ are the above basis elements. By subtracting appropriate multiples of $\sigma^1$ and $\sigma^2$ from the other basis elements, we may assume that $k_1 \ge 1$ for each $\sigma^i$ with $i \neq 1, 2$.

Assume that the elements $\sigma^1, \sigma^2, \ldots, \sigma^{2\ell - 1}, \sigma^{2\ell}$ are such that, for $i = 1, 2, \ldots, \ell$, the projection of $\{ \sigma^{2i - 1}, \sigma^{2i} \}$ to the leading eigenspace on the $j^\text{th}$ positive end of $u$, is a basis for that eigenspace if $j = i$ and vanishes if $1 \le j \le \ell$ and $j \neq i$. Assume also that, for each $\sigma^i$ with $i = 2\ell + 1, 2\ell + 2, \ldots, 2n - 5, 2n - 4$, we have $k_j \ge 1$ for $j = 1, 2, \ldots, \ell$. There must be a pair of vectors among $\sigma^{2\ell + 1}, \sigma^{2\ell + 2}, \ldots, \sigma^{2n - 5}, \sigma^{2n - 4}$ whose projections to the leading eigenspace on the positive end labeled $\ell + 1$ are linearly independent. For if not, row reduction yields a cokernel element $\sigma$ with $\sum_{j = 1}^{\ell + 1} k_j \ge n - 2$, which contradicts \Cref{rmk:winding-degen-relation}. After possibly relabeling $\sigma^{2\ell + 1}, \sigma^{2\ell + 2}, \ldots, \sigma^{2n - 5}, \sigma^{2n - 4}$, we may assume that $\sigma^{2\ell + 1}$ and $\sigma^{2\ell + 2}$ are the above basis elements. By subtracting appropriate multiples of $\sigma^{2\ell + 1}$ and $\sigma^{2\ell + 2}$ from $\sigma^{2\ell + 3}, \sigma^{2\ell + 4}, \ldots, \sigma^{2n - 5}, \sigma^{2n - 4}$, we may assume that $k_{\ell + 1} \ge 1$ for each $\sigma^i$ with $i \neq 2\ell + 1, 2 \ell + 2$.

After step $n - 2$ of this algorithm, we arrive at our desired basis.
\end{proof}

%===================================%
% SECTION 6.4						 %
% ASYMPTOTIC OPERATOR DEFORMATION %
%===================================%

\subsection{Deformation of the asymptotic operator}
\label{subsec:asymptotic-operator-deformation}

To make our calculations easier, we now replace the elements of $\Coker(D_u^n)$, which we identify with $\Ker(D_u^N)^*$, with meromorphic $(0, 1)$-forms by perturbing the asymptotic operator $A_{\beta_0}$ for the orbit $\beta_0$. First, define a homotopy of the asymptotic operator by
\begin{equation*}
	A_{\beta_0, \nu}
		=
			- j_0 \frac{ \bdry }{ \bdry t }
%				-
%				\begin{pmatrix}
%					\frac{1}{2}		&	0		\\
%					0			&	\frac{1}{2}
%				\end{pmatrix}
				-
				(1 - \nu)
				\begin{pmatrix}
					0		&	\epsilon	\\
					\epsilon	&	0
				\end{pmatrix},
\end{equation*}
where $\nu \in [0, 1]$. If we use a coordinate system in a neighborhood of $\beta_0$ where we identify $(2\pi, x, y) \sim (0, x, y)$, then $A_{\beta_0, 1}$ is equivalent to the operator $- i \frac{\bdry}{\bdry t} - \frac{1}{2}$. This operator is non-degenerate, as is its pullback to odd covers of $\beta_0$; however, its pullback to an even cover of $\beta_0$ is degenerate.
%We remark here that the coordinate system we use in a neighborhood of $\beta_0$ differs from the one given in \Cref{thm:eliminate-elliptics} in that we identify $(2\pi, x, y) \sim (0, x, y)$.
%When $k$ is odd, the operator is non-degenerate throughout the homotopy. However, when $k$ is even, the operator is non-degenerate when $0 \le \nu < 1$ and singular when $\nu = 1$.
More precisely, let $\lambda_{+, \nu}$ denote the smallest positive eigenvalue of the pullback $A_{\beta_0^{2k}, \nu}$ and $\lambda_{-, \nu}$ the largest negative eigenvalue. Then both $\lambda_{+, \nu}$ and $\lambda_{-, \nu}$ monotonically converge to $0$ as $\nu \to 1$.

We correct for the degeneration for even covers of $\beta_0$ by putting asymptotic weights $\bm \delta_\nu = (\delta_\nu, \ldots, \delta_\nu)$ on our Sobolev spaces for $(D_u^N)^*$, where $\delta_\nu = (1 - \nu) \lambda_{-, 0} + \nu \delta$ and $\delta$ is a sufficiently small positive real number that depends on $n$. When $\nu = 1$, the operator $(D_u^N)^*$ is complex-linear, and the elements of $\Ker(D_u^N)^*$ can be written as $\sigma_i(s, t) \otimes (ds - i dt)$ in cylindrical coordinates over the $i^\text{th}$ positive end, where $\sigma_i(s, t)$ satisfies the equation
\begin{equation*}
	(\sigma_i)_s - i (\sigma_i)_t + \frac{1}{2} \sigma_i = 0.
\end{equation*}
If we set $\eta_i(s, t) = e^{- s / 2} \sigma_i(s, t)$ over such an end, we see that $\eta_i$ is anti-meromorphic in the usual sense. Finally, we single out the real $1$-dimensional subspace of the $0$-eigenspace of $A_{\beta_0^k, 1}$ that corresponds to the $\lambda_{+, 0}$-eigenspace of $A_{\beta_0^k, 0}$ by requiring that the leading eigenfunction in the asymptotic expansion of $\eta$ near an even-multiplicity end be a real scalar multiple of the vector in $\C$ representing the stable direction of $\beta_0^k$ for all $t$. We say that the leading eigenfunction \textbf{follows the stable direction} of $\beta_0^k$.

\begin{defn}
\label{defn:replacement}
	A meromorphic $(0, 1)$-form $\eta$ is a \textbf{replacement} for $\sigma \in \Ker(D_u^N)^*$ if, in cylindrical coordinates $(s, t)$ near each puncture, we have $\eta(s, t) = e^{- s / 2} \sigma(s, t)$.
\end{defn}

\begin{rmk}
\label{rmk:replacement}
	The point of using replacements instead of using elements of $\Ker(D_u^N)^*$ directly is that we can write down explicit expressions for replacements and hence explicit equations for the zero sets $\mathcal Z_m$ and $\mathcal Z$.
\end{rmk}

%====================%
% SECTION 6.5		   %
% THE GLUING PROBLEM %
%====================%

\subsection{The gluing problem}
\label{subsec:gluing-problem}

We now write down a collection of meromorphic $(0, 1)$-forms on $\dot \Sigma$ that are replacements, in the sense of \Cref{defn:replacement}, for the basis for $\Ker(D_u^N)^*$ from \Cref{prop:basis-construction}.

\begin{notation}
\label{notation:condense-replacements}
	Set
	\begin{equation*}
		Q_k(z)
		=
		\frac{ \mathbb A_k(z) }{ B(z) }
	\end{equation*}
	for $k = 1, \ldots, n - 2$,
	\begin{equation}
	\label{eqn:r_i}
		r_i
		=
		\frac{ B(p_i) }{ A_i(p_i) }
	\end{equation}
	for $i = 1, \ldots, n$,
	\begin{equation*}
		r_{-1}
		=
		1,
	\end{equation*}
	and
	\begin{equation*}
		r_{-j}
		=
		\frac{ A(q_j) }{ B_j(q_j) }
	\end{equation*}
	for $j = 2, \ldots, n - 2$. Note that $\tau_i |r_i| = e^{i \theta} r_i$ for $i > 0$, $\tau_{-j} |r_{-j}| = e^{-i \theta} r_{-j}$ for $j \ge 2$, and $\tau_{-1} |r_{-1}| = e^{-i \theta / 3} r_{-1}$.
\end{notation}

\begin{prop}
\label{prop:main-gluing-one-forms}
	The meromorphic $(0, 1)$-forms
	\begin{equation*}
		\eta_k(z)
			=
				\overline{ Q_k(z) } d \bar z
	\end{equation*}
	are replacements for a basis of $\Ker(D_u^N)^*$ as constructed in \Cref{prop:basis-construction}.
\end{prop}

\begin{proof}
Near $p_i$, we have $z = p_i + \tau_i e^{-\tilde s - it}$, where $(\tilde s, t) \in [R, \infty) \times (\R / 2 \pi \Z)$. Hence
\begin{align*}
	u(z)
		=
			e^{i \theta}
			e^{\tilde s + it}
			\frac{ B(p_i + e^{-\tilde s - it}) }{ A_i(p_i + e^{-\tilde s - it}) },
\end{align*}
so
\begin{equation*}
	\log| u(z) |
		=
			\tilde s
			+
			\log \left| \frac{ B(p_i + e^{-\tilde s - it}) }{ A_i(p_i + e^{-\tilde s - it}) } \right|.
\end{equation*}
Recall that we require $u(s, t) = (s, t, \widetilde u(s, t))$ in cylindrical coordinates. Thus, we must change our $\tilde s$-coordinate to
\begin{equation*}
	s
		=
			\tilde s
			+
			\log \left| \frac{ B(p_i + e^{-\tilde s - it}) }{ A_i(p_i + e^{-\tilde s - it}) } \right|.
\end{equation*}
If $\tilde s \gg 0$, we have
\begin{equation*}
	s
		\approx
			\tilde s + \log \left| \frac{ B(p_i) }{ A_i(p_i) } \right|
		=
			\tilde s + \log |r_i|,
\end{equation*}
and consequently $z \approx p_i + \tau_i |r_i| e^{-s - it} = p_i + r_i e^{-s - it}$ near $p_i$. A similar change must be made in cylindrical coordinates around the negative punctures $q_j$, $j = 2, \ldots, n - 2$.

Now fix a value of $k$. We claim that each $\eta_k$ has winding number $1$ at $p_k$, $p_{n - 1}$, and $p_n$, has winding number $2$ at all other $p_i$, has winding number $1$ at $q_1$, and has winding number $0$ at all other $q_j$.

If we change to cylindrical coordinates around $p_k$, we can write $z = p_k + \tau_k e^{-s - it}$, $(s, t) \in [R, \infty) \times (\R / 2 \pi \Z)$. Then the first term in the asymptotic expansion of $\eta_k$ in the coordinates $(\tilde s, t)$ is approximately
\begin{equation*}
	- e^{-i \theta} \overline{r_k Q_k(p_k)} e^{-\tilde s + it} \otimes (d\tilde s - i dt),
\end{equation*}
which has winding number $1$. Similarly, the winding number of $\eta_k$ it cylindrical coordinates on the positive ends at $p_{n - 1}$ and $p_n$ is also $1$. The first term in the asymptotic expansion of $\eta_k$ vanishes in cylindrical coordinates around $p_i$, $i \neq k, n-1, n$, and the winding number at each of those ends is $2$.

If we change to cylindrical coordinates around $q_j$, $j = 2, \ldots, n - 2$, we can write $z = q_j + \tau_{-j} e^{s + it}$, $(s, t) \in (-\infty, -R] \times (\R / 2 \pi \Z)$. Then the first term in the asymptotic expansion of $\eta_k$ in the coordinates $(\tilde s, t)$ is approximately
\begin{equation*}
	\frac{ \overline{\mathbb A_k(q_j)} }{ \overline{ B_j(q_j) } } \otimes (d\tilde s - i dt),
\end{equation*}
which has winding number $0$. For $q_1$, if we change coordinates to $\zeta = z^{-1}$, we see that
\begin{equation*}
	\eta_k = - \frac{ 1 }{ \overline{\zeta}^2 } \frac{ \overline{ \mathbb A_k(\zeta) } }{ \overline{ B(\zeta) } } d \bar \zeta.
\end{equation*}
Hence, if we change to cylindrical coordinates around $\zeta = 0$, we can write $z = \tau_{-1} e^{(s + it) / 3}$, $(s, t) \in (-\infty, -R] \times (\R / 6 \pi \Z)$. Then the first term in the asymptotic expansion of $\eta_k$ in the coordinates $(\tilde s, t)$ is approximately
\begin{equation*}
	- e^{-i \theta / 3} e^{(-\tilde s + it) / 3} \otimes (d \tilde s - i dt),
\end{equation*}
which has winding number $1$.
\end{proof}

\begin{notation}
\label{notation:derivative-matrix}
	For each $\ell \ge 1$, set
	\begin{equation}
	\label{eqn:derivative-matrix}
		B_\ell
		=
		\frac{1}{(\ell - 1)!}
		\begin{pmatrix}
			\frac{d^{\ell - 1} Q_1}{d z^{\ell - 1}}(p_1) & \cdots & \frac{d^{\ell - 1} Q_1}{d z^{\ell - 1}}(p_n) \\
			\vdots					& \ddots &	\vdots					\\
			\frac{d^{\ell - 1} Q_{n - 2}}{d z^{\ell - 1}}(p_1) & \cdots & \frac{d^{\ell - 1} Q_{n - 2}}{d z^{\ell - 1}}(p_n)
		\end{pmatrix}
		\;\, \text{and} \;\;
		\mathbf v_\ell
		=
		e^{i \ell \theta} e^{- \ell T}
		\begin{pmatrix}
		r_1^{\ell} \alpha_{1, \ell} \\
		\vdots \\
		r_n^{\ell} \alpha_{n, \ell}
		\end{pmatrix}.
	\end{equation}
\end{notation}

\begin{cor}
\label{cor:zero-set-eqns}
	If $R \gg 0$, the section $\mathfrak s_m$ on $[R, \infty) \times \mathcal (M_n / \R)$ is close to a section whose zero set is defined by
	\begin{equation}
	\label{eqn:order-m-zero-set-eqns}
		\sum_{\ell = 1}^m B_\ell \mathbf v_\ell = 0.
	\end{equation}
\end{cor}

\begin{proof}
We make the same change to the $s$-coordinate near a positive puncture as in \Cref{prop:main-gluing-one-forms}. Near $p_i$, we have
\begin{align*}
	\overline{ \eta_k(z) }
		&=
			\left[ \sum_{\ell = 1}^\infty \frac{1}{(\ell - 1) !} \frac{d^{\ell - 1} Q_k}{d z^{\ell - 1}}(p_i) (z - p_i)^{\ell - 1} \right] dz \\
		&=
			- \left[ \sum_{\ell = 1}^\infty \frac{1}{(\ell - 1) !} \frac{d^{\ell - 1} Q_k}{d z^{\ell - 1}}(p_i) \tau_i^\ell e^{- \ell( s + it ) } \right] \otimes (ds + i dt) \\
		&\approx
			- \left[ \sum_{\ell = 1}^\infty \frac{1}{(\ell - 1) !} \frac{d^{\ell - 1} Q_k}{d z^{\ell - 1}}(p_i) \tau_i^\ell e^{- \ell( \tilde s - \log |r_i| + it ) } \right] \otimes (d\tilde s + i dt) \\
		&=
			- \left[ \sum_{\ell = 1}^\infty \frac{1}{(\ell - 1) !} \frac{d^{\ell - 1} Q_k}{d z^{\ell - 1}}(p_i) \tau_i^\ell |r_i|^\ell e^{- \ell(\tilde s + it ) } \right] \otimes (d\tilde s + i dt) \\
		&=
			- \left[ \sum_{\ell = 1}^\infty \frac{1}{(\ell - 1) !} \frac{d^{\ell - 1} Q_k}{d z^{\ell - 1}}(p_i) e^{i \ell \theta} r_i^\ell e^{- \ell( \tilde s + it ) } \right] \otimes (d\tilde s + i dt)
\end{align*}

Since the obstruction bundle $\mathcal O$ is complex in this case, we can, following Hutchings-Taubes \cite{HT1} identify the section $\mathfrak s_m$ with a section $\mathfrak s_m^{\C}$ defined by
\begin{equation*}
	\mathfrak s_m^{\C}(T, u)(\sigma)
		=
			\mathfrak s_m(T, u)(\sigma)
			+
			i \mathfrak s_m(T, u)(-i \sigma).
\end{equation*}
As noted in \cite{HT1}, the definition of $\mathfrak s_m^{\C}$ is equivalent to the replacing the real inner products in the original definition with complex inner products. We identify $\mathfrak s_m$ with its complexification and compute
\begin{align*}
	\mathfrak s_m(T, [u])(\eta_k)
		&=
			\sum_{j = 1}^n \left\langle \sum_{\ell = 1}^m \alpha_{j, \ell} e^{i \ell t}, - \sum_{\ell = 1}^\infty \frac{1}{(\ell - 1)!} \overline{ \frac{d^{\ell - 1} Q_k}{d z^{\ell - 1}}(p_j) r_j^\ell } e^{- i \ell \theta} e^{ -\ell T } e^{ \ell it } \right\rangle \\
		&=
			- \sum_{j = 1}^n \sum_{\ell = 1}^m \frac{1}{(\ell - 1)!} \frac{ d^{\ell - 1} Q_k }{ dz^{\ell - 1} }(p_j) r_j^\ell e^{i \ell \theta} e^{-\ell T} \alpha_{j, \ell},
\end{align*}
which is the $k^\text{th}$ component of $- \sum_{\ell = 1}^m B_\ell \mathbf v_\ell$.
\end{proof}

\begin{rmk}
\label{rmk:partial-frac}
	It will be useful in \Cref{sec:model-existence-proof} to note that the partial fraction decomposition of $Q_i$ is
	\begin{equation}
	\label{eqn:partial-fractions}
		Q_i(z)
		=
		1 - \sum_{k = 2}^{n - 2} \frac{ \mathbb A_i(q_k) }{ B_k(q_k) } \frac{ 1 }{ q_k - z }.
	\end{equation}
\end{rmk}

\begin{cor}
\label{cor:top-punctures-fixed}
	The equations
	\begin{equation}
	\label{eqn:lin-zero-set-eqns}
		(p_{n - 1} - p_n) \alpha_{k, 1}
		-
		(p_k - p_n) \alpha_{n - 1, 1}
		+
		(p_k - p_{n - 1}) \alpha_{n, 1}
		=
		0,
	\end{equation}
	$k = 1, 2, \ldots, n - 2$, determine $\mathcal Z_1$. Moreover, if $(T, [u]) \in \mathcal Z_1$, then $p_1, p_2, \ldots, p_{n - 2}$ are determined by $p_{n - 1}$, $p_n$, and the coefficients $\alpha_{1, 1}, \alpha_{2, 1}, \ldots, \alpha_{n, 1}$.
\end{cor}

\begin{proof}
We use the notation $r_i$ from \Cref{notation:condense-replacements}. Note that $Q_k(p_i) = 0$ when $i \neq k$, $n - 1$, or $n$. Note also that, when $m = 1$, \eqref{eqn:order-m-zero-set-eqns} has an overall factor of $e^{-T} e^{i \theta}$. Thus, \eqref{eqn:order-m-zero-set-eqns} reduces to
\begin{align*}
	0
		&=
			Q_k(p_k) r_k \alpha_{k, 1} + Q_k(p_{n - 1}) r_{n - 1} \alpha_{n - 1, 1} + Q_k(p_n) r_n \alpha_{n, 1} \\
		&=
			\frac{ \alpha_{k, 1} }{ (p_k - p_{n - 1})(p_k - p_n) } + \frac{ \alpha_{n - 1, 1} }{ (p_{n -1} - p_k)(p_{n - 1} - p_n) } + \frac{ \alpha_{n, 1} }{ (p_n - p_k)(p_n - p_{n - 1}) },
\end{align*}
$k = 1, \ldots, n - 2$, which is equivalent to \eqref{eqn:lin-zero-set-eqns}.
\end{proof}

%=========================%
% SECTION 6.6			      %
% AUXILIARY GLUING PROBLEM %
%=========================%

\subsection{The auxiliary gluing problem}
\label{subsec:auxiliary-gluing-problem}

There is one case in \Cref{thm:degen-class-thm} that is not addressed by the prototypical gluing problem: the case where the curve in $\widehat{X}$ has a double cover of a plane, where we must glue a branched cover with two multiplicity $1$ positive ends and one multiplicity $2$ negative end. Accordingly, we now calculate the zero set of the obstruction section for the moduli space $\modsp{1, 1}{2}$. This calculation is slightly different, due to the presence of the multiplicity $2$ negative end. To begin, let $\mathcal M$ denote the moduli space $\modsp{1, 1}{2}$. Any $u \in \mathcal M$ has one simple branch point, and we can and do make the identification $\mathcal M = \R / 4 \pi \Z$. Note that $\ind(u) = \dim \Coker(D_u^N) = 1$.

The obstruction bundle $\mathcal O$ has rank $1$. Choose a trivialization $\tau$ of $\xi$ over $\beta_0$ so that $\mu_\tau(\beta_0) = 1$, as before. Any element $\sigma \in \Ker(D_u^N)^*$ satisfies
\begin{equation*}
	0	\ge 		\# \sigma^{-1}(0)															
		=		\chi(\dot \Sigma) + \rho_\tau(\sigma_1) + \rho_\tau(\sigma_2) - \rho_\tau(\sigma_{-1})
		\ge 		0.
\end{equation*}
Thus, every non-zero element of $\Ker(D_u^N)^*$ is non-vanishing.

We parametrize $\mathcal M$ in the following way. Fix the positive puncture $p_1 = 1$, set $p_2 = - p_1 - -1$, and let the negative puncture lie at infinity. Then send $\theta \in \R / 4 \pi \Z$ to
\begin{gather*}
	u_\theta \colon \C \setminus \{ \pm p, 0 \}			\to 			\C^*		\\
										z	\mapsto 		\frac{ e^{i \theta} }{ z^2 - 1 }.
\end{gather*}
The markers at the positive ends are given by $\tau_1 = e^{i \theta}$ and $\tau_2 = - e^{i \theta}$, while the marker at the negative end is determined by $\tau_{-1}^2 = e^{-i \theta}$. The meromorphic $(0, 1)$-form
\begin{equation*}
	\eta_\theta(z)
		=
			e^{-i \theta / 2} d \bar z
\end{equation*}
is a replacement for a spanning element of $\Ker(D_{u_\theta}^N)^*$. In particular, it follows the stable direction at the negative end.

To compute the zero set $\mathcal Z_1$, note that, up to a real scalar multiple, we have
\begin{equation*}
	\mathfrak s_1(u_\theta)(\eta_\theta)
		=
			\left\langle \alpha_1,  - e^{i \theta / 2} \right\rangle
			+
			\left\langle \alpha_2,  e^{i \theta / 2} \right\rangle
		=
			\left\langle \alpha_2 - \alpha_1, e^{- i \theta / 2} \right\rangle.
\end{equation*}
Thus, there are two values of $\theta \in \R / 4 \pi \Z$ that such that $\mathfrak s_1(u_\theta)(\eta_\theta) = 0$.

The branched covers corresponding to these two values of $\theta$ differ only in the choice of asymptotic marker at the negative end. Thus, the two curves we obtain by gluing also differ only in the choice of asymptotic marker at the multiplicity $2$ negative end in question. Moduli spaces for ECH consist of \textit{holomorphic currents} that are not asymptotically marked, so we have over-counted by a factor of $2$.

%============
% SECTION 6.7
% NON-GLUING
%============

\subsection{Non-Gluing Results}
\label{subsec:non-gluing}

We now show that the linearized obstruction section for certain branched covers of $\R \times \beta_0$ never has zeros; these results are used in the proof of \Cref{lemma:sequences-in-pre-image}. We assume that the coefficients $\alpha_{j, 1}$ used in the definition of the linearized section are all distinct and non-zero.

\begin{notation}
\label{notation:non-gluing}
	Let $n \ge 2$ and let $a_1, \ldots, a_k$ be any positive integers that sum to $n$. Let $\mathcal M^{g, n, 1}(a_1, \ldots, a_k)$ be the moduli space of genus $g \ge 0$ branched covers of $\R \times \beta_0$ with $k$ positive ends with multiplicities $a_1, \ldots, a_k$ and $n$ negative ends all with multiplicity $1$. Let $\mathcal M^{g, n, 2}(a_1, \ldots, a_k)$ be the moduli space of genus $g \ge 0$ branched covers of $\R \times \beta_0$ with $k$ positive ends with multiplicities $a_1, \ldots, a_k$, one negative end with multiplicity $3$, and $n - 3$ negative ends all with multiplicity $1$. In either case, we let $\ell$ be the number of $a_j$ that are odd and order the positive ends so that $a_1, \ldots, a_\ell$ are odd and $a_{\ell + 1}, \ldots, a_k$ are even.
\end{notation}

\begin{prop}
\label{prop:non-gluing}
	The linearized obstruction sections over $\mathcal M^{g, n, 1}(a_1, \ldots, a_k)$ and $\mathcal M^{g, n, 2}(a_1, \ldots, a_k)$ have no zeros for any $n \ge 2$, any $g \ge 0$, and any positive integers $a_1, \ldots, a_k$ that sum to $n$ and such that $\ell < n$.
\end{prop}

We prove \Cref{prop:non-gluing} by exhibiting, for each relevant branched cover $u$, an element $\sigma \in \Ker(D_u^N)^*$ such that $\mathfrak s_0(u)(\sigma) \neq 0$. We begin by describing the space of replacements for such $\sigma$.

\begin{lemma}
\label{lemma:replacement-iso-1}
	Let $u \colon \dot \Sigma \to \R \times Y_+$ be a branched cover in $\mathcal M^{g, n, 1}(a_1, \ldots, a_k)$ with $\ell < k$, where $\ell$ is as in \Cref{notation:non-gluing}. Every element $\sigma \in \Ker(D_u^N)^*$ has a replacement in the space $V_{g, n, 1}$ of meromorphic $(0, 1)$-forms on $\Sigma$ with a pole of order at most $1$ at each point corresponding to a negative puncture and, for all $j = 1, \ldots, k$, a zero of order at least $\lceil \frac{a_j}{2} \rceil - 1$ at the point corresponding to the $j^\text{th}$ positive puncture. The map $\Ker(D_u^N)^* \to \Lambda^{0, 1} T^* \Sigma$ that sends $\sigma \in \Ker(D_u^N)^*$ to its replacement is an isomorphism onto a real subspace of $V_{g, n, 2}$ with real codimension $k - \ell$.
\end{lemma}

\begin{proof}
	Let $\pi_\R \colon \R \times Y_+ \to \R$ be the projection onto the $\R$-factor and define $s = \pi_\R \circ u \colon \dot \Sigma \to \R$. After perturbing the asymptotic operator as in \Cref{subsec:asymptotic-operator-deformation}, multiplying an element $\sigma \in \Ker(D_u^N)^*$ by $e^{- s / 2}$ yields a $(0, 1)$-form that is anti-meromorphic away from the level sets of $s$. Continuity then implies that $e^{- s / 2} \sigma$ is globally anti-meromorphic on $\dot \Sigma$, and we define our map as $\sigma \mapsto e^{- s / 2} \sigma$. The proof of \Cref{prop:main-gluing-one-forms} shows that, in cylindrical coordinates, the form $d \bar z$ has winding number $1$ around a point corresponding to a positive puncture and winding number $-1$ around a point corresponding to a negative puncture. The first statement in the lemma now follows.
	
	The map $\Ker(D_u^N)^* \to \Lambda^{0, 1}T^*\Sigma$ is real-linear and injective. The space $V_{g, n, 1}$ is isomorphic to $\overline{L(K - D)}$, where $K$ is a canonical divisor of $\Sigma$,
	\begin{equation*}
		D
		=
		\sum_{j = 1}^k \left( \left\lceil \frac{a_j}{2} \right\rceil - 1 \right) p_j
		-
		\sum_{j = 1}^n q_j,
	\end{equation*}
	and $L(K - D)$ is the space of meromorphic functions $f$ on $\Sigma$ with $(f) \ge D - K$. Since $\ell < n$,
	\begin{equation*}
		\deg(D)
		=
		-k - \frac{ n - \ell }{ 2 }
		<
		0,
	\end{equation*}
	so by the Riemann-Roch theorem,
	\begin{equation*}
		\dim_\C V_{g, n, 1}
		=
		\dim_\C L(K - D)
		=
		\frac{ n - \ell }{ 2 } + k + g - 1.
	\end{equation*}
	
	Branched covers $u \in \mathcal M^{g, n, 1}(a_1, \ldots, a_k)$ have exactly $2g + k + n - 2$ branch points. Thus, the argument used in the proof of \Cref{prop:coker-dim} shows that $\dim_\R \Ker D_u^N = 0$ and $\ind(D_u^N) = \ind(u) - 2 (k + n + 2g - 2) = 2 - 2g - k - n$. Hence, $\dim_\R \Ker(D_u^N)^* = n + k + 2g - 2$. The second statement in the lemma now follows.
\end{proof}

A small modification of the proof of \Cref{lemma:replacement-iso-1} proves the following lemma.

\begin{lemma}
\label{lemma:replacement-iso-2}
	Let $u \colon \dot \Sigma \to \R \times Y_+$ be a branched cover in $\mathcal M^{g, n, 2}(a_1, \ldots, a_k)$ with $\ell < k$, where $\ell$ is as in \Cref{notation:non-gluing}. Every element $\sigma \in \Ker(D_u^N)^*$ has a replacement in the space $V_{g, n, 2}$ of meromorphic $(0, 1)$-forms on $\Sigma$ with a pole of order at most $2$ at the point corresponding to the multiplicity $3$ negative puncture, a pole of order at most $1$ at each other point corresponding to a negative puncture and, for all $j = 1, \ldots, k$, a zero of order at least $\lceil \frac{a_j}{2} \rceil - 1$ at the point corresponding to the $j^\text{th}$ positive puncture. The map $\Ker(D_u^N)^* \to \Lambda^{0, 1} T^* \Sigma$ that sends $\sigma \in \Ker(D_u^N)^*$ to its replacement is an isomorphism onto a real subspace of $V_{g, n, 2}$ with real codimension $k - \ell$.
\end{lemma}

\begin{proof}[Proof of \Cref{prop:non-gluing}]
	First, we consider the linearized section over the moduli space $\mathcal M^{g, n, 1}(a_1, \ldots, a_k)$. Given $u \in \mathcal M^{g, n, 1}(a_1, \ldots, a_k)$ with domain $\dot \Sigma$, it suffices to exhibit an element $\sigma \in \Ker(D_u^N)^*$ with winding number $\lceil \frac{a_k}{2} \rceil$ at the $k^\text{th}$ positive end and with winding number greater than $\lceil \frac{a_j}{2} \rceil$ at the $j^\text{th}$ positive end for all $j = 1, \ldots, k - 1$. Recall the divisor $D$ from the proof of \Cref{lemma:replacement-iso-1} and consider the divisors
	\begin{equation*}
		D_1
		=
		D + \sum_{j = 1}^{k - 1} p_j
		\quad \text{and} \quad
		D_2
		=
		D + \sum_{j = 1}^k p_j.
	\end{equation*}
	Since $\ell < n$, both $D_1$ and $D_2$ have negative degree, and the Riemann-Roch theorem implies that
	\begin{equation*}
		\dim_\C (L(K - D_1) / L(K - D_2)) = 1,
	\end{equation*}
	where $K$ is a canonical divisor on $\Sigma$. Let $\eta'$ be a meromorphic $1$-form on $\Sigma$ corresponding to an element of $L(K - D_1)$ that remains non-zero in the quotient. If $a_k$ is even, we can multiply $\eta'$ by a complex number to get a form $\eta$ such that $\overline \eta$ is in the image of the map $\Ker(D_u^N)^* \to \Lambda^{0, 1} T^* \Sigma$ from \Cref{lemma:replacement-iso-1}. We then take $\sigma$ to be the element in $\Ker(D_u^N)^*$ corresponding to $\overline \eta$. The proof for $\mathcal M^{g, n, 2}(a_1, \ldots, a_k)$ is a small modification of the previous argument.
\end{proof}

\begin{prop}
\label{prop:non-gluing-aux}
	The linearized obstruction section over $\mathcal M^{0, n, 1}(1, \ldots, 1)$, where the branched covers have $n$ positive ends all with multiplicity $1$, has no zeros for any $n \ge 2$.
\end{prop}

\begin{proof}
	Let $u \in \mathcal M^{0, n, 1}(1, \ldots, 1)$. Label the positive ends by $1, \ldots, n$ and denote the corresponding positive punctures by $p_1, \ldots, p_n$. Denote the negative punctures by $q_1, \ldots, q_n$. Note that $\dim_\R \Ker(D_u^N)^* = 2n - 2$. Using the row-reduction strategy from the proof of \Cref{prop:basis-construction}, we can find a basis $\{ \sigma^1, \sigma^2, \ldots, \sigma^{2n - 3} \sigma^{2n - 4} \}$ for $\Ker(D_u^N)^*$ such that for $k = 0, \ldots, n - 1$, the projections of $\sigma^{2k + 1}$ and $\sigma^{2k + 2}$ to the leading eigenspace on the $j^\text{th}$ positive end are linearly independent for $j = k$ and $j = n$ and vanish otherwise. Viewing the obstruction bundle as a complex vector bundle, we see that the meromorphic $(0, 1)$-form
	\begin{equation*}
		\eta
		=
		\frac{ (\bar z - \bar p_2) \cdots (\bar z - \bar p_n) }{ (\bar z - \bar q_1) \cdots (\bar z - \bar q_n) } d \bar z
	\end{equation*}
	is a replacement for an element of $\Ker(D_u^N)^*$, and thus
	\begin{equation*}
		\mathfrak s_1(u)(\eta)
		=
		\frac{ \alpha_{1, 1} - \alpha_{n, 1} }{ p_1 - p_n }
		\neq 0,
	\end{equation*}
	as desired.
\end{proof}

%==================%
% SECTION 7			%
% GLUING MODELS	%
%==================%

\section{Gluing Models and Evaluation Map Calculations}
\label{sec:models}

In this section, we construct models for the curves obtained after performing the gluing procedure from \Cref{sec:obs-bund-glue}. We then use those models to compute the degree of certain evaluation maps and prove \Cref{thm:gluing-top-piece}. Throughout this section, we identify $[u] \in \mathcal M_n / \R$ with the representative determined by the parametrization from \Cref{sec:obs-bund-glue}.

%===============%
% SECTION 7.1	 %
% GLUING MODELS %
%===============%

\subsection{Gluing models}
\label{subsec:gluing-models}

Assume that the point $(T, [u]) \in [R, \infty) \times \mathcal (M_n / \R)$ glues to $u_1$. Denote by $u \# u_1$ the curve obtained by gluing. Part of the domain of $u \# u_1$ can be identified with the Riemann surface $\Sigma_1$ obtained from the domain $\dot \Sigma$ of the branched cover $u$ by truncating the positive ends at height $T$. Over $\Sigma_1$, we can write $u \# u_1$ as the graph of a section $\nu$ of the pullback of the normal bundle of the trivial cylinder $\R \times \beta_0$ in $\R \times Y_+$. Since the normal bundle is a trivial complex holomorphic line bundle, we can view $\nu$ as a complex-valued function on $\Sigma_1$. By the argument used to prove \Cref{lemma:replacement-iso-1}, we may assume that $\nu$ is genuinely holomorphic away from the punctures. We view $\Sigma_1$ as the extended complex plane $\widehat \C$ with a finite number of disks removed. Using the analysis in the proof of \cite[Proposition 8.7.2]{BH1}, we can write $\nu$ in cylindrical coordinates in an annulus around a positive puncture $p_i$ as
\begin{equation*}
	\nu(s, t)
	=
	\left( s, t, \sum_{\ell = 1}^\infty \left[ (\alpha_{i, \ell} e^{-\ell T} + d_{i, \ell}) e^{\lambda_i s} f_i(t) \right] + (\text{lower-order terms}) \right),
\end{equation*}
where the $d_{i, \ell}$ are constants, depending on $T$, coming from the perturbation in the obstruction bundle gluing construction and satisfy $|\alpha_{i, \ell}| \gg |e^{\ell T} d_{i, \ell}|$.

\begin{defn}
\label{defn:full-model}
	A \textbf{full model} associated to $(T, [u]) \in [R, \infty) \times (\mathcal M_n / \R)$ is complex-valued function $g$ on the domain of $u$ that is holomorphic except for essential singularities at $p_1, \ldots, p_n$, that has a zero of order $2$ at infinity, that has simple zeros at $q_2, \ldots, q_{n - 2}$, and such that the principal part of $g$ at $p_i$ in cylindrical coordinates $(s, t) \in [T, \infty) \times (\R / 2 \pi \Z)$, where $z = p_i + e^{-(s + it)}$, is
	\begin{equation*}
		\sum_{\ell = 1}^\infty (\alpha_{i, \ell} e^{-\ell T} + d_{i, \ell}) r_i^\ell e^{\ell(s + it)},
	\end{equation*}
	where the $r_i$ are as in \Cref{notation:condense-replacements}.
\end{defn}

Note that a full model associated to $(T, [u])$, if one exists, is unique and is given by
\begin{equation}
\label{eqn:full-model-function}
	g(z)
	=
	\sum_{i = 1}^n \sum_{\ell = 1}^\infty
	\frac{ e^{i \ell \theta} r_i^\ell (\alpha_{i, \ell} e^{-\ell T} + d_{i, \ell}) }{ (z - p_i)^\ell },
\end{equation}
Thus, whether or not $(T, [u])$ has a full model is determined by the locations of the zeros of $g$.

\begin{defn}
\label{defn:models}
	An \textbf{order $m$ model} associated to $(T, [u]) \in [R, \infty) \times (\mathcal M_n / \R)$ is a meromorphic function $g$ on the domain of $u$ with a pole of order $m$ at each positive puncture $p_i$, a zero of order $2$ at infinity, and simple zeros at $q_2, \ldots, q_{n - 2}$, such that the principal part of $g$ at $p_i$ in cylindrical coordinates $(s, t) \in [T, \infty) \times (\R / 2 \pi \Z)$, where $z = p_i + e^{-(s + it)}$, is
	\begin{equation*}
		\sum_{\ell = 1}^m \alpha_{i, \ell} e^{-\ell T} r_i^\ell e^{\ell(s + it)}.
	\end{equation*}
\end{defn}

Note that an order $m$ model for $(T, [u]) \in [R, \infty) \times \mathcal M_n$, if one exists, is unique and is given by \begin{equation*}
	g(z)
	=
	\sum_{i = 1}^n \sum_{\ell = 1}^m
	\frac{ e^{i \ell \theta} e^{- \ell T} r_i^\ell \alpha_{i, \ell} }{ (z - p_i)^\ell }.
\end{equation*}
As before, whether or not $(T, [u])$ has an order $m$ model is determined by the locations of the zeros of $g$. As the name suggests, the models defined above are related to curves obtained by gluing branched covers of trivial cylinders. An order $m$ model associated to $u$ is the approximation to the function $\nu$ on $\Sigma_1$ obtained by truncating the principal part at each singularity to the leading $m$ terms.

The following theorem describes when a branched cover has an associated order $m$ model. The proof is given in \Cref{sec:model-existence-proof}.

\begin{thm}
\label{thm:models}
	A point $(T, [u]) \in [R, \infty) \times \mathcal M_n / \R$ has an associated order $m$ model if and only if $(T, [u]) \in \mathcal Z_m$.
\end{thm}

An examination of the proof of \Cref{thm:models} yields the following corollary.

\begin{cor}
\label{cor:full-model}
	A point $(T, [u]) \in [R, \infty) \times \mathcal M_n / \R$ has an associated full model if and only if $(T, [u]) \in \mathcal Z$.
\end{cor}

%=======================%
% SECTION 7.2			 %
% MODEL EVALUATION MAPS %
%=======================%

\subsection{Evaluation maps in the models}
\label{subsec:model-evaluation-map}

We compute the evaluation map for the curves obtained by gluing branched covers of trivial cylinders as in \Cref{sec:obs-bund-glue}. We also define an evaluation map that is suitable for use in an order $1$ model, allowing for translations of the glued curve in the $\R$-direction of $\R \times Y_+$, and compute the count of gluings. Throughout these calculations, $s$ denotes the $\R$-coordinate of $\R \times Y_+$.

The curves obtained by gluing branched covers in $\mathcal M_n / \R$ to $u_1$ live in the moduli space $\mathcal M = \mathcal M_{\R \times Y_+}^{2n - 3, 1 + n(n - 1)/2}(\bm \alpha, \bm \beta)$. Let $\mathcal E$ denote the end of $\mathcal M$ consisting of glued curves. The gluing procedure gives a diffeomorphism $[R, \infty) \times \mathcal Z \cong \mathcal E / \R$, where the $\R$-action on $\mathcal E$ is given by translating curves in the $\R$-direction of $\R \times Y_+$, and the full models from \Cref{defn:full-model} give a section $\mathcal E / \R \to \mathcal E$. We obtain all curves in $\mathcal E$ by translating curves in the image of this section in the $\R$-direction of $\R \times Y_+$, and thus we have an identification $[R, \infty) \times \mathcal Z \times \R \cong \mathcal E$. We first compute expressions for the components of the evaluation map on curves in the image of this section, i.e., on $[R, \infty) \times \mathcal Z \times \{ 0 \}$. We will abuse notation and let $\ev_j$ denote the restriction to $\mathcal E$ of the evaluation map on $\mathcal M$ at the negative puncture $q_j$.

\begin{prop}
\label{cor:full-model-coeffs}
	For any $(T, [u]) \in [R, \infty) \times \mathcal Z$, we have
	\begin{equation*}
		\ev_1(T, [u], 0)
		=
		e^{i \theta / 3} \sum_{i = 1}^n \frac{ p_i B(p_i) }{ A_i(p_i) } (\alpha_{i, 1} e^{-T} + d_{i, 1})
		+
		e^{4 i \theta / 3} \sum_{i = 1}^n \frac{ B(p_i)^2 }{ A_i(p_i)^2 } (\alpha_{i, 2} e^{-2T} + d_{i, 2})
	\end{equation*}
	and, for $k = 2, \ldots, n - 2$,
	\begin{equation*}
		\ev_k(T, [u], 0)
		=
		\frac{1}{B_k(q_k)}
		\sum_{i = 1}^n \sum_{\ell = 1}^\infty
		(-1)^{\ell - 1} \ell \frac{ A_i(q_k) B_k(p_i)^\ell }{ A_i(p_i)^\ell } (\alpha_{i, \ell} e^{- \ell T} + d_{i, \ell}) e^{(\ell - 1) i \theta}
	\end{equation*}
\end{prop}

\begin{proof}
The result follows directly from \eqref{eqn:full-model-function} after making the same change to the $s$-coordinate as in \Cref{prop:main-gluing-one-forms}.
\end{proof}

To obtain expressions for the components of the evaluation map on all of $\mathcal E$, we need only multiply the above expressions by appropriate exponential functions to account for the effect on the evaluation map of translating a curve in the $\R$-direction of $\R \times Y_+$. The following proposition follows immediately.

\begin{prop}
\label{defn:full-model-evmap}
	Let $\ev_j$ denote the restriction to $\mathcal E \cong [R, \infty) \times \mathcal Z \times \R$ of the evaluation map on $\mathcal M$ at the negative puncture $q_j$. Then
	\begin{equation*}
		\ev_1(T, [u], s)
		=
		e^{-s/3}
		\ev_1(T, [u], 0)
	\end{equation*}
	and, for $k = 2, \ldots, n - 2$,
	\begin{equation*}
		\ev_k(T, [u], s)
		=
		e^{-s}
		\ev_k(T, [u], 0)
	\end{equation*}
\end{prop}

\begin{defn}
\label{defn:model-ev-map}
	The \textbf{model evaluation map} on $\mathcal Z_1 \times \R$ is defined as follows. Let $(T, [u], s) \in \mathcal Z_1 \times \R$ and let $g$ be the model of order $1$ associated to $(T, [u])$. We define the model evaluation map at the negative puncture $q_j$, $j = 2, 3, \ldots, n - 2$, to be the leading complex asymptotic coefficient in the Fourier-type expansion of $g(z)$ in cylindrical coordinates around the puncture $q_j$, multiplied by $e^{-s}$:
	\begin{equation}
	\label{eqn:model-evmap-mult1}
		\mev_{j}(T, [u], s)
			=
				e^{-s} e^{- i \theta} r_{-j} g'(q_j).
	\end{equation}
	Let $h(\zeta) = g(\zeta^{-1})$. We define the model evaluation map at the negative puncture $q_1$ to be the leading complex asymptotic coefficient in the Fourier-type expansion of $h(\zeta)$ in cylindrical coordinates around $\zeta = 0$, which corresponds to the puncture $q_1$, multiplied by $e^{-s/3}$:
	\begin{equation}
	\label{eqn:model-evmap-mult3}
		\mev_{1}(T, [u], s)
			=
				e^{-s/3} e^{ -2 i \theta / 3} \frac{h''(0)}{2}.
	\end{equation}
	If $I$ is a subset of $\{ 1, 2, \ldots, n - 2 \}$, we define the model evaluation map at the negative punctures $q_j$, $j \in I$, by
	\begin{equation*}
		\mev_{I}(T, [u], s)
			=
				\left( \mev_{j}(T, [u], s) \right)_{j \in I}.
	\end{equation*}
\end{defn}

We now compute the degree of $\mev_{I}$ on $(\mathcal Z_1 \cap (\{T\} \times (\mathcal M_n / \R))) \times \R$ when $T$ is fixed and sufficiently large. For simplicity of notation, set
\begin{equation}
\label{eqn:H_i}
	H_i
	=
	\frac{ B(p_i) }{ A_i(p_i) } \alpha_{i, 1},
\end{equation}
so that
\begin{equation*}
	g(z)
	=
	e^{i \theta} e^{-T}
	\sum_{i = 1}^n \frac{ B(p_i) }{ A_i(p_i) } \frac{ \alpha_{i, 1} }{ z - p_i }
	=
	e^{i \theta} e^{-T}
	\sum_{i = 1}^n \frac{ H_i }{ z - p_i }.
\end{equation*}

\begin{lemma}
\label{lemma:mult-3-coeff}
	If $h(\zeta) = g(\zeta^{-1})$, we have
	\begin{equation*}
		h''(0)
		=
		2 e^{i \theta / 3} e^{-T}
		\sum_{i = 1}^n p_i H_i.
	\end{equation*}
\end{lemma}

\begin{proof}
We want to compute the leading coefficient in the Taylor expansion of $e^{ -2 i \theta / 3} h(\zeta)$ at $\zeta = 0$. Since $h'(0) = 0$, we have
\begin{equation}
\label{eqn:h'(0)=0}
	\sum_{i = 1}^n H_i
	=
	0,
\end{equation}
and hence
\begin{align*}
	e^{ -2 i \theta / 3} h(\zeta)
		&=
			e^{i \theta / 3} e^{-T}
			\zeta
			\sum_{i = 1}^n \frac{ H_i }{ 1 - p_i \zeta } \\
		&=
			e^{i \theta / 3} e^{-T}
			\zeta
			\sum_{i = 1}^n
			\left[ \frac{ H_i }{ 1 - p_i \zeta } - H_i \right] \\
		&=
			e^{i \theta / 3} e^{-T}
			\zeta^2
			\sum_{i = 1}^n
			\frac{ p_i H_i }{ 1 - p_i \zeta }.
\end{align*}
The result follows.
\end{proof}

\begin{lemma}
\label{lemma:derived-eqns}
	If $(T, [u]) \in \mathcal Z_1$, we have
	\begin{equation*}
		\sum_{i = 1}^n \frac{ H_i }{ (q_{k_1} - p_i) \cdots (q_{k_j} - p_i) }
		=
		0,
	\end{equation*}
	where $j \in \{ 1, \ldots, n - 3 \}$ and $k_1, \ldots, k_j$ are distinct elements of $\{ 2, \ldots, n - 2 \}$.
\end{lemma}

\begin{proof}
The proof is by induction on $j \le n - 3$. Since $g(q_k) = 0$ for $k = 2, 3, \ldots, n - 2$, we have
\begin{equation*}
	0
		=
			\sum_{i = 1}^n \frac{ H_i }{ q_k - p_i }
\end{equation*}
for $k = 2, 3, \ldots, n - 2$, which establishes the case $j = 1$.

Now assume the result for some $j \le n - 4$. If $k_1, \ldots, k_{j + 1}$ are distinct elements of $\{2, \ldots, n - 2\}$, we see that
\begin{align*}
	0
		&=
			\sum_{i = 1}^n \frac{ H_i }{ (q_{k_1} - p_i) \cdots (q_{k_j} - p_i) }
			-
			\sum_{i = 1}^n \frac{ H_i }{ (q_{k_1} - p_i) \cdots (q_{k_{j - 1}} - p_i) (q_{k_{j + 1}} - p_i) } \\
		&=
			\sum_{i = 1}^n \frac{ H_i }{ (q_{k_1} - p_i) \cdots (q_{k_{j - 1}} - p_i) }
			\left[ \frac{ 1 }{ q_{k_j} - p_i } - \frac{ 1 }{ q_{k_{j + 1}} - p_i } \right] \\
		&=
			(q_{k_{j + 1}} - q_{k_j})
			\sum_{i = 1}^n \frac{ H_i }{ (q_{k_1} - p_i) \cdots (q_{k_{j + 1}} - p_i) }.
\end{align*}
Since $q_{k_{j + 1}} - q_{k_j} \neq 0$, the inductive step follows.
\end{proof}

\begin{lemma}
\label{lemma:model-rewrite}
	If $(T, [u]) \in \mathcal Z_1$, we can write
	\begin{equation*}
		g(z)
		=
		e^{i \theta} e^{-T}
		B(z)
		\sum_{i = 1}^n \frac{ \alpha_{i, 1} }{ A_i(p_i) } \frac{ 1 }{ z - p_i }.
	\end{equation*}
\end{lemma}

\begin{proof}
We show inductively that
\begin{equation*}
	g(z) = e^{i\theta} e^{-T} \prod_{k = 2}^m (z - q_k) \sum_{j = 1}^n \frac{ H_j }{ \prod_{k = 2}^m (p_j - q_k) } \frac{ 1 } {z - p_j}
\end{equation*}
for $m = 2, \ldots, n - 2$. The lemma then follows by taking $m = n - 2$.

For the case $m = 2$, use \Cref{lemma:derived-eqns} to write
\begin{align*}
	g(z)
		&=
			e^{i \theta} e^{-T}
			\sum_{i = 1}^n \frac{ H_i }{ z - p_i } \\
		&=
			e^{i \theta} e^{-T}
			\sum_{i = 1}^n \left[ \frac{ H_i }{ z - p_i } - \frac{ H_i }{ q_2 - p_i } \right] \\
		&=
			e^{i \theta} e^{-T}
			(q_2 - z)
			\sum_{i = 1}^n \frac{ H_i }{ (q_2 - p_i)(z - p_i) } \\
		&=
			e^{i \theta} e^{-T}
			(z - q_2)
			\sum_{i = 1}^n \frac{ H_i }{ (p_i - q_2) } \frac{ 1 }{ z - p_i }.
\end{align*}

Now assume that the lemma is true for some $m$ with $2 \le m \le n - 3$. By \Cref{lemma:derived-eqns}, we have
\begin{align*}
	g(z)
		&=
			e^{i \theta} e^{-T}
			\prod_{k = 2}^m (z - q_k)
			\sum_{j = 1}^n \frac{ H_j }{ \prod_{k = 2}^m (p_j - q_k) } \frac{ 1 } {z - p_j} \\
		&=
			e^{i \theta} e^{-T}
			\prod_{k = 2}^m (z - q_k)
			\sum_{j = 1}^n \left[ \frac{ H_j }{ \prod_{k = 2}^m (p_j - q_k) } \frac{ 1 } {z - p_j} - \frac{ H_j }{ (q_{m + 1} - p_j) \prod_{k = 2}^m (p_j - q_k) } \right] \\
		&=
			e^{i \theta} e^{-T}
			\prod_{k = 2}^{m + 1} (z - q_k)
			\sum_{j = 1}^n \frac{ H_j }{ \prod_{k = 2}^{m + 1} (p_j - q_k) } \frac{ 1 }{ z - p_j },
\end{align*}
and we are done.
\end{proof}

\begin{lemma}
\label{lemma:mult-1-coeff}
	For $k = 2, \ldots, n - 2$, we have
	\begin{equation*}
		e^{- i \theta} r_{-k}
		g'(q_k)
		=
		(-1)^n e^{-T}
		\frac{ (\alpha_{n - 1, 1} - \alpha_{n, 1}) q_k - (p_n \alpha_{n - 1, 1} - p_{n - 1} \alpha_{n, 1}) }{ p_{n - 1} - p_n }.
	\end{equation*}
\end{lemma}

\begin{proof}
Set
\begin{equation*}
	\Delta
	=
	\prod_{1 \le i < j \le n} (p_i - p_j)
	\quad \text{and} \quad
	\Delta_i
	=
	\prod_{\substack{ 1 \le j < k \le n \\ j, k \neq i }} (p_j - p_k).
\end{equation*}
Let $E_\ell$ be the $\ell^\text{th}$ elementary symmetric polynomial and set
\begin{equation*}
	E_{\ell, i}
	=
	E_\ell( p_1, \ldots, \widehat{p_i}, \ldots, p_n ).
\end{equation*}
The leading coefficient in the power series expansion of $e^{- i \theta} r_{-k} g(z)$ at $z = q_k$ is
\begin{align*}
	r_{-k} e^{-T}
	B_k(q_k) \sum_{i = 1}^n \frac{ \alpha_{i, 1} }{ A_i(p_i) } \frac{ 1 }{ q_k - p_i }
	& =
	e^{-T} \sum_{i = 1}^n \frac{ A_i(q_k) }{ A_i(p_i) } \alpha_{i, 1} \\
	& =
	\frac{ e^{-T} }{ \Delta } \sum_{i = 1}^n (-1)^{i - 1} A_i(q_k) \alpha_{i, 1} \Delta_i \\
	&=
	\frac{ e^{-T} }{ \Delta } \sum_{i = 1}^n (-1)^{i - 1} \alpha_{i, 1} \Delta_i \sum_{ \ell = 0 }^{n - 1} (-1)^\ell E_{\ell, i}q_k^{n - 1 - \ell} \\
	& =
	\frac{ e^{-T} }{ \Delta } \sum_{\ell = 0}^{n - 1} (-1)^\ell q_k^{n - 1 - \ell} \sum_{i = 1}^n (-1)^{i - 1} E_{\ell, i} \alpha_{i, 1} \Delta_i \\
	& =
	\frac{ e^{-T} }{ \Delta } \sum_{\ell = 0}^{n - 1} (-1)^\ell q_k^{n - 1 - \ell}
	\det M_\ell,
\end{align*}
where
\begin{equation*}
	M_\ell
	=
	\begin{pmatrix}
		E_{\ell, 1} \alpha_{1, 1}	&	\cdots	&	E_{\ell, n} \alpha_{n, 1}	\\
		p_1^{n - 2}			&	\cdots	&	p_n^{n - 2}			\\
		\vdots				&	\ddots	&	\vdots				\\
		p_1					&	\cdots	&	p_n					\\
		1					&	\cdots	&	1
	\end{pmatrix}.
\end{equation*}

\begin{claim}
\label{claim:det-reduction}
	We have
	\begin{equation*}
		(-1)^n \frac{ \det M_\ell }{ \Delta }
		=
		\begin{cases}
			\hfil \frac{ p_n \alpha_{n - 1, 1} - p_{n - 1} \alpha_{n, 1} }{ p_{n - 1} - p_n }	&	\ell = n - 1	\\
			\hfil \frac{ \alpha_{n - 1, 1} - \alpha_{n, 1} }{ p_{n - 1} - p_n }			&	\ell = n - 2 \\
			\hfil 0													&	\ell < n - 2
		\end{cases}.
	\end{equation*}
\end{claim}

\noindent The lemma follows from \Cref{claim:det-reduction}. The proof of the claim is an exercise in careful row-reduction and is given in \Cref{sec:det-calc}.
\end{proof}

\begin{prop}
\label{prop:order-1-evmap-degree}
	Let $I = \{1, 2, \ldots, n - 2\}$ and suppose that $R \gg 0$ is sufficiently large in the prototypical gluing problem. For any fixed $T \ge R$ and any admissible asymptotic restriction $\mathbf c \in (\C^*)^{n - 2}$, the degree of the restriction of $\mev_{I}$ to $(\mathcal Z_1 \cap (\{T\} \times (\mathcal M_n / \R))) \times \R$ is $1 \pmod 2$.
\end{prop}

\begin{proof}
Choose an admissible asymptotic restriction $\mathbf c = (c_1, \ldots, c_{n - 2}) \in (\C^*)^{n - 2}$. We must count solutions of the equations
\begin{align*}
	e^{- s / 3 } e^{i \theta / 3} e^{-T}
	\sum_{i = 1}^n p_i H_i
		&=
			c_1, \\
	(-1)^n e^{-s} e^{-T}
	\frac{ (\alpha_{n - 1, 1} - \alpha_{n, 1}) q_k - (p_n \alpha_{n - 1, 1} - p_{n - 1} \alpha_{n, 1}) }{ p_{n - 1} - p_n }
		&=
			c_k,
\end{align*}
$k = 2, \ldots, n - 2$, where $T$ is fixed by our assumptions, the $p_i$ are fixed, distinct points in $\C$ by \Cref{cor:top-punctures-fixed}, $q_2, \ldots, q_{n - 2}$ are allowed to vary in $\C \setminus \{ p_1, \ldots, p_n \}$, $s$ is allowed to vary in $\R$, and $\theta$ is allowed to vary in $(\R / 6 \pi \Z)$. The last $n - 3$ equations have solutions
\begin{equation*}
	q_k
	=
	\frac{ (-1)^n e^s e^T (p_{n - 1} - p_n) c_k + (p_n \alpha_{n - 1, 1} - p_{n - 1} \alpha_{n, 1}) }{ \alpha_{n - 1, 1} - \alpha_{n, 1} },
\end{equation*}
$k = 2, \ldots, n - 2$. Now substitute back into the first equation and consider the norm of the left-hand side. If $s$ is very large and positive, the norm is larger than $|c_1|$, while if $s$ is very large and negative, the norm is smaller than $|c_1|$. Thus, there are an odd number of values of $s$ for which the norms of both sides of the first equation are equal. Since $\theta \in \R / 6 \pi \Z$, there is a unique choice of $\theta$ that solves the first equation.
\end{proof}

%================================%
% SECTION 7.3					    %
% REDUCTION TO FIRST-ORDER MODEL %
%================================%

\subsection{Reduction to the first-order model}
\label{subsec:first-order-reduction}

We now show that the full evaluation map and the evaluation map in the order $1$ model have the same degree.

\begin{prop}
\label{prop:reduce-first-order}
	Let $I = \{1, 2, \ldots, n - 2\}$ and suppose that $R \gg 0$ is sufficiently large in the prototypical gluing problem. For any fixed $T \ge R$ and any admissible asymptotic restriction $\mathbf c \in (\C^*)^{n - 2}$, the mod $2$ degree of the restriction of $\mev_{I}$ to $(\mathcal Z_1 \cap (\{T\} \times (\mathcal M_n / \R))) \times \R$ is equal to the mod $2$ degree of the restriction of $\ev_{I}$ to $(\mathcal Z \cap (\{T\} \times (\mathcal M_n / \R))) \times \R$.
\end{prop}

\begin{proof}[Proof of \Cref{prop:reduce-first-order}]
Let $\mathbf c$ be an admissible asymptotic restriction and let $T \gg R$. Define a map $F \colon (\mathcal M_n / \R) \times \R \to \C^{n - 2} \times \C^{n - 2}$ by
\begin{equation*}
	F([u], s)
	=
	( \mathfrak s(T, [u])(\eta_1), \ldots, \mathfrak s(T, [u])(\eta_{n - 2}), \ev_{I}(T, [u], s) ).
\end{equation*}
Thus, the set of all $(T, [u]) \in \{T\} \times (\mathcal M_n / \R)$ that glue to $u_+$ and satisfy the admissible asymptotic restrictions $\mathbf c$ when translated in the $\R$-direction by $s$ is $F^{-1}( \{ \mathbf 0 \} \times \{ \mathbf c \} )$.

Define a homotopy $F_\nu$ of $F$ in the following way. For $\nu \in [0, \frac{1}{2}]$, we can define a homotopy $\mathfrak s_\nu$ of the obstruction section $\mathfrak s$ such that $\mathfrak s_0 = \mathfrak s$ and $\mathfrak s_{\frac{1}{2}}$ is the linear portion of $\mathfrak s$; see, e.g., \cite[Section 8.7]{BH1} for a similar construction. Let $\ev_k(T, [u], s)$ denote the evaluation map at the puncture $q_k$, $k = 1, 2, \ldots, n - 2$. Define a homotopy of each $\ev_k$ for $\nu \in [0, \frac{1}{2}]$ by replacing each $\alpha_{i, \ell} e^{-\ell T} + d_{i, \ell}$ with $\alpha_{i, \ell} e^{- \ell T} + (1 - 2 \nu) d_{i, \ell}$, and set $\ev_{I, \nu} = (\ev_{k, \nu})_{k = 1}^{n - 2}$. Then for $\nu \in [0, \frac{1}{2}]$, we define
\begin{equation*}
	F_\nu([u], s)
	=
	( \mathfrak s_\nu(T, [u])(\eta_1), \ldots, \mathfrak s_\nu(T, [u])(\eta_{n - 2}), \ev_{I, \nu}(T, [u], s) ).
\end{equation*}

For $\nu \in [\frac{1}{2}, 1]$, let $\eta_{k, \nu}$ denote the linear interpolation from $\eta_k$ to $\Pi \eta_k$, the projection of $\eta_k$ onto the leading eigenspace at each positive end. Let $\ev_{k, \nu}$, $\nu \in [\frac{1}{2}, 1]$, be the linear interpolation that kills all terms with $\ell \ge 2$ and set $\ev_{I, \nu} = (\ev_{k, \nu})_{j = 1}^{n - 2}$. Then for $\nu \in [\frac{1}{2}, 1]$, we define
\begin{equation*}
	F_\nu([u], s)
	=
	( \mathfrak s_{\frac{1}{2}}(T, [u])(\eta_{1, \nu}), \ldots, \mathfrak s_{\frac{1}{2}}(T, [u])(\eta_{n - 2, \nu}), \ev_{I, \nu}(T, [u], s) ).
\end{equation*}
Note that $F_0 = F$ and
\begin{equation*}
	F_1([u], s)
	=
	( \mathfrak s_1(T, [u])(\eta_1), \ldots, \mathfrak s_1(T, [u])(\eta_{n - 2}), \mev_{I}(T, [u], s) ).
\end{equation*}

\begin{claim}
\label{claim:no-bdry-solns}
	Let $K \subset (\mathcal M_n / \R) \times \R$ be a compact set such that $F_1^{-1}(\{\bm 0, \bm c\})$ is contained in the interior of $\{T\} \times K$ . If $T \gg 0$ is sufficiently large, then $F_\nu^{-1}(\{\bm 0, \bm c\}) \cap (\{T\} \times \bdry K) = \varnothing$ for all $\nu \in [0, 1]$.
\end{claim}

\begin{proof}[Proof of \Cref{claim:no-bdry-solns}]
Suppose that the claim is false. Then there is a sequence $\{T_k, \nu_k\}$ with $T_k \to \infty$ such that $F_{\nu_k}^{-1}(\{\bm 0, \bm c\}) \cap (\{T_k\} \times \bdry K) \neq \varnothing$ for all $k$. We will arrive at a contradiction by showing that the homotopy $F_\nu$ is very small on $\{T\} \times K$ if $T$ is sufficiently large.

Since $K$ is compact, there is a large positive constant $C$ such that $|p_i - p_j| > C^{-1}$ and $|q_i - q_j| > C^{-1}$ for all $i \neq j$. In addition, we may assume that $C^{-1} < s < C$. We may also assume that each of the punctures $p_1, \ldots, p_n, q_2, \ldots, q_{n - 2}$, except possibly for one of the positive punctures $p_i$, is contained in the disk of diameter $C$ centered at the origin in $\C$. If all of the punctures in question are contained in the disk, then $|p_i - q_j| < C$ for all $i \neq j$. In this case, we have, for $\nu \in [\frac{1}{2}, 1]$ and $k = 2, \ldots, n - 2$,
\begin{align*}
	| \ev_{k, \nu}(T, [u], s) - \mev_{k}(T, [u], s) |
		&\le
			\frac{ e^C }{ |B_k(q_k)| }
			\sum_{i = 1}^n
			\sum_{\ell = 2}^\infty
			\ell \left| \frac{ A_i(q_k) B_k(p_i)^\ell \alpha_{i, \ell} }{ A_i(p_i)^\ell } \right| e^{- \ell T} \\
		&\le
			e^C
			\sum_{i = 1}^n
			\sum_{\ell = 2}^\infty
			\ell C^{(\ell + 1)(2n - 5)} |\alpha_{i, \ell}| e^{- \ell T},
\end{align*}
and the right-hand side can be made as small as we like by taking $T$ to be sufficiently large.

Now assume that one of the positive punctures, say $p_j$, is outside of the above-mentioned disk. There is a constant $D > C$, depending only on $n$ and $C$, such that if $|p_j| > D$, then
\begin{equation*}
	\left| \frac{ B_k(p_j)^\ell }{ A_j(p_j)^\ell } \right| < 1 \,\, \text{and} \,\,
	\left| \frac{ A_i(q_k) }{ A_i(p_i)^\ell } \right| < 1
\end{equation*}
for $i \neq j$ and $\ell \ge 2$. Moreover, when $C \le |p_j| \le D$, we have
\begin{equation*}
	\left| \frac{ B_k(p_j)^\ell }{ A_j(p_j)^\ell } \right| \le C^{(n - 1)\ell}(C + D)^{(n - 3)\ell} \,\, \text{and} \,\,
	\left| \frac{ A_i(q_k) }{ A_i(p_i)^\ell } \right| \le C^{(n - 1)\ell + n - 2} (C + D)
\end{equation*}
for all $i \neq j$. It follows that, in this case, we can make $| \ev_{k, \nu}(T, [u], s) - \mev_{k}(T, [u], s) |$ as small as we like by taking $T$ to be sufficiently large.

Since $|d_{i, \ell}| \ll |\alpha_{i, \ell}| e^{-\ell T}$, a similar estimate shows that, for $\nu \in [0, \frac{1}{2}]$, we can make $| \ev_{j, \nu}(T, [u], s) - \ev_{j, 0}(T, [u], s) |$ as small as we like by taking $T$ to be sufficiently large. Another similar estimate shows the same result for the leading coefficient at the puncture $q_1$.

Now we prove a similar result for the homotopy of the obstruction section. By \eqref{eqn:partial-fractions}, we have
\begin{equation*}
	\frac{ 1 }{ (\ell - 1)! } \frac{ d^{\ell - 1} Q_i }{ d z^{\ell - 1} }(z)
	=
	- \sum_{k = 2}^{n - 2} \frac{ \mathbb A_i(q_k) }{ B_k(q_k) } \frac{ 1 }{ (q_k - z)^\ell }
\end{equation*}
for $i = 1, \ldots, n$ and $\ell \ge 1$. If each of $p_1, \ldots, p_n, q_2, \ldots, q_{n - 2}$ is in the disk of diameter $C$ centered at the origin, then, for $\nu \in [\frac{1}{2}, 1]$,
\begin{align*}
	| \mathfrak s_{\frac{1}{2}}(T, [u])(\eta_{k, \nu}) - \mathfrak s_1(T, [u])(\eta_k) |
		&\le
			\sum_{i = 1}^n
			\sum_{\ell = 2}^\infty
			\sum_{k = 2}^{n - 2}
			e^{- \ell T} |\alpha_{i, \ell}|
			\left| \frac{ \mathbb A_i(q_k) B_k(p_i)^\ell }{ B_k(q_k) A_i(p_i)^\ell } \right| \\
		&\le
			C^{2n - 6}
			\sum_{i = 1}^n
			\sum_{\ell = 2}^\infty
			\sum_{k = 2}^{n - 2}
			e^{- \ell T} |\alpha_{i, \ell}|
			C^{(2n - 4)\ell} \\
		&=
			n(n - 3)C^{2n - 6}
			\sum_{\ell = 2}^\infty
			e^{- \ell T} |\alpha_{i, \ell}|
			C^{(2n - 4)\ell},
\end{align*}
which can be made as small as we like by taking $T$ to be sufficiently large. When some positive puncture, say $p_j$, is outside of the disk of diameter $C$, the same argument used for the evaluation map shows that we can make
\begin{equation*}
	| \mathfrak s_{\frac{1}{2}}(T, [u])(\eta_{k, \nu}) - \mathfrak s_1(T, [u])(\eta_k) |
\end{equation*}
as small as we like by taking $T$ to be sufficiently large. A similar estimate shows that, for $\nu \in [0, \frac{1}{2}]$,
\begin{equation*}
	| \mathfrak s_\nu (T, [u])(\eta_k) - \mathfrak s_{\frac{1}{2}}(T, [u])(\eta_k) |
\end{equation*}
can be made as small as we like by taking $T$ to be sufficiently large.

We now finish the proof of the claim. Since $K$ is compact and $F_1(\{T\} \times \bdry K)$ does not intersect $\{\bm 0, \bm c\}$, the distance between $F_1(\{T\} \times \bdry K)$ and $\{\bm 0, \bm c\}$ is bounded below by a positive constant. By our above estimates, the homotopy $F_\nu$ can be made as small as we on $\{T\} \times K$ like by taking $T$ to be sufficiently large. Thus, the distance between $F_\nu(\{T\} \times \bdry K)$ and $\{\bm 0, \bm c\}$ is bounded below by a (possibly smaller) positive constant for all $\nu \in [0, 1]$, which contradicts our assumption that $T_k \to \infty$ in the sequence $\{T_k, \nu_k\}$.
\end{proof}

The claim implies \Cref{prop:reduce-first-order}: if the number of points in $F_0^{-1}(\{\bm 0, \bm c\})$ is even, there is a large compact subset $K \subset (\mathcal M_n / \R) \times \R$ containing $F_1^{-1}(\{\bm 0, \bm c\})$ in its interior such that $F_\nu^{-1}(\{\bm 0, \bm c\}) \cap (\{T\} \times \bdry K) \neq \varnothing$ for some $\nu \in (0, 1)$.
\end{proof}

\begin{proof}[Proof of \Cref{thm:gluing-top-piece}]
Combine \Cref{prop:order-1-evmap-degree} and \Cref{prop:reduce-first-order}.
\end{proof}

%======================%
% SECTION 8				%
% THE COBORDISM MAP	%
%======================%

\section{The Cobordism Map}
\label{sec:cob-map-defn}

In this section, we complete the proof of \Cref{thm:main-thm} using the evaluation map discussed in \Cref{sec:ev-map} together with the degree calculations in \Cref{sec:models}. We adapt the truncation procedure used in \cite[Sections 1.3 and 7.3]{HT1}; see \cite[Section 5.4]{H3} for an overview. Throughout this section, $\mathcal M$ denotes the subset of $\mathcal M_{\R \times Y_+}^{2n - 3, 1 + n(n - 1) / 2}(\bm \alpha, \bm \beta)$ consisting of curves with with $n - 3$ negative ends of multiplicity $1$ at $\beta_0$ and one negative end of multiplicity $3$ at $\beta_0$, and $\mathcal M_+$ denotes the subset of $\mathcal M_{\R \times Y_+}^{1, 1 + n(n - 1) / 2}(\bm \alpha, \bm \beta)$ consisting of curves with $n$ negative ends of multiplicity $1$ at $\beta_0$. Note that $u_1 \in \mathcal M_+$ and that the curves obtained by gluing branched covers in $\mathcal M_n / \R$ to $u_1$ live in $\mathcal M$ and that all curves in $\mathcal M$ are somewhere injective. For each curve in $\mathcal M$, label the negative ends asymptotic to $\beta_0$ by the elements of $I = \{1, 2, \ldots, n - 2\}$, where the multiplicity $3$ negative end is given the label $1$. There is an evaluation map $\ev_I \colon \mathcal M \to \C^{n - 2}$.

%=======================
% SECTION 8.1
% SETUP FOR TRUNCATION
%=======================

\subsection{Setup for truncation}
\label{subsec:truncation-setup}

We first collect all of the necessary definitions and auxiliary results for the truncation procedure.

\begin{defn}
\label{defn:generic-restriction}
	For $n \ge 3$, we define the set $\mathcal R_n$ of \textbf{generic asymptotic restrictions} on curves in $\mathcal M$ recursively as follows. Let $\mathcal R_n'$ denote the set of regular values of $\ev_I \colon \mathcal M \to \C^{n - 2}$. For any orbit sets $\bm \gamma_\pm$ with $\mathcal A(\bm \gamma_-) < \mathcal A(\bm \gamma_+) \le \mathcal A(\bm \alpha)$ such that $\beta_0$ is an orbit in $\bm \gamma_-$ with multiplicity $m \le n$ and for all $k = 1, \ldots, 2n - 4$, let $\mathcal V_{m, k}(\bm \gamma_+, \bm \gamma_-)$ be the set of somewhere injective $J$-holomorphic curves from $\bm \gamma_+$ to $\bm \gamma_-$ with Fredholm index $k$, with ECH index at most $\binom{n}{2}$, and such that at most one negative end asymptotic to $\beta_0$ has multiplicity $3$ and the remaining negative ends asymptotic to $\beta_0$ all have multiplicity $1$. Label the negative ends of such curves asymptotic to $\beta_0$ by the elements of $I_m = \{ 1, \ldots, m \}$; if there is a negative end at $\beta_0$ with multiplicity $3$, we require that it be labeled by $1$. Let $\widetilde{\mathcal R}_{m, k}(\bm \gamma_+, \bm \gamma_-)$ be the set of regular values of $\ev_{I_m} \colon \mathcal V_{m, k}(\bm \gamma_+, \bm \gamma_-) \to \C^m$. For each subset $G = \{ j_1, \cdots, j_m \} \subset I$, let $\pi_{G} \colon \C^{n - 2} \to \C^m$ be the projection $(c_1, \ldots, c_{n - 2}) \mapsto (c_{j_1}, \ldots, c_{j_m})$. Each $\widetilde{\mathcal R}_{m, k}(\bm \gamma_+, \bm \gamma_-)$ is a countable intersection of dense, open sets, and hence the same is true for $\mathcal R^{G}_{m, k}(\bm \gamma_+, \bm \gamma_-) = \pi_{G}^{-1}(\widetilde{\mathcal R}_{m, k}(\bm \gamma_+, \bm \gamma_-))$. Finally, define
	\begin{equation*}
		\mathcal R_n
		=
		\left( \bigcap_{m, G, \bm \gamma_\pm, k} \mathcal R_{m, k}^{G}(\bm \gamma_+, \bm \gamma_-) \right)
		\cap
		\mathcal R_n',
	\end{equation*}
	where the intersection is taken over all $m \le n$, all subsets $G \subset I_m$, all pairs of orbit sets $\bm \gamma_\pm$ with $\mathcal A(\bm \gamma_-) < \mathcal A(\bm \gamma_+) \le \mathcal A(\bm \alpha)$ such that $\beta_0$ is an orbit in $\bm \gamma_-$ with multiplicity $m$, and all $k = 1, \ldots, 2n - 4$. Let $\widetilde{\mathcal V}_{m, k}(\bm \gamma_+, \bm \gamma_-) \subset \mathcal V_{m, k}(\bm \gamma_+, \bm \gamma_-)$ denote the subset of curves with ECH index $\binom{n}{2}$.
\end{defn}

\begin{rmk}
\label{rmk:empty-pre-images}
	The pre-image of any $\mathbf c \in \mathcal R_n$ under $\ev_I \colon \mathcal V_{n, k}(\bm \gamma_+, \bm \gamma_-) \to \C^{n - 2}$ is empty for any $k \le 2n - 5$. If $(c_1, \ldots, c_{n - 2}) \in \mathcal R_n$, then $(c_{j_1}, \ldots, c_{j_m}) \in \mathcal R_{m + 2}$ for all subsets $\{ j_1, \ldots, j_m \} \subset I$, and all $m = 1, \ldots, n - 2$.
\end{rmk}

\begin{notation}
\label{notation:pre-image}
	If $\mathbf c \in \C^{n - 2}$ is a generic, admissible asymptotic restriction, let $K_{\mathbf c}$ denote the pre-image of $\mathbf c$ under $\ev_I \colon \mathcal M \to \C^{n - 2}$.
\end{notation}

\begin{rmk}
\label{cor:num-pre-image-comps}
	Since $\mathbf c$ is generic, $K_{\mathbf c}$ is a real $1$-dimensional submanifold of $\mathcal M$.
\end{rmk}

\begin{defn}
\label{defn:gluing-pairs}
	A pair $(w_-, w_+)$ of $J$-holomorphic curves is an \textbf{asymptotically restricted gluing pair} with asymptotic restriction $\mathbf c \in \C^{n - 2}$ if $w_- \in \widetilde{\mathcal V}_{n, 2n - 4}(\bm \gamma, \bm \beta)$ for some orbit set $\bm \gamma$ with $\mathcal A(\bm \beta) < \mathcal A(\bm \gamma) < \mathcal A(\bm \alpha)$, $\ev_I(w_-) = \mathbf c$, and $w_+ \in \mathcal M_{\R \times Y_+}^{1, 1}(\bm \alpha, \bm \gamma)$. We denote the set of all asymptotically restricted gluing pairs with asymptotic restriction $\mathbf c$ by $\mathcal S_{\mathbf c}$.
\end{defn}

\begin{defn}
\label{defn:inverted-gluing-pair}
	A pair $(w_-, w_+)$ of $J$-holomorphic curves is an \textbf{inverted asymptotically restricted gluing pair} with asymptotic restriction $\mathbf c \in \C^{n - 2}$ if $w_+ \in \widetilde{\mathcal V}_{n, 2n - 4}(\bm \alpha, \bm \gamma)$ for some orbit set $\bm \gamma$ with $\mathcal A(\bm \beta) < \mathcal A(\bm \gamma) < \mathcal A(\bm \alpha)$, $\ev_I(w_+) = \mathbf c$, and $w_- \in \mathcal M_{\R \times Y_+}^{1, 1}(\bm \gamma, \bm \beta)$ is a curve that contains $n - 3$ copies of $\R \times \beta_0$ and an unbranched cover of $\R \times \beta_0$ of multiplicity $3$. We denote the set of all asymptotically restricted gluing pairs with asymptotic restriction $\mathbf c$ by $\mathcal S^\text{inv}_{\mathbf c}$.
\end{defn}

\begin{defn}
\label{defn:close-to-breaking-degenerate}
	For $R \gg 0$, let $\mathcal G_R$ be the intersection of $K_{\mathbf c}$ with the end of $\mathcal M$ corresponding to $\mathcal Z$ where the gluing parameter $T > R$.
\end{defn}

\begin{rmk}
\label{rmk:num-of-components}
	By \Cref{prop:order-1-evmap-degree} and \Cref{prop:reduce-first-order}, the number of components of $\mathcal G_R$ is finite and odd.
\end{rmk}

The next two definitions are analogues of \cite[Definition 1.10]{HT1}.

\begin{defn}
\label{defn:close-to-breaking-restricted}
	Let $\mathbf c \in \C^{n - 2}$ be a generic, admissible asymptotic restriction, let $(w_-, w_+) \in \mathcal S_{\mathbf c}$, let $\delta > 0$, and choose a product metric on $\R \times Y_+$. Let $\mathcal C_{\mathbf c, \delta}(w_-, w_+)$ be the set of $J$-holomorphic curves in $K_{\mathbf c}$ whose images can be decomposed into two surfaces with boundary $C_- \cup C_+$ such that the following hold.
	\begin{enumerate}
		\item
		There is a real number $R_+$ and a section $\psi_+$ of the normal bundle of $w_+$ with $|\psi_+| < \delta$ and such that $C_+$ is obtained by translating the portion of the image of $\exp_{w_+}(\psi_+)$ with $s \ge - 1 / \delta$ by $R_+$ in the $s$-direction. Here, $\exp_{w_+}$ is the exponential map on $w_+$ in the normal direction.
		
		\item
		There is a real number $R_-$ and a section $\psi_-$ of the normal bundle of $w_-$ with $|\psi_-| < \delta$ and such that $C_-$ is obtained by translating the portion of the image of $\exp_{w_-}(\psi_-)$ with $s \le 1 / \delta$ by $R_-$ in the $s$-direction. Here, $\exp_{w_-}$ is the exponential map on $w_-$ in the normal direction.
		
		\item
		We have $R_+ - R_- > 2 / \delta$.
		
		\item
		The positive boundary circles of $C_-$ agree with the negative boundary circles of $C_+$.
	\end{enumerate}
	Let $\mathcal G_{\mathbf c, \delta}(w_-, w_+)$ be the set of curves in $\mathcal C_{\mathbf c, \delta}(w_-, w_+)$ that have Fredholm index $2n - 3$ and ECH index $1 + \binom{n}{2}$.
\end{defn}

\begin{defn}
\label{defn:close-to-breaking-restricted-inverted}
	Let $\mathbf c \in \C^{n - 2}$ be a generic, admissible asymptotic restriction, let $(w_-, w_+) \in \mathcal S^\text{inv}_{\mathbf c}$, let $\delta > 0$, and choose a product metric on $\R \times Y_+$. Let $\mathcal C^\text{inv}_{\mathbf c, \delta}(w_-, w_+)$ be the set of $J$-holomorphic curves in $K_{\mathbf c}$ whose images can be decomposed into two surfaces with boundary $C_- \cup C_+$ such that the following hold.
	\begin{enumerate}
		\item
		There is a real number $R_+$ and a section $\psi_+$ of the normal bundle of $w_+$ with $|\psi_+| < \delta$ and such that $C_+$ is obtained by translating the portion of the image of $\exp_{w_+}(\psi_+)$ with $s \ge - 1 / \delta$ by $R_+$ in the $s$-direction. Here, $\exp_{w_+}$ is the exponential map on $w_+$ in the normal direction.
		
		\item
		There is a real number $R_-$ and a section $\psi_-$ of the normal bundle of $w_-$ with $|\psi_-| < \delta$ and such that $C_-$ is obtained by translating the portion of the image of $\exp_{w_-}(\psi_-)$ with $s \le 1 / \delta$ by $R_-$ in the $s$-direction. Here, $\exp_{w_-}$ is the exponential map on $w_-$ in the normal direction.
		
		\item
		We have $R_+ - R_- > 2 / \delta$.
		
		\item
		The positive boundary circles of $C_-$ agree with the negative boundary circles of $C_+$.
	\end{enumerate}
	Let $\mathcal G^\text{inv}_{\mathbf c, \delta}(w_-, w_+)$ be the set of curves in $\mathcal C^\text{inv}_{\mathbf c, \delta}(w_-, w_+)$ that have Fredholm index $2n - 3$ and ECH index $1 + \binom{n}{2}$.
\end{defn}

Now we give an analogue of \cite[Lemma 1.11]{HT1}.

\begin{lemma}
\label{lemma:curves-close-to-gluing}
	Let $\mathbf c \in \C^{n - 2}$ be a generic, admissible asymptotic restriction and let $(w_-, w_+) \in \mathcal S_{\mathbf c}$ (resp.\ $\mathcal S^\text{inv}_{\mathbf c}$). There exists a $\delta_0 > 0$ such that for any $\delta \in (0, \delta_0)$ and any sequence $\{ v_k \}$ in $\mathcal G_{\mathbf c, \delta}(w_-, w_+)$ (resp.\ $\mathcal G^\text{inv}_{\mathbf c, \delta}(w_-, w_+)$), the sequence $\{ [v_k] \}$ has a subsequence that converges either to a curve in $K_{\mathbf c} / \R$ or to the building $[w_-] \cup [w_+]$.
	
	There exists an $R_0 \gg 0$ such that for any $R > R_0$ and any sequence $\{ v_k \}$ in $\mathcal G_R$, the sequence $\{ [v_k] \}$ has a subsequence that converges either to a curve in $K_{\mathbf c} / \R$ or to a building $[w_-] \cup [w_+] \in (\mathcal M_n / \R) \times (\mathcal M_+ / \R)$.
\end{lemma}

\begin{proof}
	The lemma is immediate for sequences in $\mathcal G_R$ since $\mathcal G_R$ is a finite union of open subsets of $K_{\mathbf c}$ by \Cref{rmk:num-of-components}. So assume that $(w_-, w_+) \in \mathcal S_{\mathbf c}$; the proof when $(w_-, w_+) \in \mathcal S^\text{inv}_{\mathbf c}$ is similar. By the compactness results in \cite{BEHWZ}, a subsequence of $\{ [v_k] \}$ converges to an SFT building $[w_1] \cup \cdots \cup [w_\ell]$, where the levels go from bottom to top as we read from left to right.
	
	If $\ell > 1$, then $[w_1]$ must contain $[w_-]$ and $[w_\ell]$ must contain $[w_+]$. There are no other levels by \Cref{lemma:supersimple-fredholm-index}, so $\ell = 2$. There are no other components of $[w_1]$ or $[w_2]$ by \Cref{lemma:supersimple-fredholm-index} since the negative orbit set of $[w_-]$ is $\bm \beta$ and the positive orbit set of $[w_+]$ is $\bm \alpha$. Thus, $[w_1] = [w_-]$, and $[w_2] = [w_+]$.
	
	If $\ell = 1$, then $[w_1]$ contains either one component with Fredholm index $2n - 3$ or one component with Fredholm index $1$ and another with Fredholm index $2n - 4$. There are no other components of $[w_1]$ by \Cref{lemma:supersimple-fredholm-index} since the negative orbit set of $[w_-]$ is $\bm \beta$ and the positive orbit set of $[w_+]$ is $\bm \alpha$. By continuity, $[w_1] \in K_{\mathbf c} / \R$.
\end{proof}

Finally, we give an analogue of \cite[Definition 1.12]{HT1}.

\begin{defn}
\label{defn:gluing-pair-open-set}
	Given an $R > R_0$, by \Cref{lemma:curves-close-to-gluing} there is an open subset $U \subset K_{\mathbf c} / \R$ such that $\mathcal G_{R'} \subset U \subset \mathcal G_R$ for some $R' > R$ and whose closure $\overline U$ in $K_{\mathbf c} / \R$ has finitely many endpoints.
	
	Let $\mathbf c \in \C^{n - 2}$ be a generic, admissible asymptotic restriction, let $(w_-, w_+) \in \mathcal S_{\mathbf c}$ (resp.\ $\mathcal S^\text{inv}_{\mathbf c}$), and let a $\delta \in (0, \delta_0)$. By \Cref{lemma:curves-close-to-gluing}, there is an open set $U_{\mathbf c}(w_-, w_+) \subset K_{\mathbf c} / \R$ (resp.\ $U^\text{inv}_{\mathbf c}(w_-, w_+)$) such that $\mathcal G_{\mathbf c, \delta'}(w_-, w_+) \subset U_{\mathbf c}(w_-, w_+) \subset \mathcal G_{\mathbf c, \delta}(w_-, w_+)$ (resp.\ $\mathcal G^\text{inv}_{\mathbf c, \delta'}(w_-, w_+) \subset U^\text{inv}_{\mathbf c}(w_-, w_+) \subset \mathcal G^\text{inv}_{\mathbf c, \delta}(w_-, w_+)$) for some $\delta' \in (0, \delta)$ and whose closure $\overline{U}_{\mathbf c}(w_-, w_+)$ (resp.\ $\overline{U}^\text{inv}_{\mathbf c}(w_-, w_+)$) in $K_{\mathbf c}$ has finitely many endpoints.
\end{defn}

%===================================
% SECTION 8.2
% TRUNCATION AND THE COBORDISM MAP
%===================================

\subsection{Truncation and the cobordism map}
\label{subsec:truncation-cob-map}

We now truncate $K_{\mathbf c}$ to obtain a compact $1$-manifold with boundary, which we use to prove \Cref{thm:main-thm}. We begin with an analogue of \cite[Lemma 7.23]{HT1}. In the following proof, we call covers of trivial cylinders \textbf{connectors}; an unbranched cover of a trivial cylinder is called a \textbf{trivial connector}, while a branched cover is called a \textbf{non-trivial connector}. A cover of $\R \times \beta_0$ is called a connector \textbf{over} $\beta_0$. A $J$-holomorphic curve that is not a cover of a trivial cylinder is called a \textbf{non-connector}.

\begin{lemma}
\label{lemma:sequences-in-pre-image}
	Let $\mathbf c \in \C^{n - 2}$ be a generic, admissible asymptotic restriction. Any sequence $\{ v_k \}$ in $K_{\mathbf c}$ has a subsequence that converges to a curve in $K_{\mathbf c} / \R$ or to a $2$-level building $[w_-] \cup [w_+]$ in $(\mathcal M_n / \R) \times (\mathcal M_+ / \R)$, $\mathcal S_{\mathbf c} / \R$, or $\mathcal S^\text{inv}_{\mathbf c} / \R$.
\end{lemma}

\begin{proof}
	By \cite[Lemma 5.11]{H3}, we may pass to a subsequence such that every curve is in the same relative homology class after projecting to $Y_+$. By \cite[Corollary 6.10]{H2}, there is an upper bound on the genus of a somewhere injective curve that depends only on its relative homology class. Hence, we may pass to a further subsequence such that every curve has the same genus. By the compactness results in \cite{BEHWZ}, a further subsequence, which we also denote by $\{ [v_k] \}$, converges to an SFT building $[w_1] \cup \cdots \cup [w_\ell]$, where the levels go from bottom to top as we read from left to right. If $\ell = 1$, then by continuity $[w_1] \in K_{\mathbf c} / \R$. So assume that $\ell > 1$.
	
	Give the negative ends of $[w_1]$ at $\beta_0$ the labels and asymptotic markers induced from the sequence $\{ [v_k] \}$. Let $w_1$ be a representative of the class $[w_1]$. For each $i \in I$, there are two possibilities: $\ev_i(w_1) = 0 \in \C$ or $\ev_i(w_1) \neq 0$.
	
	If $\ev_i(w_1) = 0$, then there is a sequence of translates $v_k'$ of the curves $v_k$ by distances $a_k$ in the $\R$-direction such that $\ev_i(v_k') \to 0$. Since $\ev_i(v_k) = c_i$ for all $k$, \Cref{rmk:evmap-translate} implies that $a_k \to \infty$. Then for all $d > 2$, we have $\ev_i^d(w_1) = (0, \ldots, 0)$. If the component of $w_1$ containing the $i^\text{th}$ negative end is somewhere injective, then it is a trivial cylinder by \Cref{rmk:higher-order-analogue}; if the component is multiply covered, then it must be a trivial connector.
	
	If $\ev_i(w_1) \neq 0$, then we claim that some translate $w_1'$ of $w_1$ satisfies $\ev_i(w_1') = c_i$. If not, then there is some constant $C > 0$ such that every point in $\{ t \ev_i(w_1) \, | \, t > 0 \}$ is a distance at least $C$ away from $c_i$ in the standard Euclidean metric on $\C$. When $j$ is sufficiently large, we can glue appropriate representatives $\mathfrak w_{1, j}, \ldots, \mathfrak w_{\ell, j}$ of the classes $[w_1], \ldots, [w_\ell]$ to get a curve $\mathfrak w_{j}$ that represents the class $[v_k]$ and such that $\ev_i(\mathfrak w_{j}) = c_i$. It follows that for any $\delta > 0$, we have $\left| \ev_i(\mathfrak w_{1, j}) - c_i \right| < \delta$ when $j$ is sufficiently large, and we have a contradiction.
	
	Now we claim that $\ell = 2$ and that one of (1) $[w_1] \in \mathcal M_n / \R$, (2) $([w_1], [w_2]) \in \mathcal S_{\mathbf c} / \R$, or (3) $([w_1], [w_2]) \in \mathcal S^\text{inv}_{\mathbf c} / \R$ is true.
	
	First, assume that $\ev_I(w_1) \neq \mathbf 0$. A small modification of the argument above shows that there is some translate $w_1'$ of $w_1$ such that $\ev_I(w_1') = \mathbf c$. Since $\mathbf c$ is admissible, it follows that the components of $w_-$ asymptotic to $\beta_0$ are somewhere injective. Since $\mathbf c$ is generic, \Cref{rmk:empty-pre-images} implies that $\ind(w_1) \ge 2n - 4$, and hence \eqref{eqn:gen-ech-ineq} implies that $I(w_1) \ge \binom{n}{2}$. Thus, $\ind(w_2) \ge 1$, $I(w_2) \ge 1$, $\ind(w_1) \ge 2n - 4$, and $I(w_1) \ge \binom{n}{2}$, and by additivity of the Fredholm and ECH indices, we know that $\ind(w_1) + \ind(w_2) = 2n - 3$ and $I(w_1) + I(w_2) = 1 + \binom{n}{2}$. It follows that $\ind(w_2) = I(w_2) = 1$, $\ind(w_1) = 2n - 4$, and $I(w_2) = \binom{n}{2}$, so $(w_1, w_2) \in \mathcal S_{\mathbf c}$ and we are in case (2).
	
	Next, assume that $\ev_I(w_1) = \mathbf 0$ and that all connector components in the building are trivial. Thus, $w_1$ contains $n - 3$ copies of $\R \times \beta_0$ and one connected, unbranched, $3$-fold cover of $\R \times \beta_0$. The component of $w_1$ containing the negative end labeled $i$ is paired with some negative end of $w_2$, which we also label $i$. If the component of $w_2$ containing the negative end labeled $i$ is a trivial connector, it is paired with some negative end of $w_3$, which we also label $i$. Proceeding in this way, we reach a level $w_{\nu_i}$ and an end, which we label $i$, that is paired with a trivial connector in $w_{\nu_i - 1}$ and such that $\ev_i(w_{\nu_i}) \neq 0$. Such a level $w_{\nu_i}$ exists because none of the curves $v_k$ contains a component mapping to $\R \times \beta_0$.
	
	We claim that $\nu_i = 2$ for all $i$. If not, then either $\nu_i = m > 2$ for all $i$ or $\nu_i < \nu_r$ for some $i$ and $r$. In the first case, $\ind(w_m) < 2n - 4$ and there are translates $w_{\nu_1}', \ldots, w_{\nu_{n - 2}}'$ such that $\ev_i(w_{\nu_i}) = c_i$ for all $i$. Then the disjoint union $w = w_{\nu_1}' \sqcup \cdots \sqcup w_{\nu_{n - 2}}'$ is a somewhere injective curve with $\ind(w) < 2n - 4$ and $\ev_I(w) = \mathbf c$, contradicting \Cref{rmk:empty-pre-images}. In the second case, for any $\delta > 0$ there are translates $v_k'$ of $v_k$ by distances $a_k$ in the $\R$-direction when $k$ is sufficiently large so that
	\begin{equation}
	\label{eqn:case-three-contradiction}
	| \ev_i(v_k') - \ev_i(w_{\nu_i}) | < \delta \quad \text{and} \quad | \ev_r(v_k') | < \delta.
	\end{equation}
	Recall that if $\lambda_1$ is the smallest positive asymptotic eigenvalue of $\beta_0$ and $\widetilde{\lambda}_1$ is the smallest positive asymptotic eigenvalue of $\beta_0^3$, then $\ev_i(v_k') = e^{- \widetilde{\lambda}_1 a_i / 3} c_i$ if $i = 1$ and $\ev_i(v_k) = e^{- \lambda_1 a_i} c_i$ if $i > 1$. Thus, when $\delta$ is sufficiently small, the conditions in \eqref{eqn:case-three-contradiction} contradict the assumption that $\ev_I(v_k) = \mathbf c$.
	
	By a previous argument, some translate $w_2'$ of $w_2$ satisfies $\ev_I(w_2') = \mathbf c$. It follows that $w_2$ is somewhere injective, $\ind(w_2) \ge 2n - 4$, and $I(w_2) \ge \binom{n}{2}$. Since $\ind(w_1) \ge 1$ and $I(w_1) \ge 1$, we see that $\ind(w_1) = I(w_1) = 1$, $\ind(w_2) = 2n - 4$, and $I(w_2) = \binom{n}{2}$. Since $w_1$ contains $n - 3$ copies of $\R \times \beta_0$ and one unbranched cover of $\R \times \beta_0$ of multiplicity $3$, it follows that $([w_1], [w_2]) \in \mathcal S^\text{inv}_{\mathbf c}$, and we are in case (3).
	
	Now assume that $\ev_I(w_1) = \mathbf 0$ and that the building contains at least one non-trivial connector. Let $\mathfrak v_{1, 1}, \ldots, \mathfrak v_{1, m_1}$ denote the components of $w_1$ that are non-trivial connectors. Let $\widetilde I \subset I$ be the labels of the negative ends of $w_1$ that are contained in trivial connectors over $\beta_0$. As before, for each $j \in \widetilde I$, the component containing the negative end labeled $j$ is paired with some negative end of $w_2$, which we also label $j$. Proceeding as above, we reach a level $w_{\nu_j}$ and an end, which we label $j$, that is paired with a trivial connector over $\beta_0$ in $w_{\nu_j - 1}$ and such that either $\ev_j(w_{\nu_j}) \neq 0$ or the end is contained in a non-trivial connector. Let $w_{k_1}, \ldots, w_{k_r}$ be the levels containing non-connector components reached by the above procedure, and let $w_{\ell_1}, \ldots, w_{\ell_e}$ be the levels containing non-trivial connector components components reached by the above procedure. By an argument from case (2), we must have $r = 1$. Let $\mathfrak w_1$ denote the union of the relevant non-connector components in $w_{k_1}$. The curve $\mathfrak w_1$ is somewhere injective by an argument from case (2). For each $i = 1, \ldots, e$, let $\mathfrak v_{i, 1}, \ldots, \mathfrak v_{i, d_i}$ be the relevant components of the non-trivial connectors in $w_{\ell_i}$. Let $I' \subset \widetilde I$ be the subset of indices where we reach a non-connector in the above procedure.
	
	Recall the definitions of $\mathcal M^{g, k, 1}(1, \ldots, 1)$ and $\mathcal M^{g, k, 2}(1, \ldots, 1)$ from \Cref{notation:non-gluing}.
	
	\begin{claim}
	\label{claim:rule-out-gluings-1}
		In the above setup, for each $i = 1, \ldots, e$ and each $j = 1, \ldots, d_i$, there is some $k_{i, j} \ge 2$ such that $\mathfrak v_{i, j}$ is in $\mathcal M^{g_{i, j}, k_{i, j}, 1}(1, \ldots, 1)$ for some $g_{i, j} > 0$ or $\mathcal M^{g_{i, j}, k_{i, j}, 2}(1, \ldots, 1)$ for some $g_{i, j} \ge 0$.
	\end{claim}
	
	\begin{proof}[Proof of \Cref{claim:rule-out-gluings-1}]
		Let $[u]$ be a $J$-holomorphic curve obtained by gluing $[w_1] \cup \cdots \cup [w_\ell]$, where we take the gluing parameters to be large. Over any cylindrical portion of the domain of $u$ where $u$ is close to and graphical over $\R \times \beta_0$, $u$ can be written in cylindrical coordinates $(s, t)$ as $u(s, t) = (s, t, \widetilde u(s, t))$, where $\widetilde u(s, t)$ has a Fourier-type expansion
		\begin{equation*}
			\widetilde u(s, t)
			=
			\sum_{i \neq 0} c_i e^{-\lambda_i s} f_i(t).
		\end{equation*}
		Perturb $J$ as in \cite[Lemma 3.4.3]{BH2} so that $c_1 \neq 0$ over each such cylindrical portion and such that the coefficients of $f_1$ are distinct on distinct cylindrical pieces. 
		
		Pre-glue the building $[w_1] \cup \cdots \cup [w_\ell]$, with the exception of $\mathfrak v_{i, j}$ and the trivial connectors below it, to a curve $v_+$. There is a section $\psi_-$ of the normal bundle of $\mathfrak v_{i, j}$ defined on the portion of the domain of $\mathfrak v_{i, j}$ obtained by truncating the positive ends such that the perturbation of $\mathfrak v_{i, j}$ by $\psi_-$ (using an appropriate exponential map) coincides with $u$.  There is also a section $\psi_+$ of the normal bundle of (a translation of) $v_+$ on the portion of the domain of $v_+$ obtained by truncating the negative ends asymptotic to $\beta_0$ such that the perturbation of $v_+$ by $\psi_+$ coincides with $u$.
		
		By the proof of \cite[Claim 8.8.3]{BH1}, we can extend $\psi_-$ and $\psi_+$ to sections $\widetilde \psi_-$ and $\widetilde \psi_+$ defined over the whole domain of $\mathfrak v_{i, j}$ and $v_+$, respectively, that formally satisfy the necessary equation for the obstruction section over the moduli space of branched covers containing $\mathfrak v_{i, j}$ has a zero. Our perturbation of $J$ above ensures that the $\widetilde \psi_\pm$ also formally satisfy the necessary equation for the linearized obstruction section $\mathfrak s_1$ over said moduli space to have a zero. The claim now follows from \Cref{prop:non-gluing} and \Cref{prop:non-gluing-aux}.
	\end{proof}
	
	By a previous argument, $\ev_{I'}(\mathfrak w_1) = (c_j)_{j \in I'}$; hence, $\mathfrak w_1$ is somewhere injective, and $\ind(\mathfrak w_1) \ge 2 | I' |$ by \Cref{rmk:empty-pre-images}. If no $\mathfrak v_{i, j}$ has a negative end with multiplicity $3$, then
	\begin{equation*}
		|I'| + \sum_{i = 1}^e \sum_{j = 1}^{d_i} k_{i, j} = n - 2,
	\end{equation*}
	and
	\begin{align*}
		2n - 3
			&\ge
				\ind(\mathfrak w_1) + \sum_{i = 1}^e \sum_{j = 1}^{d_i} \ind(\mathfrak v_{i, j}) \\
			&\ge
				2 |I'| + 2 \sum_{i = 1}^e \sum_{j = 1}^{d_i} (k_{i, j} - 1 + g_{i, j}) \\
			&=
				2n - 4 + 2 \sum_{i = 1}^e \sum_{j = 1}^{d_i} (g_{i, j} - 1).
	\end{align*}
	Hence, $g_{i, j} = 1$ for all $i$ and $j$, and either $\ind(\mathfrak w_1) = 2 |I'| + 1$ or there exists an additional component $\mathfrak w$ with $\ind(\mathfrak w) = 1$ in some level of the building. In the latter case, $\mathfrak w$ is somewhere injective. In either case, we can glue every somewhere injective curve in the building $[w_1] \cup \cdots \cup [w_\ell]$ and produce a somewhere injective $J$-holomorphic curve $\mathfrak u$ in $\R \times Y_+$ with $I(\mathfrak u) = 1 + \binom{n}{2}$, $\Delta(\mathfrak u) = 1 + \binom{n - 2}{2}$, and $\ind(\mathfrak u) < 2n - 4$. As in the proof of \Cref{claim:rule-out-gluings-1}, we can perturb $J$ to a new, generic $J'$ to ensure that the negative ends of $\mathfrak u$ asymptotic to $\beta_0$ all non-degenerate and non-overlapping. Then equality must hold in \eqref{eqn:gen-ech-ineq}, and we have a contradiction.
	
	 Now assume that some, and hence exactly one, $\mathfrak v_{a, b}$ has a negative end with multiplicity $3$. Then
	 \begin{equation*}
		|I'| + \sum_{i = 1}^e \sum_{j = 1}^{d_i} k_{i, j} = n,
	\end{equation*}
	and
	\begin{align*}
		2n - 3
			&\ge
				\ind(\mathfrak w_1) + \sum_{i = 1}^e \sum_{j = 1}^{d_i} \ind(\mathfrak v_{i, j}) \\
			&\ge
				2 |I'| + + 2(k_{a, b} - 2 + g_{a, b}) + 2 \sum_{i \neq a, j \neq b} (k_{i, j} - 1 + g_{i, j}) \\
			&=
				2n - 4 + 2 g_{a, b} + 2 \sum_{i \neq a, j \neq b} (g_{i, j} - 1),
	\end{align*}
	so $g_{a, b} = 0$ and $g_{i, j} = 1$ if $(i, j) \neq (a, b)$. As before, either $\ind(\mathfrak w_1) = 2 |I'| + 1$ or there exists an additional component $\mathfrak w$ with $\ind(\mathfrak w) = 1$ in some level of the building, and in the latter case, $\mathfrak w$ is somewhere injective. We claim that in either case, $i = 1$, $d_1 = 1$, and $k_{1, 1} = n$. It follows that the only branched cover in the building is a curve in $\mathcal M_n$ and that we are in case (1).
	
	To prove the above assertion, glue every curve in the building $[w_1] \cup \cdots \cup [w_\ell]$ except for $\mathfrak v_{a, b}$ to produce a somewhere injective $J$-holomorphic curve $\mathfrak u$ in $\R \times Y_+$ with $I(\mathfrak u) = 1 + \binom{n}{2}$ and $\Delta(\mathfrak u) = \binom{n}{2}$. If $i > 1$, it follows that $\ind(\mathfrak u) > 1$, contradicting \eqref{eqn:gen-ech-ineq}. If $d_1 > 1$, we have the same contradiction. If $k_{1, 1} < n$, then $|I'| > 0$, and hence $\ind(\mathfrak u) > 1$, again yielding a contradiction. \qedhere	
\end{proof}

\begin{defn}
\label{defn:truncation}
	Let $\mathbf c \in \C^{n - 2}$ be a generic, admissible asymptotic restriction. Choose an open set $U$ as in \Cref{defn:gluing-pair-open-set}. For each pair $(w_-, w_+) \in \mathcal S_{\mathbf c}$ (resp.\ $\mathcal S^\text{inv}_{\mathbf c}(w_-, w_+)$), choose a $\delta > 0$ and an open set $U_{\mathbf c}(w_-, w_+)$ (resp.\ $U^\text{inv}_{\mathbf c}(w_-, w_+)$) as in \Cref{defn:gluing-pair-open-set}. The \textbf{truncation} of $K_{\mathbf c}$ is the set
	\begin{equation*}
		K_{\mathbf c}'
		=
		K_{\mathbf c}
		\setminus
		\left( U \sqcup \bigsqcup_{(w_-, w_+) \in \mathcal S_{\mathbf c}} U_{\mathbf c}(w_-, w_+) \sqcup \bigsqcup_{(w_-, w_+) \in \mathcal S^\text{inv}_{\mathbf c}} U^\text{inv}_{\mathbf c}(w_-, w_+) \right).
	\end{equation*}
	By \Cref{lemma:sequences-in-pre-image}, $K_{\mathbf c}'$ is compact.
\end{defn}

\begin{defn}
\label{defn:truncation-map}
	Let $\mathbf c \in \C^{n - 2}$ be a generic, admissible asymptotic restriction. Let $\widetilde\bdry K_{\mathbf c}'$ be the set of points of $\bdry K_{\mathbf c}'$ that lie in $\mathcal G_R$. The \textbf{truncation map} is the map
	\begin{equation*}
		\mathcal B_{\mathbf c}
		\colon
		\bdry K_{\mathbf c}' / \R
		\to
		(\mathcal S_{\mathbf c} / \R)
		\sqcup
		(\mathcal S^\text{inv}_{\mathbf c} / \R)
		\sqcup
		(\widetilde\bdry K_{\mathbf c}' / \R)
	\end{equation*}
	that sends a curve in $\widetilde\bdry K_{\mathbf c}' / \R$ to itself and every other curve in $\bdry K_{\mathbf c}'$ to the $2$-level building into which it is close to breaking.
\end{defn}

\begin{lemma}
\label{lemma:multiple-gluings-cancel}
	Let $\mathbf c \in \C^{n - 2}$ be a generic, admissible asymptotic restriction. If $([w_-], [w_+])$ is an element of $\mathcal S_{\mathbf c} / \R$ or $\mathcal S^\text{inv}_{\mathbf c}$, then the mod $2$ count of points in $\mathcal B_{\mathbf c}^{-1}([w_-], [w_+])$ is $0$ if the intermediate orbit set $\bm \gamma$ is not a generator of the ECH chain complex for $(Y_+, \lambda_+)$ and is $1$ if $\bm \gamma$ is a generator.
\end{lemma}

\begin{proof}
	We use the quotient evaluation map from \cite[Section 6]{BH1}. Let $\lambda_1$ be the smallest positive asymptotic eigenvalue of $\beta_0$, and let $\widetilde{\lambda}_1$ be the smallest positive asymptotic eigenvalue of $\beta_0^3$ If $w \in \mathcal M$, $\ev_I(w) = (a_1, \ldots, a_{n - 2})$, and $w'$ is obtained by translating $w$ a distance $s$ in the $\R$-direction, then
	\begin{equation*}
		\ev_I(w') = (a_1 e^{- \widetilde{\lambda}_1 s / 3}, a_2 e^{- \lambda_1 s}, \ldots, a_{n - 2} e^{- \lambda_1 s}).
	\end{equation*}
	The proof of \Cref{lemma:hut-eqn-4-imp-neg} shows that there are no curves $w \in \mathcal M$ with $\ev_I(w) = \bm 0$, since such curves would necessarily have $I(w) > 1 + \binom{n}{2}$. Thus, the evaluation map descends to a smooth map on the quotient $\overline{\ev}_I \colon \mathcal M / \R \to (\C^{n - 2} \setminus \{ \bm 0 \}) / \R^+ \cong S^{2n - 5}$. Given a pair $([w_-], [w_+]) \in \mathcal S_{\mathbf c} / \R$ and a neighborhood $U_-$ of $[w_-]$ in $\widetilde{\mathcal V}_{n, \bm \gamma, \bm \beta, 2n - 4}$, we can identify an open set in $\mathcal M / \R$ with $\mathcal U = [R, \infty) \times U_-$ such that as the parameter $T \in [R, \infty)$ goes to infinity, the curve breaks into a two-level building in $U_- \times \{ [w_+] \}$. The map $\overline{\ev}_I$ extends smoothly to the broken curves on the boundary of $\mathcal U$. If $U_-$ is sufficiently small, $\overline{\ev}_I$ is a submersion on $U_-$ and $\mathcal U$. Hence, for every gluing of $[w_-]$ with $[w_+]$, there is a unique end of $K_{\mathbf c}$ that is compactified by adding the building $[w_-] \cup [w_+]$. In particular, if we choose $\delta$ to be sufficiently small when truncating $K_{\mathbf c}$ and $U_-$ is sufficiently small, there is a unique endpoint of $K_{\mathbf c}'$ in $\mathcal U$. The proof when $([w_-], [w_+]) \in \mathcal S^\text{inv}_{\mathbf c}$ is similar.
\end{proof}

\begin{proof}[Proof of \Cref{thm:main-thm}]
	The mod $2$ count of points in $\widetilde \bdry K_{\mathbf c}'$ is equal to the mod $2$ count of non-canceling buildings described in \Cref{thm:degen-class-thm} such that the $\widehat X$-level has an $n$-fold degenerate cover of a plane at $\beta_0$. Note that in this proof, we denote the negative orbit set of the symplectization level by $\bm \beta$ and the negative orbit set of the cobordism level by $\bm \gamma$. By \Cref{lemma:multiple-gluings-cancel}, this count is equal to the count of gluings of pairs $([w_-], [w_+])$ in $\mathcal S_{\mathbf c} / \R$ or $\mathcal S^\text{inv}_{\mathbf c} / \R$ whose intermediate orbit set is a generator of the ECH complex. Pairs in $\mathcal S_{\mathbf c} / \R$ contribute to the count for $\Phi_{X, \lambda, J, \mathbf c} \circ \bdry$.
	
	We claim that pairs $([w_-], [w_+]) \in \mathcal S^\text{inv}_{\mathbf c} / \R$ correspond to ECH buildings that contribute to the count for $\bdry \circ \Phi_{X, \lambda, J, \mathbf c}$. We can glue $w_-$ to the curve $u_0$ in $\widehat X$ from \Cref{thm:degen-class-thm} since the non-trivial component of $w_-$ has no negative ends at $\beta_0$. Consider the subset $\mathcal N$ of the moduli space $\mathcal M_X^{1, 1 - n(n + 1) / 2}(\bm \beta, \bm \gamma)$ consisting of curves that contain a degenerate $n$-fold cover of a plane with a positive puncture at $\beta_0$. The multiplicity of any orbit in $\bm \gamma$ and $\bm \beta$ besides $\beta_0$ is $1$, so any component of a curve in $\mathcal N$ that is not the degenerate cover is somewhere injective. The degenerate covers are cut out transversely by Wendl's automatic transversality criterion \cite[Theorem 4.2.1]{BH1}. Thus, $\mathcal N$ is a transversely cut out $1$-manifold. The boundary points of its SFT compactification are two-level buildings of the form $[v_-] \cup v_0$ or $v_0 \cup [v_1]$, where $v_0$ is in $\widehat X$ and $v_{\pm 1}$ are in $\R \times Y_-$, such that $\ind(v_0) = 0$, $\ind(v_{\pm 1}) = 1$, $I(v_0) = - \binom{n}{2}$, and $I(v_{\pm 1}) = 1$. There are finitely many boundary points. Buildings of the form $v_0 \cup [v_1]$ correspond to pairs in $\mathcal S^\text{inv}_{\mathbf c} / \R$, while buildings of the form $[v_{-1}] \cup v_0$ contribute to the count for $\bdry \circ \Phi_{X, \lambda, J, \mathbf c}$ when $v_0$ is paired with $[w_+]$. The map $\Phi_{X, \lambda, J, \mathbf c}$ is independent of the choice of $\mathbf c$ by \Cref{rmk:num-of-components}. The proof of \Cref{thm:main-thm} is now complete when the gluing problem is in the case $n \ge 3$.
	
	All that remains is the case $n = 2$. When we glue branched covers in $\modsp{1, 1}{2}$ to $u_1$, the result lives in the moduli space $\mathcal M_{\R \times Y_+}^{2, 2}(\bm \alpha, \bm \beta)$. Then $\mathcal M_{\R \times Y_+}^{2, 2}(\bm \alpha, \bm \beta) / \R$ has dimension $1$, and all endpoints must be two-level buildings $[w_-] \cup [w_+]$ with $\ind(w_-) = \ind(w_+) = 1$ and $I(w_-) = I(w_+) = 1$. There is no evaluation map in this case.
\end{proof}

%======================%
% APPENDIX A			%
% EXISTENCE OF MODELS	%
%======================%

\numberwithin{equation}{section}
\numberwithin{thm}{section}

\appendix
\section{Existence of Models}
\label{sec:model-existence-proof}

In this appendix, we prove \Cref{thm:models}. Recall that our branched covers have multiplicity $1$ positive punctures $p_1, \ldots, p_n$, a multiplicity $3$ negative puncture $q_1$, and multiplicity $1$ negative punctures $q_2, \ldots, q_{n - 2}$. An order $m$ model, if it exists, is necessarily given by
\begin{equation*}
	g(z)
	=
	\sum_{i = 1}^n \sum_{\ell = 1}^m \frac{ e^{i \ell \theta} e^{- \ell T} r_i^\ell \alpha_{i, \ell} }{ (z - p_i)^\ell },
\end{equation*}
As before, set $h(\zeta) = g(\zeta^{-1})$ and note that $h$ has a removable singularity at $\zeta = 0$ and vanishes there. The function $g$ is an order $m$ model associated to a branched cover $u$ if and only if
\begin{equation}
\label{eqn:zero-set-requirement}
	 \quad h'(0) = 0 \quad \text{and} \quad g(q_2) = \cdots = g(q_{n - 2}) = 0,
\end{equation}
By \eqref{eqn:H_i}, \eqref{eqn:h'(0)=0}, \eqref{eqn:tau_i}, and \eqref{eqn:r_i}, the equations \eqref{eqn:zero-set-requirement} are equivalent to
\begin{equation}
\label{eqn:intermediate-zero-set-requirement}
	\sum_{i = 1}^n e^{i \theta} e^{-T} r_i \alpha_{i, 1}											=	0
	\quad \text{and} \quad
	\sum_{i = 1}^n \sum_{\ell = 1}^m \frac{ e^{i \ell \theta} e^{- \ell T} r_i^\ell \alpha_{i, \ell} }{ (q_{k} - p_i)^\ell }  =	0
\end{equation}
for $k = 2, \ldots, n - 2$.

\begin{notation}
\label{notation:model-matrix}
	Define $(n - 2) \times n$ matrices
	\begin{equation*}
		A_1
		=
		\begin{pmatrix}
			1				&	\cdots	&	1				\\
			(q_2 - p_1)^{-1}		&	\cdots	&	(q_2 - p_n)^{-1}		\\
			\vdots			&	\ddots	&	\vdots			\\
			(q_{n-2} - p_1)^{-1}	&	\cdots	&	(q_{n-2} - p_n)^{-1}
		\end{pmatrix}
	\end{equation*}
	and, for $\ell > 1$,
	\begin{equation*}
		A_\ell
		=
		\begin{pmatrix}
			0				&	\cdots	&	0				\\
			(q_2 - p_1)^{-\ell}	&	\cdots	&	(q_2 - p_n)^{-\ell}	\\
			\vdots			&	\ddots	&	\vdots			\\
			(q_{n-2} - p_1)^{-\ell}	&	\cdots	&	(q_{n-2} - p_n)^{-\ell}
		\end{pmatrix}.
	\end{equation*}
	Define the $(n - 2) \times nm$ block matrix
	\begin{equation*}
		A =
		\begin{pmatrix}
			A_1	&	A_2	&	\cdots	&	A_m
		\end{pmatrix}.
	\end{equation*}
	Finally, define the $n \times 1$ column vectors
	\begin{equation*}
		\bm \alpha_\ell =
		e^{i \ell \theta} e^{- \ell T}
		\left(		r_1^\ell \alpha_{1, \ell},	\cdots,	r_n^\ell \alpha_{n, \ell}	\right)^t
	\end{equation*}
	for all $\ell \ge 1$ and the $mn \times 1$ column vector
	\begin{equation*}
		\bm \alpha =
		\left(	\bm \alpha_1^t,		\bm \alpha_2^t,	\cdots,	\bm \alpha_m^t	\right)^t,
	\end{equation*}
	where a superscript $t$ indicates the transpose of a matrix.
\end{notation}

\begin{proof}[Proof of \Cref{thm:models}]
The equations \eqref{eqn:intermediate-zero-set-requirement} hold if and only if $\bm \alpha_\ell$ is in the nullspace of $A_\ell$ for all $\ell = 1, \ldots, m$, i.e., if and only if $\bm \alpha$ is in the nullspace of the block matrix $A$. We relate \eqref{eqn:intermediate-zero-set-requirement} to the equations \eqref{eqn:order-m-zero-set-eqns} defining $\mathcal Z_m$ by performing row operations on the matrix $A$ to put it in echelon form, at which point the block $A_\ell$ has been converted to the matrix $B_\ell$ from \eqref{eqn:derivative-matrix}. We first describe the row-reduction algorithm as a sequence of steps. Fix $\ell$ and write our starting matrix $A_\ell$ as a matrix of row vectors:
\begin{equation*}
	A_\ell
	=
	\begin{pmatrix}
		\text{---}	&	\mathbf A_1		&	\text{---}		\\
				&	\vdots			&				\\
		\text{---}	&	\mathbf A_{n - 2}	&	\text{---}
	\end{pmatrix}.
\end{equation*}
At every step of the row-reduction process, we will refer to the matrix obtained at that step by $\widetilde A$ and the rows of $\widetilde A$ by $\widetilde A_1, \ldots, \widetilde A_{n - 2}$. The rows of the original matrix $A_\ell$ will always be denoted $\mathbf A_1, \ldots, \mathbf A_{n - 2}$.

\s

The $j^\text{th}$ step of our row-reduction algorithm, $j = 1, \ldots, n - 2$, is as follows. If $j > 1$, then for $i = 1, \ldots, j - 1$, replace the row $\widetilde A_i$ with
\begin{equation*}
	\widetilde A_i + \frac{q_j - p_j}{p_j - p_i} \widetilde A_j
\end{equation*}
and multiply the resulting row by $\frac{p_j - p_i}{q_j - p_i}$. If $j < n - 2$, then for $i = j + 1, \ldots, n - 2$, replace the row $\widetilde A_i$ with
\begin{equation*}
	\widetilde A_i - \frac{ q_j - p_j }{ q_i - p_j } \widetilde A_j
\end{equation*}
and multiply the resulting row by $\frac{ q_i - p_j }{ q_j - q_i }$. Finally, multiply rows $1, \ldots, j$ by $-1$.

\begin{notation}
\label{notation:extra-products}
	If $I \subset \{ 1 , \ldots, n - 2 \}$ is a set of indices, define
	\begin{equation*}
		\mathbb A_{I}(z)
		=
		\prod_{\substack{i = 1 \\ i \not\in I }}^{n - 2} (z - p_i)
		\quad \text{and} \quad
		B_{I}(z)
		=
		\prod_{\substack{i = 2 \\ i \not\in I }}^{n - 2} (z - q_i),
	\end{equation*}
	where an empty product is defined to be $1$. Note that
	\begin{equation*}
		\mathbb A_{\{i\}}(z)
		=
		\mathbb A_i(z)
		\quad
		\text{and}
		\quad
		B_{\{i\}}(z)
		=
		B_i(z)
	\end{equation*}
	in the notation from \Cref{notation:products}. For any $j \in \{ 1, \ldots, n - 2 \}$, define the set $I_j = \{ j, \ldots, n - 2 \}$.
\end{notation}

\begin{claim}
\label{claim:algorithm}
	After step $j$ of the row reduction, the $i^\text{th}$ row $\widetilde A_i$ of the resulting matrix $\widetilde A$ is as follows. For $1 \le i \le j$,
	\begin{equation*}
		(-1)^j
		\left[ \mathbf A_1- \sum_{k = 2}^j \frac{ \mathbb A_{ \{i\} \cup I_{j + 1} }(q_k) }{ B_{ \{k\} \cup I_{j + 1} }(q_k) } \mathbf A_k \right].
	\end{equation*}
	For $j + 1 \le i \le n - 2$,
	\begin{equation*}
		(-1)^j
		\left[ \mathbf A_1 - \frac{ \mathbb A_{I_{j + 1}}(q_i) }{ B_{I_{j + 1}}(q_i) } \mathbf A_i - \sum_{k = 2}^j \frac{ \mathbb A_{I_{j + 1}}(q_k) }{ (q_k - q_i) B_{\{ k \} \cup I_{j + 1}}(q_k) } \mathbf A_k \right].
	\end{equation*}
	In both cases, an empty sum is defined to be the zero row vector.
\end{claim}

\begin{proof}[Proof of \Cref{claim:algorithm}]
We proceed by induction on $j$. Note that, after step $1$, the matrix $\widetilde A$ is given by
\begin{equation*}
	\widetilde A
	=
	\begin{pmatrix}
		- \mathbf A_1								\\
		(q_2 - p_1) \mathbf A_2 - \mathbf A_1			\\
		\vdots									\\
		(q_{n - 2} - p_1) \mathbf A_{n - 2} - \mathbf A_1		
	\end{pmatrix}.
\end{equation*}

Now assume that the claim holds at step $j$. Then the $(j + 1)^\text{st}$ row of $\widetilde A$ is already in the correct form for step $j + 1$, and after that step, the other rows of $\widetilde A$ are as follows. For $i \le j$, the $i^\text{th}$ row is
\begin{align*}
	&(-1)^{j + 1} \left[ \mathbf A_1 - \sum_{k = 2}^j \frac{ \mathbb A_{\{i\} \cup I_{j + 1}}(q_k) }{ B_{\{k\} \cup I_{j + 1}}(q_k) } \mathbf A_k \right. \\
	&\quad\quad\quad
		\left. + \frac{ q_{j + 1} - p_{j + 1} }{ p_{j + 1} - p_i } \left( \mathbf A_1 - \sum_{k = 2}^{j + 1} \frac{ \mathbb A_{I_{j + 1}}(q_k) }{ B_{\{k\} \cup I_{j + 2}}(q_k) } \mathbf A_k \right) \right] \frac{ p_{j + 1} - p_i }{ q_{j + 1} - p_i } \\
	&\quad=
		(-1)^{j + 1} \Bigg[
		\mathbf A_1
		- \frac{ \mathbb A_{\{i\} \cup I_{j + 2}}(q_{j + 1}) }{ B_{I_{j + 1}}(q_{j + 1}) } \mathbf A_{j + 1} \\
	&\quad\quad\quad
		- \sum_{k = 2}^j \left[ \frac{ \mathbb A_{\{i\} \cup I_{j + 1}}(q_k) }{ (q_{j + 1} - p_i) B_{\{k\} \cup I_{j + 2}}(q_k) } \right. \\
	&\quad\quad\quad\quad\quad\quad
    		\left. \cdot \Big( (p_{j + 1} - p_i)(q_k - q_{j + 1}) + (q_{j + 1} - p_{j + 1})(q_k - p_i) \Big) \mathbf A_k \right] \Bigg] \\
	&\quad=
		(-1)^{j + 1} \Bigg[
		\mathbf A_1
		- \frac{ \mathbb A_{\{i\} \cup I_{j + 2}}(q_{j + 1}) }{ B_{I_{j + 1}}(q_{j + 1}) } \mathbf A_{j + 1} \\
	&\quad\quad\quad\quad\quad
		- \sum_{k = 2}^j \left[ \frac{ \mathbb A_{\{i\} \cup I_{j + 1}}(q_k) }{ (q_{j + 1} - p_i) B_{\{k\} \cup I_{j + 2}}(q_k) } (q_{j + 1} - p_i)(q_k - p_{j + 1}) \mathbf A_k \right] \Bigg] \\
	&\quad=
		(-1)^{j + 1} \Bigg[ \mathbf A_1 - \sum_{k = 2}^{j + 1} \frac{ \mathbb A_{\{i\} \cup I_{j + 2}}(q_{k + 1}) }{ B_{\{k + 1\} \cup I_{j + 3}}(q_{k + 1}) } \mathbf A_k \Bigg].
\end{align*}

\noindent For $j + 2 \le i \le n - 2$, the $i^\text{th}$ row is
\begin{align*}
	&(-1)^j \left[ \left( \mathbf A_1 - \frac{ \mathbb A_{I_{j + 1}}(q_i) }{ B_{I_{j + 1}}(q_i) } \mathbf A_i - \sum_{k = 2}^j \frac{ \mathbb A_{I_{j + 1}}(q_k) }{ (q_k - q_i) B_{\{k\} \cup I_{j + 1}}(q_k) } \mathbf A_k \right) \right. \\
	&\quad\quad\quad
	\left. - \frac{ q_{j + 1} - p_{j + 1} }{ q_i - p_{j + 1} } \left( \mathbf A_1 - \sum_{k = 2}^{j + 1} \frac{ \mathbb A_{I_{j + 1}}(q_k) }{ B_{\{k\} \cup I_{j + 2}}(q_k) } \mathbf A_k \right) \right] \frac{ q_i - p_{j + 1} }{ q_{j + 1} - q_i } \\
	&\quad=
	(-1)^j \Bigg[
	- \mathbf A_1
	+ \frac{ \mathbb A_{I_{j + 2}}(q_i) }{ B_{I_{j + 2}}(q_i) } \mathbf A_i
	+ \frac{ \mathbb A_{I_{j + 2}}(q_{j + 1}) }{ (q_{j + 1} - q_i) B_{I_{j + 1}}(q_{j + 1}) } \mathbf A_{j + 1} \\
	&\quad\quad\quad\quad
	- \sum_{k = 1}^j \left[ \frac{ \mathbb A_{I_{j + 1}}(q_k) }{ (q_{j + 1} - q_i)(q_k - q_i) B_{\{k\} \cup I_{j + 2}}(q_k) } \right. \\
	&\quad\quad\quad\quad\quad\quad
	\left. \cdot \Big( (q_{j + 1} - p_{j + 1})(q_k - q_i) - (q_i - p_{j + 1})(q_k - q_{j + 1}) \Big) \mathbf A_k \right] \Bigg] \\
	&\quad=
	(-1)^{j + 1} \Bigg[
	\mathbf A_1
	- \frac{ \mathbb A_{I_{j + 2}}(q_i) }{ B_{I_{j + 2}}(q_i) } \mathbf A_i
	- \frac{ \mathbb A_{I_{j + 2}}(q_{j + 1}) }{ (q_{j + 1} - q_i) B_{I_{j + 1}}(q_{j + 1}) } \mathbf A_{j + 1} \\
	&\quad\quad\quad
	- \sum_{k = 1}^j \frac{ \mathbb A_{I_{j + 1}}(q_k) }{ (q_{j + 1} - q_i)(q_k - q_i) B_{\{k\} \cup I_{j + 2}}(q_k) } (q_{j + 1} - q_i)(q_k - p_{j + 1}) \mathbf A_k \Bigg] \\
	&\quad=
	(-1)^{j + 1} \Bigg[
	\mathbf A_1
	- \frac{ \mathbb A_{I_{j + 2}}(q_i) }{ B_{I_{j + 2}}(q_i) } \mathbf A_i
	-\frac{ \mathbb A_{I_{j + 2}}(q_k) }{ (q_k - p_i) B_{\{k\} \cup I_{j + 2}}(q_k) } \mathbf A_k \Bigg].
\end{align*}

\noindent The claim follows by induction.
\end{proof}

After step $n - 2$ of the row reduction, multiply the matrix $\widetilde A$ by $(-1)^{n - 2}$. By \Cref{claim:algorithm}, the $i^\text{th}$ row of the resulting matrix is
\begin{equation*}
	\mathbf A_1
	-
	\sum_{k = 2}^{n - 2} \frac{ \mathbb A_i(q_k) }{ B_k(q_k) } \mathbf A_k.
\end{equation*}
The $j^\text{th}$ entry of this row is
\begin{align*}
	1 - \sum_{k = 2}^{n - 2} \frac{ \mathbb A_i(q_k) }{ B_k(q_k) } \frac{ 1 }{ q_k - p_j }
	&\quad \text{when $\ell = 1$ and} \\
	- \sum_{k = 2}^{n - 2} \frac{ \mathbb A_i(q_k) }{ B_k(q_k) } \frac{ 1 }{ (q_k - p_j)^\ell }
	&\quad \text{when $\ell > 1$.}
\end{align*}
In either case, the result is equal to
\begin{equation*}
	\frac{ 1 }{ (\ell - 1)! } \frac{ d^{\ell - 1} Q_i }{ d z^{\ell - 1} }(p_j)
\end{equation*}
by \Cref{rmk:partial-frac}, and we are done.
\end{proof}

%==============================%
% APPENDIX B					%
% DETERMINANT CALCULATIONS		%
%==============================%

\section{Determinant Calculations}
\label{sec:det-calc}

In this appendix, we prove \Cref{claim:det-reduction}. For notational simplicity, we abbreviate $E_{\ell, i} = E_\ell(p_1, \ldots, \widehat{p_i}, \ldots, p_n)$ and $E_{\ell, (i, j)} = E_\ell(p_1, \ldots, \widehat{p_i}, \ldots, \widehat{p_j}, \ldots, p_n)$. For compactness of notation, we also abbreviate $\alpha_i = \alpha_{i, 1}$. To reduce bookkeeping with signs, we will instead compute the determinant
\begin{equation*}
	\begin{vmatrix}
		1				&	\cdots	&	1			\\
		p_1				&	\cdots	&	p_n			\\
		p_1^2			&	\cdots	&	p_n^2		\\
		\vdots			&	\ddots	&	\vdots		\\
		p_1^{n - 2}		&	\cdots	&	p_n^{n - 2}	\\
		\alpha_1 E_{\ell, 1}	&	\cdots	&	\alpha_n E_{\ell, n}
	\end{vmatrix},
\end{equation*}
which differs from our desired determinant by a factor of $(-1)^{\binom{n}{2}}$.

\begin{lemma}
\label{lemma:det-prep}
If $n \ge 4$ and $0 \le \ell \le n - 1$, we have
\begin{equation*}
	\begin{vmatrix}
		1			&	\cdots	&	1			\\
		p_1			&	\cdots	&	p_n			\\
		p_1^2		&	\cdots	&	p_n^2		\\
		\vdots		&	\ddots	&	\vdots		\\
		p_1^{n - 2}	&	\cdots	&	p_n^{n - 2}	\\
		E_{\ell, 1}		&	\cdots	&	E_{\ell, n}
	\end{vmatrix}
	=
	\begin{cases}
		0,				&	0 \le \ell \le n - 2 \\
		(-1)^{\binom{n - 1}{2}} \Delta,	&	\ell = n - 1
	\end{cases}.
\end{equation*}
\end{lemma}

\begin{proof}
The proof is by induction on $n$. Note that $E_{\ell, k} - E_{\ell, 1} = (p_1 - p_k) E_{\ell - 1, (1, k)}$ for $k \ge 2$, so that
\begin{align*}
	&
	\begin{vmatrix}
		1			&	1			&	\cdots	&	1			\\
		p_1			&	p_2			&	\cdots	&	p_n			\\
		p_1^2		&	p_2^2		&	\cdots	&	p_n^2		\\
		\vdots		&	\vdots		&	\ddots	&	\vdots		\\
		p_1^{n - 2}	&	p_2^{n - 2}	&	\cdots	&	p_n^{n - 2}	\\
		E_{\ell, 1}		&	E_{\ell, 2}		&	\cdots	&	E_{\ell, n}
	\end{vmatrix}
	=
	\begin{vmatrix}
		1			&	1					&	\cdots	&	1			\\
		0			&	p_2 - p_1				&	\cdots	&	p_n - p_1			\\
		0			&	p_2(p_2 - p_1)			&	\cdots	&	p_n(p_n - p_1)		\\
		\vdots		&	\vdots				&	\ddots	&	\vdots		\\
		0			&	p_2^{n - 3}(p_2 - p_1)	&	\cdots	&	p_n^{n - 3}(p_n - p_1)	\\
		0			&	E_{\ell, 2} - E_{\ell, 1}	&	\cdots	&	E_{\ell, n} - E_{\ell, 1}
	\end{vmatrix} \\
	&\quad\quad\quad\quad\quad\quad\quad\quad\quad\quad=
	(-1)^n A_1(p_1)
	\begin{vmatrix}
		1				&	\cdots	&	1				\\
		p_2				&	\cdots	&	p_n				\\
		\vdots			&	\ddots	&	\vdots			\\
		p_2^{n - 3}		&	\cdots	&	p_n^{n - 3}		\\
		E_{\ell - 1, (1, 2)}	&	\cdots	&	E_{\ell - 1, (1, n)}
	\end{vmatrix}.
\end{align*}

Our base case is $n = 4$. When $\ell = 0$ or $1$, the result is immediate, as $E_{0, i} = E_{0, (1, i)} = 1$ for all $i$. When $\ell = 2$, we also see that
\begin{equation*}
	\begin{vmatrix}
		1				&	1				&	1				\\
		p_2				&	p_3				&	p_4				\\
		E_{\ell - 1, (1, 2)}	&	E_{\ell - 1, (1, 3)}	&	E_{\ell - 1, (1, 4)}
	\end{vmatrix}
	=
	\begin{vmatrix}
		1			&	1			&	1				\\
		p_2			&	p_3			&	p_4				\\
		p_3 + p_4		&	p_2 + p_4		&	p_2 + p_3
	\end{vmatrix}
	=
	0.
\end{equation*}
Finally, when $\ell = 3$, we see that
\begin{equation*}
	\begin{vmatrix}
		1				&	1				&	1				\\
		p_2				&	p_3				&	p_4				\\
		E_{\ell - 1, (1, 2)}	&	E_{\ell - 1, (1, 3)}	&	E_{\ell - 1, (1, 4)}
	\end{vmatrix}
	=
	\begin{vmatrix}
		1		&	1		&	1		\\
		p_2		&	p_3		&	p_4		\\
		p_3 p_4	&	p_2 p_4	&	p_2 p_3
	\end{vmatrix}
	=
	-(p_2 - p_3)(p_2 - p_4)(p_3 - p_4),
\end{equation*}
and the result follows.

Now assume the result for $n - 1$ and all $\ell = 0, \ldots, n - 2$. We prove it for $n$ and all $\ell = 0, \ldots, n - 1$. We have
\begin{align*}
	\begin{vmatrix}
		1			&	1			&	\cdots	&	1			\\
		p_1			&	p_2			&	\cdots	&	p_n			\\
		p_1^2		&	p_2^2		&	\cdots	&	p_n^2		\\
		\vdots		&	\vdots		&	\ddots	&	\vdots		\\
		p_1^{n - 2}	&	p_2^{n - 2}	&	\cdots	&	p_n^{n - 2}	\\
		E_{\ell, 1}		&	E_{\ell, 2}		&	\cdots	&	E_{\ell, n}
	\end{vmatrix}
	&=
	(-1)^n A_1(p_1)
	\begin{vmatrix}
		1				&	\cdots	&	1				\\
		p_2				&	\cdots	&	p_n				\\
		\vdots			&	\ddots	&	\vdots			\\
		p_2^{n - 3}		&	\cdots	&	p_n^{n - 3}		\\
		E_{\ell - 1, (1, 2)}	&	\cdots	&	E_{\ell - 1, (1, n)}
	\end{vmatrix} \\
	&=
	\begin{cases}
		0,										&	0 \le \ell \le n - 2, \\
		(-1)^n A_1(p_1) (-1)^{\binom{n - 2}{2}} \Delta_1,	&	\ell = n - 1
	\end{cases} \\
	&=
	\begin{cases}
		0,						&	0 \le \ell \le n - 2, \\
		(-1)^{\binom{n - 1}{2}} \Delta,	&	\ell = n - 1
	\end{cases}
\end{align*}
by the induction hypothesis, and we are done.
\end{proof}

\begin{proof}[Proof of \Cref{claim:det-reduction}]
First, we prove the claim when $0 \le \ell \le n - 2$. We proceed by induction on $\ell$, starting with $\ell = 0$ and $\ell = 1$. The case $\ell = n - 1$ is done later as a separate calculation. For the case $\ell = 0$, note that \eqref{eqn:lin-zero-set-eqns} implies that
\begin{equation*}
	\alpha_k - \alpha_1
	=
	(p_k - p_1) \frac{ \alpha_{n - 1} - \alpha_n }{ p_{n - 1} - p_n },
\end{equation*}
for $k = 1, \ldots, n$, which then implies that
\begin{equation*}
	\begin{vmatrix}
		1			&	1			&	\cdots	&	1			\\
		p_1			&	p_2			&	\cdots	&	p_n			\\
		p_1^2		&	p_2^2		&	\cdots	&	p_n^2		\\
		\vdots		&	\vdots		&	\ddots	&	\vdots		\\
		p_1^{n - 2}	&	p_2^{n - 2}	&	\cdots	&	p_n^{n - 2}	\\
		\alpha_1		&	\alpha_2		&	\cdots	&	\alpha_n
	\end{vmatrix}
	=
	\begin{vmatrix}
		1			&	1			&	\cdots	&	1			\\
		0			&	p_2 - p_1		&	\cdots	&	p_n - p_1			\\
		0			&	p_2(p_2 - p_1)	&	\cdots	&	p_n(p_n - p_1)		\\
		\vdots			&	\vdots		&	\ddots	&	\vdots		\\
		0			&	p_2^{n - 3}(p_2 - p_1)	&	\cdots	&	p_n^{n - 3}(p_n - p_1)	\\
		0			&	\alpha_2 - \alpha_1	&	\cdots	&	\alpha_n - \alpha_1
	\end{vmatrix}
	=
	0
\end{equation*}
since, after expanding along the first column, the last row is a multiple of the first.

Now we treat the inductive step. Assume the result is true for $\ell - 1 \le n - 3$. We prove it for $\ell$. First, note that
\begin{equation*}
	p_1 \alpha_k - p_k \alpha_1
	=
	(p_k - p_1) \frac{ p_n \alpha_{n - 1} - p_{n - 1} \alpha_n }{ p_{n - 1} - p_n }
\end{equation*}
by \eqref{eqn:lin-zero-set-eqns}, so that
\begin{align*}
	E_{\ell, k} \alpha_k - E_{\ell, 1} \alpha_1
		&=
			E_{\ell - 1, (1, k)} (p_1 \alpha_k - p_k \alpha_1) + E_{\ell, (1, 2)}(\alpha_k - \alpha_1) \\
		&=
			(p_k - p_1) \left[ E_{\ell - 1, (1, k)} \frac{ p_n \alpha_{n - 1} - p_{n - 1} \alpha_n }{ p_{n - 1} - p_n } + E_{\ell, (1, k)} \frac{ \alpha_{n - 1} - \alpha_n }{ p_{n - 1} - p_n } \right].
\end{align*}
Hence,
\begin{align*}
	&
	\begin{vmatrix}
		1				&	1				&	\cdots	&	1			\\
		p_1				&	p_2				&	\cdots	&	p_n			\\
		p_1^2			&	p_2^2			&	\cdots	&	p_n^2		\\
		\vdots			&	\vdots			&	\ddots	&	\vdots		\\
		p_1^{n - 2}		&	p_2^{n - 2}		&	\cdots	&	p_n^{n - 2}	\\
		E_{\ell, 1} \alpha_1	&	E_{\ell, 2} \alpha_2	&	\cdots	&	E_{\ell, n} \alpha_n
	\end{vmatrix}
	=
	\begin{vmatrix}
		1				&	1				&	\cdots	&	1			\\
		0				&	p_2 - p_1			&	\cdots	&	p_n - p_1			\\
		0				&	p_2(p_2 - p_1)		&	\cdots	&	p_n(p_n - p_1)		\\
		\vdots			&	\vdots			&	\ddots	&	\vdots		\\
		0				&	p_2^{n - 3}(p_2 - p_1)&	\cdots	&	p_n^{n - 3}(p_n - p_1)	\\
		0				&	E_{\ell, 2} \alpha_2 - E_{\ell, 1} \alpha_1	&	\cdots	&	E_{\ell, n} \alpha_n - E_{\ell, 1} \alpha_1
	\end{vmatrix} \\
	&\quad\quad\quad\quad\quad\quad\, =
	(-1)^{n - 1} A_1(p_1) \frac{ p_n \alpha_{n - 1} - p_{n - 1} \alpha_n }{ p_{n - 1} - p_n }
	\begin{vmatrix}
		1				&	\cdots	&	1				\\
		p_2				&	\cdots	&	p_n				\\
		\vdots			&	\ddots	&	\vdots			\\
		p_2^{n - 3}		&	\cdots	&	p_n^{n - 3}		\\
		E_{\ell - 1, (1, 2)} 	&	\cdots	&	E_{\ell - 1, (1, n)}
	\end{vmatrix} \\
	&\quad\quad\quad\quad\quad\quad\quad\quad\quad\quad\quad+
	(-1)^{n - 1} A_1(p_1) \frac{ \alpha_{n - 1} - \alpha_n }{ p_{n - 1} - p_n }
	\begin{vmatrix}
		1				&	\cdots	&	1				\\
		p_2				&	\cdots	&	p_n				\\
		\vdots			&	\ddots	&	\vdots			\\
		p_2^{n - 3}		&	\cdots	&	p_n^{n - 3}		\\
		E_{\ell, (1, 2)}		&	\cdots	&	E_{\ell, (1, n)}
	\end{vmatrix}
\end{align*}

When $0 \le \ell \le n - 3$, both determinants vanish by \Cref{lemma:det-prep}. When $\ell = n - 2$, \Cref{lemma:det-prep} tells us that the first determinant vanishes and that the second is
\begin{equation*}
	(-1)^{\binom{n - 1}{2} - 1} \frac{ \alpha_{n - 1} - \alpha_n }{ p_{n - 1} - p_n } \Delta.
\end{equation*}
The result follows in this case.

We now prove the claim when $\ell = n - 1$. In this case, we have
\begin{equation*}
	E_{n - 1, k} \alpha_k - E_{n - 1, 1} \alpha_1
	=
	(p_k - p_1) E_{n - 2, (1, k)} \frac{ p_n \alpha_{n - 1} - p_{n - 1} \alpha_n }{ p_{n - 1} - p_n },
\end{equation*}
so, by \Cref{lemma:det-prep},
\begin{align*}
	&
	\begin{vmatrix}
		1				&	1				&	\cdots	&	1			\\
		p_1				&	p_2				&	\cdots	&	p_n			\\
		p_1^2			&	p_2^2			&	\cdots	&	p_n^2		\\
		\vdots			&	\vdots			&	\ddots	&	\vdots		\\
		p_1^{n - 2}		&	p_2^{n - 2}		&	\cdots	&	p_n^{n - 2}	\\
		E_{\ell, 1} \alpha_1	&	E_{\ell, 2} \alpha_2	&	\cdots	&	E_{\ell, n} \alpha_n
	\end{vmatrix}
	=
	\begin{vmatrix}
		1				&	1				&	\cdots	&	1			\\
		0				&	p_2 - p_1			&	\cdots	&	p_n - p_1			\\
		0				&	p_2(p_2 - p_1)		&	\cdots	&	p_n(p_n - p_1)		\\
		\vdots			&	\vdots			&	\ddots	&	\vdots		\\
		0				&	p_2^{n - 3}(p_2 - p_1)&	\cdots	&	p_n^{n - 3}(p_n - p_1)	\\
		0				&	E_{\ell, 2} \alpha_2 - E_{\ell, 1} \alpha_1	&	\cdots	&	E_{\ell, n} \alpha_n - E_{\ell, 1} \alpha_1
	\end{vmatrix} \\
	&\quad\quad\quad\quad\quad\quad=
	(-1)^{n - 1} A_1(p_1) \frac{ p_n \alpha_{n - 1} - p_{n - 1} \alpha_n }{ p_{n - 1} - p_n }
	\begin{vmatrix}
		1			&	\cdots	&	1			\\
		p_2			&	\cdots	&	p_n			\\
		\vdots		&	\ddots	&	\vdots		\\
		p_2^{n - 3}	&	\cdots	&	p_n^{n - 3}	\\
		E_{n - 2, (1, 2)}	&	\cdots	&	E_{n - 2, (1, n)}
	\end{vmatrix} \\
	&\quad\quad\quad\quad\quad\quad=
	(-1)^{n - 1} A_1(p_1) \frac{ p_n \alpha_{n - 1} - p_{n - 1} \alpha_n }{ p_{n - 1} - p_n } (-1)^{\binom{n - 2}{2}} \Delta_1 \\
	&\quad\quad\quad\quad\quad\quad=
	(-1)^{\binom{n - 1}{2} - 1} \frac{ p_n \alpha_{n - 1} - p_{n - 1} \alpha_n }{ p_{n - 1} - p_n } \Delta.
\end{align*}
As noted at the beginning of \Cref{sec:det-calc}, the determinant we computed differs from the desired one by a factor of $(-1)^{\binom{n}{2}}$, so we are done.
\end{proof}

%==============%
% REFERENCES	%
%==============%

\end{document}